\documentclass[sn-mathphys,sort,numbers]{sn-jnl}

\jyear{2021}%

\usepackage{graphicx}
\usepackage{pythonhighlight}
\usepackage{amsmath,amssymb,amsthm}
\usepackage{mathtools}
\usepackage{pythonhighlight}
\usepackage{subfigure}
\usepackage{float}
\usepackage{algorithm}
\usepackage{algorithmicx}
\usepackage{algpseudocode}
\usepackage{enumerate}

\theoremstyle{thmstyleone}%
\newtheorem{theorem}{Theorem}[section]
\newtheorem{proposition}{Proposition}[section]%
\newtheorem{lemma}{Lemma}[section]%
\newtheorem{assumption}{Assumption}[section]%
\newtheorem{corollary}{Corollary}[section]%

\theoremstyle{thmstyletwo}%
\newtheorem{remark}{Remark}%

\theoremstyle{thmstylethree}%
\newtheorem{definition}{Definition}[section]%

\renewcommand{\proofname}{\indent{Proof.}}

\numberwithin{equation}{section}

\begin{document}

\title[Proximal Quasi-Newton Method for Manifold Optimization]{Proximal Quasi-Newton Method for Composite Optimization over the Stiefel Manifold}


\author[1]{\fnm{Qinsi} \sur{Wang}}\email{qinsiwang20@fudan.edu.cn}

\author*[1]{\fnm{Wei Hong} \sur{Yang}}\email{whyang@fudan.edu.cn}
%

\affil*[1]{\orgdiv{School of Mathematical Sciences}, \orgname{Fudan University}, \orgaddress{\street{220 Handan Street}, \city{Shanghai}, \postcode{200433},\country{China}}}
%
%


\abstract{In this paper, we consider the composite optimization problems over the Stiefel manifold. 
A successful method to solve this class of problems is the proximal gradient method proposed by [Chen et al., SIAM J. Optim., 30 (2020), pp. 210–239]. 
Motivated by the proximal Newton-type techniques in the Euclidean space, we present a Riemannian proximal quasi-Newton method, named ManPQN, to solve the composite optimization problems. 
The global convergence of the ManPQN method is proved and iteration complexity for obtaining an $\epsilon$-stationary point is analyzed. 
Under some mild conditions, we also establish the local linear convergence result of the ManPQN method. 
Numerical results are encouraging, which shows that the proximal quasi-Newton technique can be used to accelerate the proximal gradient method.
}





\maketitle

\section{Introduction}\label{sec1}

In this paper, we consider the following composite optimization problem
\begin{eqnarray}
&\mathop{\min}\limits_{X}&F(X):=f(X)+h(X),\label{eq1_prob}\\
\nonumber &s.t. & X^\top X=I_r,
\end{eqnarray}
where $f:\mathbb{R}^{n\times r}\rightarrow\mathbb{R}$ is a smooth function and $h:\mathbb{R}^{n\times r}\rightarrow\mathbb{R}$ is a convex nonsmooth function.
The feasible set ${\rm St}(n,r):=\{X\in\mathbb{R}^{n\times r}:X^\top X=I_r\}$ is referred to as the Stiefel manifold. 

Problem \eqref{eq1_prob} has wide applications in many fields such as machine learning, signal processing and numerical linear algebra. 
For example, when $f(X)=- {\rm tr}(X^TA^TAX)$ and $h(X)=\|X\|_1$, \eqref{eq1_prob} is just the sparse PCA problem;
When $f(X)={\rm tr}(X^TMX)$ and $h(X)=\|X\|_{2,1}$, where $M$ is a given matrix and $\|X\|_{2,1}:=\sum_{i=1}^n\|X(i,:)\|_2$, 
\eqref{eq1_prob} becomes the UFS (unsupervised feature selection) problem,
which finds features that represent the distribution of the input data best, both to reduce the dimension of the data and to eliminate the noisy features. 
For more applications of composite optimization over the Stiefel manifold, we refer the readers to \cite{absil2008, absil2017, mashiqian2020}.

In the Euclidean setting, we can use the subgradient method to solve the composite optimization problems.
In \cite{shor1967}, Shor proved the convergence of subgradient method for nonsmooth convex optimization.
There are several works which extend subgradient methods from Euclidean space to Riemannian manifolds.
Ferreira and Oliveira \cite{subgradient1998} studied the convergence of subgradient method.
Grohs and Hosseini \cite{subgradient2015} proposed an $\varepsilon$-subgradient method and established the global converge result.
In \cite{subgradient_tr2015}, they proposed an algorithm which combines the subgradient method with the trust-region model.

Another interesting approach is to use the operator splitting method to separate the manifold constraint part and the nonsmooth function part.
In the Euclidean space, it is known that the operator splitting method needs more iterations to achieve the same accuracy as the proximal gradient method. 
Chen et al. \cite{mashiqian2020} showed similar results in numerical experiments for the Riemannian setting.

A popular and efficient method for composite optimization in the Euclidean space is the proximal gradient (PG) method \cite{nes2013}.
For a nonsmooth convex function $h$, the proximal mapping of $h$ is defined by
\begin{equation}\label{prox_map}
{\rm prox}_h(x)=\arg\min_{y}\left\{h(y)+\frac{1}{2}\|y-x\|_2^2\right\}.
\end{equation}
To minimize $f(x)+h(x)$, where $f$ is smooth, the PG method takes the step 
\[
x_{k+1}={\rm prox}_{th}(x_k-t\nabla f(x_k)),
\]
where $t>0$. It is well known that
\begin{equation}\label{eq1_prox}
x_{k+1}=\arg\min_{x} \left\{\langle\nabla f(x_k),x\rangle+\frac{1}{2t}\|x-x_k\|_2^2+h(x)\right\}.
\end{equation}
There are several techniques can be used to accelerate the PG method. 
Beck and Teboulle \cite{fista} use Nesterov's acceleration technique to propose an algorithm, named FISTA, 
which shows the improved rate $O(1/k^2)$.

Another acceleration technique is the proximal Newton-type method, which is proposed in \cite{pn1996fukushima,pn2014}. Based on \eqref{prox_map}, Lee et al. \cite{pn2014} define the scaled proximal mapping as
\begin{equation}\label{scale_prox_map}
{\rm prox}_h^B(x)=\arg\min_{y}\left\{h(y)+\frac{1}{2}(y-x)^TB(y-x)\right\}.
\end{equation}
where $B$ is a positive definite matrix.
The motivation of the proximal Newton-type method is to replace the term $\|x-x_k\|_2^2/(2t)$ in \eqref{eq1_prox} by $(x-x_k)^TB_k(x-x_k)/2$,
where $B_k$ is an approximate matrix of $\nabla^2f(x_k)$.
The proximal Newton-type method takes the step
\begin{eqnarray}
\nonumber x_{k+1}&=&{\rm prox}_h^{B_k}(x_k-{B_k}^{-1}\nabla f(x_k))\\
&=&\arg\min_{x} \left\{\langle\nabla f(x_k),x\rangle+\frac{1}{2}(x-x_k)^TB_k(x-x_k)+h(x)\right\}.  
\end{eqnarray}
Superlinear convergence rate of the proximal Newton-type method has been established in \cite{pn2014,pn2022}. In \cite{pn2021}, Nakayama proved the local linear rate of the inexact proximal quasi-Newton method.

Recently, many researchers have extended proximal-type methods to Riemannian manifolds.
See \cite{ppaonmanifold2002,bento2016,boumal2016,yyx2019,ppa2021, Zhang2021ARS,huang2021riemannian,huang2022}. 
In \cite{yyx2019}, Gao et al. 
propose a parallelized proximal linearized augmented Lagrangian algorithm for solving optimization problems over the Stiefel manifold and show its convergence property.
Oviedo \cite{ppa2021} designs an iterative proximal point algorithm to minimize a continuously differentiable function over the Stiefel manifold, which avoids to construct geodesics and can be proved to converge globally. 
Zhang et al. \cite{Zhang2021ARS} extend the smoothing steepest descent method for nonconvex and non-Lipschitz optimization from Euclidean space to Riemannian manifolds. The method can be applied to solve composite optimization problems.
In \cite{mashiqian2020}, Chen et al. propose a retraction-based proximal gradient method, named ManPG, for composite optimization over the Stiefel manifold. 
Numerical experiments in \cite{mashiqian2020} show that the performance of ManPG outperforms subgradient-based methods and ADMM-type methods for sparse PCA and CM problems.
At the $k$-th iteration, the ManPG method solves the following subproblem to get a descent direction at $X_k$:
\begin{equation}\label{eq1_subprob}
V_k=\mathop{\arg\min}\limits_{V\in{\rm T}_{X_k}\mathcal{M}}\left\{\langle{\rm grad}f(X_k),V\rangle+\frac{1}{2t}\|V\|^2+h(X_k+V)\right\}.
\end{equation}
To accelerate the proximal gradient method, we propose a Riemannian proximal quasi-Newton method (ManPQN) to solve \eqref{eq1_prob},
which replaces $\frac{1}{2t}\|V\|^2$ in \eqref{eq1_subprob} by $\frac{1}{2}\|V\|_{\mathcal{B}_k}^2$, where $\mathcal{B}_k$ is a linear operator on ${\rm T}_{X_k}\mathcal{M}$ updated by a quasi-Newton strategy.
The ManPQN method can be regarded as an extension of the proximal quasi-Newton method from Euclidean space to Riemannian manifolds.


Our main contributions can be summarized as follows. 
We propose a proximal quasi-Newton method named ManPQN to solve \eqref{eq1_prob}.
Under some mild conditions, we establish the global convergence and local linear convergence results of ManPQN.
Since the computational cost of updating $\mathcal{B}_k$ by the quasi-Newton strategy is very large, 
we use a strategy for updating $\mathcal{B}_k$, which can reduce the amount of computation significantly.
Numerical experiments demonstrate that the proximal quasi-Newton technique can be used to accelerate the ManPG method.

The organization of the paper is shown as follows. In section \ref{sec2}, 
we introduce some notations and definitions that will be frequently used throughout the paper.
In section \ref{sec3}, we propose the ManPQN algorithm in detail. 
The global convergence of ManPQN is proved and iteration complexity for obtaining an $\epsilon$-stationary point is analyzed in section \ref{sec4}.
Under some mild conditions, we also establish the local linear convergence result of the ManPQN method in this section. 
In section \ref{sec_num_exp}, we compare the ManPQN method with ManPG related methods in the numerical experiments.
The paper ends with some conclusions and a short discussion on possible future works.

\section{Notations and Preliminaries}\label{sec2}

In this section, we introduce some notations and definitions which will
be used in the rest of the paper.
Throughout this section, $\mathcal{M}$ denotes a general manifold.

\begin{definition}(Tangent space \cite[p.33]{absil2008})\label{def1}
Let $X\in\mathcal{M}$.
A tangent vector $\xi_X$ to a manifold $\mathcal{M}$ at $X$ is a mapping from $\mathfrak{F}_X(\mathcal{M})$ to $\mathbb{R}$ such that there exists a smooth curve $\gamma$ on $\mathcal{M}$ with $\gamma(0)=X$, satisfying
\[
\xi_X f=\dot{\gamma}(0)f:=\left.\frac{{\rm d}(f(\gamma(t)))}{{\rm d}t}\right\vert_{t=0},
\]
for all $f\in\mathfrak{F}_X(\mathcal{M})$, where $\mathfrak{F}_X(\mathcal{M})$ denotes the set of smooth real-valued functions defined on a neighborhood of $X$ over the manifold $\mathcal{M}$.
The tangent space to $\mathcal{M}$ at $X$, denoted by ${\rm T}_X\mathcal{M}$, is the set of all tangent vectors to $\mathcal{M}$ at $X$.
The tangent bundle ${\rm T}\mathcal{M}:=\cup_X{\rm T}_X\mathcal{M}$ consists of all tangent vectors to $\mathcal{M}$.
\end{definition}

The Riemannian manifold $\mathcal{M}$ is a manifold whose tangent spaces are endowed with a smoothly varying inner product $\langle\xi,\zeta\rangle_X$,
where $\xi,\zeta\in {\rm T}_X\mathcal{M}$. The norm induced by the Riemannian inner product is denoted by $\|\cdot\|_X$, where $\|\xi\|_X:=(\langle\xi,\xi\rangle_X)^{1/2}$. In this work, for the Stiefel manifold ${\rm St}(n,r)$, we consider the Riemannian metric defined by
$ \langle\xi,\zeta\rangle_X:= tr(\xi^\top \zeta)$. For simplicity of notation, 
we write $\|\cdot\|$ instead of $\|\cdot\|_X$ in the rest of the paper.

The tangent space of the Stiefel manifold ${\rm St}(n,r)$ at $X$ is \cite[p.42]{absil2008}
 \[{\rm T}_X{\rm St}(n,r)=\{V \mid V^\top X+X^\top V=0\}. \]
The projection operator onto ${\rm T}_X{\rm St}(n,r)$ is given by \cite[p.48]{absil2008}
\begin{equation}\label{p}
{\rm Proj}_{{\rm T}_X{\rm St}(n,r)} Z = Z-\frac{1}{2}X(X^\top Z+Z^\top X).
\end{equation}

For a differentiable function $f :\mathcal{M}\rightarrow\mathbb{R}$,
the derivative of $f$ at $X\in \mathcal{M}$, denoted by $\textrm{D}f(X)$,
is an element of the dual space to $T_{X}\mathcal{M}$ which satisfies $\textrm{D}f(X)\xi=\xi f$ for all $\xi\in T_{X}\mathcal{M}$.
The gradient of a smooth function $f$ at $X$, denoted by ${\rm grad} f(X)$ \cite[p.46]{absil2008}, is defined as the unique element of ${\rm T}_X\mathcal{M}$ that satisfies
\begin{equation*}
\langle {\rm grad} f(X),\xi\rangle = {\rm D} f(X)[\xi]\quad\forall\xi\in{\rm T}_X\mathcal{M}.
\end{equation*}
If $\mathcal{M}$ is an embedded submanifold of an Euclidean space $E$, then by \cite[p.48]{absil2008},
\begin{equation}\label{vv}
{\rm grad} f(X) = {\rm Proj}_{{\rm T}_X\mathcal{M}} \nabla f(X),
\end{equation}
where ${\rm Proj}_{{\rm T}_X\mathcal{M}}$ is the projection operator from $E$ onto ${\rm T}_X\mathcal{M}$ and $\nabla f(X)$ is the Euclidean gradient of $f$ at $X$.
For any $\xi\in{\rm T}_X\mathcal{M}$, from \eqref{vv}, it follows that
\begin{equation*}
\langle{\rm grad} f(X),\xi\rangle=\langle\nabla f(X),\xi\rangle.
\end{equation*}

Next we introduce the definitions of the retraction and the vector transport.

\begin{definition}(Retraction \cite[Definition 4.1.1]{absil2008})\label{retraction}
A retraction on a manifold $\mathcal{M}$ is a smooth mapping ${\bf R}$ from the tangent bundle ${\rm T}\mathcal{M}$ onto $\mathcal{M}$ with the following properties. Let ${\bf R}_X$ denote the restriction of ${\bf R}$ to ${\rm T}_X\mathcal{M}$.
\begin{enumerate} 
\item[(1)] ${\bf R}_X(0_X)=X$, where $0_X$ denotes the zero element of ${\rm T}_X\mathcal{M}$.

\item[(2)] With the canonical identification ${\rm T}_{0_X}({\rm T}_X\mathcal{M})\simeq{\rm T}_X\mathcal{M}$, ${\bf R}_X$ satisfies 
\begin{eqnarray}\label{r}
{\rm D}{\bf R}_X(0_X)={\bf id}_{{\rm T}_X\mathcal{M}},
\end{eqnarray}
where ${\rm D}{\bf R}_X(0_X)$ denotes the differential of the retraction ${\bf R}_X$ at the zero element $0_X\in{\rm T}_X\mathcal{M}$ and ${\bf id}_{{\rm T}_X\mathcal{M}}$ denotes the identity mapping on ${\rm T}_X\mathcal{M}.$
\end{enumerate}
\end{definition}

The following properties for a retraction ${\bf R}$ will be used in section \ref{sec4}.
For a proof, see \cite{2order_boudness2016}.

\begin{proposition}\label{prop2_1} (\cite{2order_boudness2016})
Suppose $\mathcal{M}$ is a compact embedded submanifold of an Euclidean space $E$, 
and ${\bf R}$ is a retraction.
Then there exists $M_1, M_2>0$ such that for all $X\in\mathcal{M}$ and for all $\xi\in{\rm T}_X\mathcal{M}$,
\begin{eqnarray}
\|{\bf R}_X(\xi)-X\| &\leq& M_1\|\xi\| ,\label{eq5}\\
\|{\bf R}_X(\xi)-X-\xi\| &\leq& M_2\|\xi\|^2.\label{eq6}
\end{eqnarray}
\end{proposition}

Next we will consider the transport of a vector from one tangent space ${\rm T}_X\mathcal{M}$ into another one ${\rm T}_Y\mathcal{M}$.

\begin{definition} (Vector Transport \cite[Definition 8.1.1]{absil2008})
A vector transport associated with a retraction ${\bf R}$ is defined as a continuous function $\mathcal{T}:{\rm T}\mathcal{M}\times {\rm T}\mathcal{M}\rightarrow {\rm T}\mathcal{M}, (\eta_X,\xi_X)\mapsto{\rm T}_{\eta_X}(\xi_X)$, which satisfies the following conditions:

(i)~$\mathcal{T}_{\eta_X}:{\rm T}_X\mathcal{M}\rightarrow {\rm T}_{{\bf R}_X(\eta_X)}\mathcal{M}$ is a linear invertible map,

(ii)~$\mathcal{T}_{0_X}(\xi_X)=\xi_X$.
\end{definition}

Let $Y:={\bf R}_X(\eta_X)$, where $\eta_X\in{\rm T}_{X}\mathcal{M}$. For ease of notation, we denote
\begin{eqnarray*}\label{iso}
\mathcal{T}_{X,Y}(\xi_X):=\mathcal{T}_{\eta_X}(\xi_X),
\end{eqnarray*}
where $\xi_X\in {\rm T}_{X}\mathcal{M}$.

For any $h\in{\rm T}_{X}\mathcal{M}$, we use $h^\flat$ to denote the linear function on $T_{X}\mathcal{M}$ induced by
\[
h^{\flat}\eta:=\langle h,\eta\rangle={\rm tr}(h^T\eta)\quad\forall\eta\in{\rm T}_{X}\mathcal{M}.
\]
Let $H=[h_1,\dots,h_m]$, where $h_i\in T_{X}\mathcal{M}$ for any $i=1,\dots,m$.
Define $H^{\flat}:{\rm T}_{X}\mathcal{M}\to\mathbb{R}^m$ by
\[
H^{\flat}\eta=[h_1^{\flat}\eta,\dots,h_m^{\flat}\eta]^T\quad\forall\eta\in{\rm T}_{X}\mathcal{M}.
\]

Finally, we introduce the definitions of generalized Calrke subdifferential and regular function.

\begin{definition}[Generalized Calrke subdifferential \cite{mashiqian2020,ywh2013}]
For a locally Lipschitz function $\psi$ on a Riemannian manifold $\mathcal{M}$,
the Riemannian generalized directional derivative of $\psi$ at $X\in\mathcal{M}$ in the direction $V\in{\rm T}_X\mathcal{M}$ is defined by
\begin{equation*}
\psi^{\circ}(X;V)=\mathop{\lim\sup}\limits_{Y\to X,t\downarrow0}\frac{\psi\circ\phi^{-1}(\phi(Y)+tD\phi(X)[V])-\psi\circ\phi^{-1}(\phi(Y))}{t},
\end{equation*}
where $(\phi,U)$ is a coordinate chart at $X$. The generalized Clarke subdifferential of $\psi$ at $X\in\mathcal{M}$,
denoted by $\hat{\partial}\psi(X)$, is a subset of ${\rm T}_X\mathcal{M}$ whose support function is $\psi^{\circ}(X;\cdot)$, defined by
\begin{equation*}
\hat{\partial}\psi(X)=\{\xi\in{\rm T}_X\mathcal{M}: \langle\xi,V\rangle\leq \psi^{\circ}(X;V), \forall~V\in{\rm T}_X\mathcal{M} \}.
\end{equation*}
\end{definition}

In the rest of this section, we assume that $\mathcal{M}$ be an embedded submanifold of an Euclidean space $E$.

\begin{definition}[Regular function \cite{ywh2013}]
Let $f$ be a function defined on $E$.
We say that $f$ is regular at $X\in\mathcal{M}$ along ${\rm T}_X\mathcal{M}$ if
\begin{enumerate}
\item[(1)] for all $V\in{\rm T}_X\mathcal{M}$, 
\[
f^{\prime}(X;V)=\lim_{t\downarrow 0} \frac{f(X+tV)-f(X)}{t},
\]
exists, and
\item[(2)] for all $V\in{\rm T}_X\mathcal{M}$,
$f^{\prime}(X;V)=f^{\circ}(X;V).$
\end{enumerate}
\end{definition}

For a regular nonsmooth function $\psi(X)$, 
it is proved in \cite[Theorem 5.1]{ywh2013} that $\hat{\partial}\psi(X)={\rm Proj}_{{\rm T}_X\mathcal{M}} (\partial \psi(X))$,
where $\partial \psi(X)$ is the Euclidean generalized Clarke subdifferential of $\psi$ at $X$.
Since $f$ is continuously differentiable and $h$ is convex,
by \cite[Lemma 5.1]{ywh2013}, the objective function $F=f+h$ is regular.
Then the first-order necessary condition of problem \eqref{eq1_prob} can be written as
\begin{equation}\label{eq2_opt}
0\in \hat{\partial}F(X^*)={\rm grad}f(X^*)+{\rm Proj}_{{\rm T}_{X^*}\mathcal{M}} (\partial h(X^*)),
\end{equation}
where $X^*$ is a local optimal solution of \eqref{eq1_prob}.

\section{Proximal Quasi-Newton Method}\label{sec3}

In this section, we propose a proximal quasi-Newton algorithm over the Stiefel manifold, named ManPQN,
which takes a proximal quasi-Newton step at each iteration, and uses the
nonmonotone line search strategy to determine the stepsize.
In the rest of the paper, we use $\mathcal{M}$ to denote the Stiefel manifold.

\subsection{The ManPQN Algorithm}\label{sec31}

At the current iterate $X_k$, the ManPQN method solves the following problem to get the descent direction $V_k$:
\begin{eqnarray}
V_k=&\mathop{\arg\min}
& \langle\nabla f(X_k),V\rangle+\frac{1}{2}\|V\|_{\mathcal{B}_k}^2+h(X_k+V)\label{eq3_subprob}\\
& {\rm s.t.}& \mathcal{A}_k(V):= V^TX_k+X_k^TV=0,\label{eq3_constraint}
\end{eqnarray}
where $\mathcal{B}_k$ is a symmetric positive definite operator on ${\rm T}_{X_k}\mathcal{M}$,
and $\|V\|_{\mathcal{B}_k}^2=\langle V, \mathcal{B}_k[V]\rangle$.
From \eqref{eq3_constraint}, we know that $V_k\in{\rm T}_{X_k}\mathcal{M}$.

To guarantee the positive definiteness of $\mathcal{B}_k$, we use a damped LBFGS strategy to generate $\mathcal{B}_k$.
We leave the details of the updating in the next subsection.

After $V_k$ is obtained, we apply the nonmonotone line search technique to determine the stepsize $\alpha_k$.
The nonmonotone technique was first introduced in \cite{gll1986}.
Specifically, $\alpha_k$ is set to be $\gamma^{N_k}$, where $N_k$ is the smallest integer such that
\begin{eqnarray}\label{nonmonotone}
F({\bf R}_{X_k}(\alpha_k V_k))\leq\mathop{\max}\limits_{\max\{0,k-m\}\leq j\leq k}F(X_j)-\frac{1}{2}\sigma\alpha_k\|V_k\|_{\mathcal{B}_k}^2,
\end{eqnarray}
in which $m>0$ is an integer, $\sigma,\gamma\in(0,1)$.
For simplicity, we use the notation
\begin{equation}\label{lk}
l(k):=\mathop{\arg\max}\limits_{\max\{k-m,0\}\leq j\leq k}F(X_{j}).
\end{equation}
Then $F(X_{l(k)})=\max\limits_{\max\{k-m,0\}\leq j\leq k}F(X_{j})$.

Now we summarize the ManPQN method as follows:

\begin{algorithm}
\caption{Proximal quasi-Newton algorithm with nonmonotone line search to solve problem \eqref{eq1_prob}}\label{algo1}
\begin{algorithmic}[1]
\Require Initial point $X_0\in\mathcal{M},$ $\gamma, \sigma\in(0,1),$ $m>0$ are integers.
\For{$k = 0,1,\dots$}
		\If{$k\geq1$}
		\State Update $\mathcal{B}_k$ by quasi-Newton strategy;
		\Else
		\State Set $\mathcal{B}_k=I$;
		\EndIf
		\State Solve the subproblem \eqref{eq3_subprob} to get the search direction $V_k$; 
		\State Set initial stepsize $\alpha_k=1$;
\While{\eqref{nonmonotone} is not satisfied}
\State $\alpha_k=\gamma\alpha_k$;
\EndWhile
\State Set $X_{k+1}={\bf R}_{X_k}(\alpha_k V_k)$;
\EndFor
\end{algorithmic}
\end{algorithm}

\subsection{Damped LBFGS Method}\label{sec32}

If the operator $\mathcal{B}_k$ in \eqref{eq3_subprob} is updated by the Riemannian BFGS method, the total amount of computation is very large.
It is well known that the Limited memory BFGS method (LBFGS) \cite{lbfgs1980} is more suitable for large scale problems.
Comparing with the BFGS method, the LBFGS method only needs the information of the last $p$ steps, where $p$ is the memory size.

To give the Riemannian LBFGS update formula, we introduce some notations first.
For $k\geq0$, let $g_k={\rm grad}f(X_k)\in{\rm T}_{X_k}\mathcal{M}$ and $\mathcal{T}_{k,k+1}:=\mathcal{T}_{X_{k},X_{k+1}}$.
Define
\begin{eqnarray*}
S_{k}&:=& \mathcal{T}_{k,k+1}({\bf R}_{X_{k}}^{-1}(X_{k+1})), \ Y_{k} := g_{k+1}-\mathcal{T}_{k,k+1}(g_{k}).
\end{eqnarray*}

Given an initial estimate $\mathcal{B}_{k,0}$ at iteration $k$ and two sequences $\{S_j\}$ and $\{Y_j\}$,
$j=k-p,\dots,k-1$, the Riemannian LBFGS method updates $\mathcal{B}_{k,i}$ recursively as
\begin{eqnarray}
\label{eq3_bk}
\mathcal{B}_{k,i}&=&\widetilde{\mathcal{B}}_{k,i-1} - \frac{\widetilde{\mathcal{B}}_{k,i-1}S_j(\widetilde{\mathcal{B}}_{k,i-1}S_j)^{\flat}}{S_j^{\flat}\widetilde{\mathcal{B}}_{k,i-1}S_j} + \rho_jY_jY_j^{\flat},\\
\nonumber j &=& k-(p-i+1),\ i=1,\dots,p,
\end{eqnarray}
where
\[
\widetilde{\mathcal{B}}_{k,i-1}=\mathcal{T}_{j,j+1}\circ\mathcal{B}_{k,i-1}\circ\mathcal{T}_{j,j+1}^{-1} \ {\rm and} \ \rho_j=\frac{1}{Y_j^{\flat}S_j}.
\]
The output $\mathcal{B}_{k,p}$ is set to be $\mathcal{B}_{k}$, the Riemannian LBFGS operator.
It is easy to see $\mathcal{B}_{k}$ is a linear operator on ${\rm T}_{X_{k}}\mathcal{M}$.

The inverse of $\mathcal{B}_{k,i}$ is denoted by $\mathcal{H}_{k,i}$, $i=0,1,\cdots,p$.
By the Sherman-Morrison-Woodbury
formula \cite{smw1950}, we have
\begin{eqnarray}
\nonumber\mathcal{H}_{k,i} &=& ({\bf id}-\rho_jS_jY_j^{\flat})\mathcal{T}_{j,j+1}\mathcal{H}_{k,i-1}\mathcal{T}_{j,j+1}^{-1}({\bf id}-\rho_jY_jS_j^{\flat})+\rho_jS_jS_j^{\flat},\\
\nonumber j &=& k-(p-i+1),\ i=1,\dots,p,
\end{eqnarray}
where ${\bf id}$ denotes the identity map.
It is obvious that $\mathcal{H}_{k}=\mathcal{H}_{k,p}$, where $\mathcal{H}_{k}$ is the inverse of $\mathcal{B}_{k}$.

From \eqref{eq3_bk}, we know that the update of $\mathcal{B}_k$ is computationally expensive because it involves the calculation of vector transports.
Since the Stiefel manifold $\mathcal{M}$ is a submanifold of a $\mathbb{R}^{n\times r}$,
we can calculate the Euclidean differences $s_k$ and $y_k$ to replace $S_k$ and $Y_k$, that is
\begin{eqnarray*}
s_k=X_{k+1}-X_k\in\mathbb{R}^{n\times r},~~\quad y_k=g_{k+1}-g_k\in\mathbb{R}^{n\times r}.
\end{eqnarray*}
Calculating $s_k$ and $y_k$ is much cheaper than calculating $S_k$ and $Y_k$. Such a strategy was used in \cite{wzw2010,wzw2018}.

To reduce the computational cost further,
we will use a simple and easily computed $\mathbf{B}_k$ to approximate $\mathcal{B}_k$.
Such a strategy was used in \cite{diag2020,diagonal2022} for the Euclidean setting.
In the ManPQN method, we solve the following subproblem to get $V_k$:
\begin{equation}\label{duoduo}
V_k=\mathop{\arg\min}\limits_{V\in{\rm T}_{X_k}\mathcal{M}}\left\{\langle\nabla f(X_k),V\rangle+\frac{1}{2}{\rm tr}(V^T\mathbf{B}_k[V])+h(X_k+V)\right\},
\end{equation}
where
\begin{eqnarray}\label{hanwu}
\mathbf{B}_{k}[V] ={\rm Proj}_{{\rm T}_{X_k}\mathcal{M}}\big(({\rm diag}B_k)({\rm Proj}_{{\rm T}_{X_k}\mathcal{M}}V)\big),
\end{eqnarray}
in which $B_k\in\mathbb{R}^{n\times n}$ is a symmetric matrix.
By \eqref{hanwu}, we can deduce that
\begin{eqnarray}\label{shasha}
{\rm tr}(V^T\mathbf{B}_k[V])={\rm tr}(V^T({\rm diag}B_k)V)\quad\forall V\in{\rm T}_{X_k}\mathcal{M}.
\end{eqnarray}
To ensure $V_k$ is a descent direction of $F$ at $X_k$,
a sufficient condition is that $B_k$ is a positive definite matrix.
To guarantee the positive definitness of $B_k$, we use the damped technique employed in \cite{damp2017,damp1978}.
Specifically, for $k\geq 0$, define
\begin{eqnarray}\label{jia}
\overline{y}_{k-1} = \beta_{k-1}{y}_{k-1} + (1-\beta_{k-1})H_{k,0}^{-1}{s}_{k-1},
\end{eqnarray}
where we set $H_{k,0}=(1/\delta)I$ for some $\delta>0$ and
\begin{eqnarray*}
\beta_{k-1}=
\left\{
\begin{array}{ll}
\frac{0.75{\rm tr}(s_{k-1}^TH_{k,0}^{-1}s_{k-1})}{{\rm tr}(s_{k-1}^TH_{k,0}^{-1}s_{k-1}) - {\rm tr}(s_{k-1}^Ty_{k-1})},
&\textrm{if $ {\rm tr}(s_{k-1}^Ty_{k-1}) <  0.25 {\rm tr}(s_{k-1}^TH_{k,0}^{-1}s_{k-1})$};\\
1,&\textrm{otherwise}.
\end{array}
\right.
\end{eqnarray*}
Then we can define $B_k$ as follows:
\begin{eqnarray}\label{eq3_eu_bk}
\left\{\begin{array}{l}
B_{k}=B_{k,p},\\
B_{k,i}=B_{k,i-1}-\frac{B_{k,i-1}s_{j}s_{j}^TB_{k,i-1}}{{\rm tr}(s_{j}^TB_{k,i-1}s_{j})}
+\frac{\overline{y}_{j}\overline{y}_{j}^T}{{\rm tr}(s_{j}^T\overline{y}_{j})},\\
j=k-(p-i+1),i=1,\dots,p,\\
B_{k,0}=\delta I_n.
\end{array}
\right.
\end{eqnarray}
Similar to the proof of \cite[Lemma 3.1]{damp2017},
we can show that $B_k$ are positive definite matrices for all $k$. 
We omit the detail.
The inverse of $B_{k,i}$ and $B_k$ are denoted by $H_{k,i}$ and $H_k$.
Thus, it holds that
\begin{eqnarray} \label{eq3_lbfgs_hk}
\left\{
\begin{array}{l}
H_k=H_{k,p},\\
H_{k,i} = (I - \overline{\rho}_js_j\overline{y}_j^{T})H_{k,i-1} (I - \overline{\rho}_j\overline{y}_js_j^{T})+\overline{\rho}_j s_js_j^{T},\\
j = k-(p-i+1),\ i=1,\dots,p,\\
H_{k,0}=(1/\delta)I_n,
\end{array}\right.
\end{eqnarray}
where $\overline{\rho}_j=1/{\rm tr}(s_{j}^T\overline{y}_{j})$.

The following lemma shows that $\mathbf{B}_{k}$ and its inverse are all bounded operators.
In the proof of the lemma, we will use the notation
\begin{eqnarray}\label{wang}
\varrho:=\sup_{X\in\mathcal{M}}\|\nabla f(X)\|.
\end{eqnarray}
Since $\mathcal{M}$ is a compact set in $\mathbb{R}^{n\times r}$,
it follows that $\varrho<\infty$.
We will also use the notations
\begin{eqnarray}\label{kappa}
\kappa_1=\left(\frac{\upsilon^{2p}-1}{\upsilon^2-1}\cdot \frac{4}{\delta}+ \frac{\upsilon^{2p}}{\delta}\right)^{-1},\qquad
\kappa_2=\delta+\frac{4(L+\varrho+\delta)^2}{\delta}p,
\end{eqnarray}
where $\upsilon:=(4L+4\varrho+5\delta)/\delta$, and $\delta$ is as in \eqref{eq3_eu_bk}.

\begin{lemma}\label{lm_bk}
Suppose that $\nabla f$ is Lipschitz continuous with the Lipschitz constant $L$.
Suppose $\mathbf{B}_{k}$ and $B_k$ are defined by \eqref{hanwu} and \eqref{eq3_eu_bk} respectively.
Denote $\|V\|_{\mathbf{B}_{k}}^2:={\rm tr}(V^T\mathbf{B}_k[V])$.
Then for all $k$, we have
\begin{eqnarray}\label{kan}
\kappa_1\|V\|^2
\leq\|V\|_{\mathbf{B}_{k}}^2\leq \kappa_2\|V\|^2\quad\forall V\in {\rm T}_{X_k}\mathcal{M}.
\end{eqnarray}
\end{lemma}

\begin{proof}
To prove \eqref{kan}, by \eqref{shasha}, we only need to show $\lambda_{\max}(B_k)\leq\kappa_2$ and $\lambda_{\max}(H_k)\leq\kappa_1^{-1}$.
By \eqref{jia}, We have
\begin{eqnarray}\label{de}
{\rm tr}(s_{k-1}^T\overline{y}_{k-1})=\max\{0.25\delta\|s_{k-1}\|^2,{\rm tr}(s_{k-1}^Ty_{k-1})\}\geq0.25\delta\|s_{k-1}\|^2.
\end{eqnarray}
For simplicity, we use the notation $P_k={\rm Proj}_{{\rm T}_{X_k}\mathcal{M}}$.
From \eqref{p}, it follows that $\|P_k-P_{k-1}\|\leq\|X_k-X_{k-1}\|$,
Thus, by \eqref{vv}, we have
\begin{eqnarray*}
\|y_{k-1}\|&=&\|P_k\nabla f(X_k)-P_{k-1}\nabla f(X_{k-1})\| \\
&\leq&\|P_k\nabla f(X_k)-P_k\nabla f(X_{k-1})\|+\|P_k\nabla f(X_{k-1})-P_{k-1}\nabla f(X_{k-1})\| \\
&\leq& (L+\varrho)\|s_{k-1}\|,
\end{eqnarray*}
which together with \eqref{jia} implies
\begin{eqnarray}\label{cheng}
\|\overline{y}_{k-1}\|\leq(L+\varrho+\delta)\|s_{k-1}\|\quad\forall k.
\end{eqnarray}
By \eqref{eq3_eu_bk}, \eqref{de} and \eqref{cheng}, it holds that
\begin{eqnarray*}
\|B_{k,i}\|_2&\leq&\left\|B_{k,i-1}-\frac{B_{k,i-1}s_{j}s_{j}^TB_{k,i-1}}{{\rm tr}(s_{j}^TB_{k,i-1}s_{j})}\right\|_2
+\left\|\frac{\overline{y}_{j}\overline{y}_{j}^T}{{\rm tr}(s_{j}^T\overline{y}_{j})}\right\|_2\\
&\leq&\|B_{k,i-1}\|_2+\frac{(L+\varrho+\delta)^2\|s_{j}\|^2}{0.25\delta\|s_{j}\|^2},\quad j=k-(p-i+1),i=1,\dots,p,
\end{eqnarray*}
which together with $B_{k,0}=\delta I$ implies $\lambda_{\max}(B_k)\leq\kappa_2$.

Recall that $\overline{\rho}_j=1/{\rm tr}(s_{j}^T\overline{y}_{j})$.
By \eqref{cheng} again, we have $\overline{\rho}_j\left\|s_j\overline{y}_j^T\right\|\leq4(L+\varrho+\delta)/\delta$.
Thus, by \eqref{eq3_lbfgs_hk}, \eqref{de} and \eqref{cheng}, we have
\begin{eqnarray*}
\|H_{k,i}\|_2&\leq&\|I - \overline{\rho}_js_j\overline{y}_j^{T}\|_2\cdot
\|H_{k,i-1}\|_2\cdot\| I - \overline{\rho}_j\overline{y}_js_j^{T}\|_2+\left\|\overline{\rho}_j s_js_j^{T}\right\|_2\\
&\leq&(1 + \overline{\rho}_j\|s_j\overline{y}_j^{T}\|_2)^2
\|H_{k,i-1}\|_2 +4/\delta\\
&\leq&(1+4(L+\varrho+\delta)/\delta)^2\|H_{k,i-1}\|_2+4/\delta,\quad j=k-(p-i+1),i=1,\dots,p.
\end{eqnarray*}
By the above inequality, it is easy to prove $\lambda_{\max}(H_k)\leq\kappa_1^{-1}$ by induction. We omit the detail.

\end{proof}

Lemma \ref{lm_bk} shows that $\{\mathbf{B}_k\}$ are uniformly positive definite if $\mathbf{B}_k$ is generated by \eqref{hanwu}.
In the following, we show how to solve the subproblem \eqref{eq3_subprob}.
Let $\Lambda$ be the Lagrange multiplier for the constraint \eqref{eq3_constraint}.
Then $\Lambda\in\mathbb{S}^r$, where $\mathbb{S}^r$ denotes the set of symmetric $r\times r$ matrices.
The Lagrangian function for \eqref{eq3_subprob} is
\begin{eqnarray}
\mathcal{L}_k(V,\Lambda)&=& \langle\nabla f(X_k),V\rangle+\frac{1}{2}\|V\|_{\mathbf{B}_k}^2
+h(X_k+V)- \langle\mathcal{A}_k^* (\Lambda),V\rangle\nonumber\\
&=&\langle\nabla f(X_k)-\mathcal{A}_k^* (\Lambda),V\rangle+\frac{1}{2}{\rm tr}(V^T({\rm diag}B_k)V)
+h(X_k+V),\label{eq3_lag}
\end{eqnarray}
where $\mathcal{A}_k^*$ denotes the adjoint operator of $\mathcal{A}_k$.
For a fixed $\Lambda$, $\mathcal{L}_k(V,\Lambda)$ is a strongly convex function of $V$ since $B_k$ is a positive definite matrix.
We use $V(\Lambda)$ to denote the unique minimum of $\mathop{\min}\limits_V \mathcal{L}_k(V,\Lambda)$.
By \eqref{scale_prox_map} and \eqref{eq3_lag}, we have
\begin{eqnarray}
V(\Lambda) &=& {\rm prox}_h^{\mathbf{B}_k}\big(X_k-({\rm diag}B_k)^{-1}(\nabla f(X_k)-\mathcal{A}_k^*(\Lambda))\big)-X_k  \label{eq3_vk}\\
&=& {\rm prox}_h^{\mathbf{B}_k}(B(\Lambda))-X_k ,\nonumber
\end{eqnarray}
where $B(\Lambda):=X_k-({\rm diag}B_k)^{-1}(\nabla f(X_k)-\mathcal{A}_k^*(\Lambda))$ and $H_k$ is the inverse of $B_k$.
Substituting \eqref{eq3_vk} into \eqref{eq3_constraint} yields
\begin{equation}\label{eq3_e}
E(\Lambda)\equiv\mathcal{A}_k(V(\Lambda)) = V(\Lambda)^TX_k+X_k^TV(\Lambda) = 0.
\end{equation}

Similar to the discussion at \cite[p.221]{mashiqian2020},
we can prove that $E(\Lambda)$ is a monotone and Lipschitz continuous operator on $\mathbb{S}^r$.
Then the adaptive regularized semismooth Newton (ASSN) method can be used to solve \eqref{eq3_e}.
We give a brief description in the following. 

The vectorization of \eqref{eq3_e} can be written as
\begin{eqnarray}
\nonumber  {\rm vec}(E(\Lambda)) 
&=&
(X_k^T  \otimes  I_r)K_{nr}\mathrm{vec}(V (\Lambda )) + (I_r \otimes  X_k^T )\mathrm{vec}(V (\Lambda ))\\
\nonumber
&=& (K_{rr}+I_{r^2}) (I_r\otimes X_k^T) \big[{\rm prox}_{h}^{\mathbf{B}_k}({\rm vec}(X_k-({\rm diag}B_k)^{-1}\nabla f(X_k))\\
\nonumber & &  \quad\quad\quad\quad\quad\quad\quad\quad\quad\quad+ 2(I_r\otimes (({\rm diag}B_k)^{-1}X_k)) {\rm vec}(\Lambda)) - {\rm vec}(X_k)\big],
\end{eqnarray}
where $K_{nr}$ and $K_{rr}$ are comutation matrices.
Define
\begin{eqnarray}
\label{eq3_jac} \mathcal{G}({\rm vec}(\Lambda)) &:=& 2 (K_{rr}+I_{r^2}) (I_r\otimes X_k^T)
\mathcal{J}(y)\vert_{y={\rm vec}(B(\Lambda))}
(I_r\otimes (({\rm diag}B_k)^{-1}X_k)),\quad\quad
\end{eqnarray}
where $\mathcal{J}(y)$ is the generalized Jacobian of ${\rm prox}_h^{\mathbf{B}_k}(y)$. Then, we know that $\mathcal{G}({\rm vec}(\Lambda)) \in\partial {\rm vec}(E({\rm vec}(\Lambda)))$.

Denote $\overline{\rm vec}(\Lambda)$ as the vectorization of the lower triangular part of $\Lambda$.
Then there exists a duplication matrix $U_r\in\mathbb{R}^{r^2\times\frac{1}{2}r(r+1)}$ such that $\overline{\rm vec}(\Lambda)=U_r^+{\rm vec}(\Lambda)$, 
where $U_r^+=(U_r^TU_r)^{-1}U_r$ is the Moore-Penrose inverse of $U_r$. By \eqref{eq3_jac}, the generalized Jacobian of $\overline{\rm vec}(E(U_r\overline{\rm vec}(\Lambda)))$ can be written as
\begin{eqnarray}
\nonumber  \mathcal{G}(\overline{\rm vec}(\Lambda)) &=& U_r^+\mathcal{G}({\rm vec}(\Lambda))U_r
\\
\nonumber
&=& 4 U_r^+ (I_r\otimes X_k^T)
\mathcal{J}(y)\vert_{y={\rm vec}(B(\Lambda))}
(I_r\otimes (({\rm diag}B_k)^{-1}X_k))U_r.
\end{eqnarray}
At the current iterate $\Lambda_l$, to get the Newton direction $d_l$, we can apply the conjugate gradient method to solve the following equation
\begin{equation}\label{biao}
(\mathcal{G}(\overline{\rm vec}(\Lambda_l))+\eta I)d = - \overline{\rm vec}(E(\Lambda_l)),
\end{equation}
where $\eta>0$ is a regularization parameter. 
Then, we use the same strategy as that in \cite{assn2018} to obtain the next iterate $\Lambda_{l+1}$.
For more details, we refer the reader to \cite{mashiqian2020,assn2018}.

\section{Convergence Analysis of ManPQN Algorithm}\label{sec4}

In this section, we study the convergence properties of the ManPQN algorithm.
Under mild assumptions, we prove the global convergence of Algorithm \ref{algo1}.
We also analyze the local convergence rate of the ManPQN method.
It is proved that the iterates of the algorithms converge locally linearly to the nondegenerate local minimum point.

\subsection{Global Convergence}\label{sec41}

First, we give the following assumption which is required in the rest of the paper.

\begin{assumption}\label{as_func}
$f:\mathbb{R}^{n\times r}\rightarrow\mathbb{R}$ is smooth, and $\nabla f$ is Lipschitz continuous with Lipschitz constant $L$;
$h:\mathbb{R}^{n\times r}\rightarrow\mathbb{R}$ is a convex but nonsmooth function, and $h$ is Lipschitz continuous with Lipschitz constant $L_h$.
\end{assumption}

Given $X_k$, we denote the objective function of \eqref{eq3_subprob} by $\phi_k$, that is
\begin{equation}\label{eq3_subprob_func}
\phi_k(V) := \langle\nabla f(X_k),V\rangle+\frac{1}{2}\|V\|_{\mathcal{B}_k}^2+h(X_k+V),
\end{equation}
where $\mathcal{B}_{k}$ is a linear operator on $\mathbb{R}^{n\times r}$ satisfying $\mathcal{B}_k{\rm T}_{X_k}\mathcal{M} \subseteq {\rm T}_{X_k}\mathcal{M}$.
In our convergence analysis, the only requirement of $\mathcal{B}_k$ is that it satisfies \eqref{kan},
that is, $\kappa_1\|V\|^2
\leq\|V\|_{\mathcal{B}_{k}}^2\leq \kappa_2\|V\|^2$ for all $V\in {\rm T}_{X_k}\mathcal{M}$.
This and Assumption \ref{as_func} imply that $\phi_k$ is a convex function on ${\rm T}_{X_k}\mathcal{M}$.

Since $V_k=\mathop{\arg\min}\limits_{V\in{\rm T}_{X_k}\mathcal{M}}\phi_k(V)$, by \eqref{vv}, we have
\begin{equation}\label{eq4_sub_opt}
0\in{\rm Proj}_{{\rm T}_{X_k}\mathcal{M}}\partial\phi_k(V_k)={\rm grad}f(X_k) +\mathcal{B}_k[V_k]+{\rm Proj}_{{\rm T}_{X_k}\mathcal{M}}\partial h(X_k+V_k).
\end{equation}
If $V_k=0$, then $X_k$ satisfies \eqref{eq2_opt}, and therefore is a stationary point of \eqref{eq1_prob};
If $V_k\neq0$, the following result shows that $V_k$ is a descent direction of $\phi_k$.
The proof is similar to that of \cite[Lemma 5.1]{mashiqian2020}.
We give a proof for completeness.

\begin{lemma}\label{lm41_phi_dec}
Suppose Assumption \ref{as_func} holds. For any $\alpha\in [0,1]$, it holds that
\begin{eqnarray}\label{lai}
\phi_k(\alpha V_k)-\phi_k(0)\leq\frac{\alpha(\alpha-2)}{2}\|V_k\|_{\mathcal{B}_k}^2.
\end{eqnarray}
\end{lemma}

\begin{proof}
By \eqref{eq4_sub_opt}, there exists $\xi\in\partial h(X_k+V_k)$ such that ${\rm grad}f(X_k) +\mathcal{B}_k[V_k]+{\rm Proj}_{{\rm T}_{X_k}\mathcal{M}}\xi=0$.
From $\xi+\nabla f(X_k)\in\partial\left(\phi_k-\frac{1}{2}\|\cdot\|_{\mathcal{B}_k}^2\right)(V_k)$,
it follows that
\begin{eqnarray}
\phi_k(0) -\phi_k(V_k)
&\geq& \langle\nabla f(X_k)+\xi,-V_k\rangle-\frac{1}{2}\|V_k\|_{\mathcal{B}_k}^2 \nonumber\\
&=& \langle{\rm grad}f(X_k)+{\rm Proj}_{{\rm T}_{X_k}\mathcal{M}}\xi+\mathcal{B}_k[V_k],-V_k\rangle+\frac{1}{2}\|V_k\|_{\mathcal{B}_k}^2 \nonumber\\
&=& \frac{1}{2}\|V_k\|_{\mathcal{B}_k}^2.    \label{eq33_1}
\end{eqnarray}
Since $h$ is a convex function, for all $0\leq\alpha\leq1$, we have
\begin{equation}\label{eq33_3}
h(X_k+\alpha V_k)-h(X_k)\leq \alpha(h(X_k+V_k)-h(X_k)).
\end{equation}
Combining \eqref{eq33_1} and \eqref{eq33_3} yields
\begin{eqnarray}
\nonumber \phi_k(\alpha V_k) -\phi_k(0) &=& \langle\nabla f(X_k),\alpha V_k\rangle
+ \frac{1}{2}\|\alpha V_k\|_{\mathcal{B}_k}^2 + h(X_k+\alpha V_k)-h(X_k)\\
\nonumber &\leq&  \alpha\left(
\langle\nabla f(X_k),V_k\rangle
+\frac{\alpha}{2}\|V_k\|_{\mathcal{B}_k}^2 +h(X_k+V_k)-h(X_k)\right)\\
\nonumber &=&  \alpha(\phi_k(V_k)-\phi_k(0)+\frac{\alpha-1}{2}\|V_k\|_{\mathcal{B}_k}^2)\\
\nonumber &\leq&  \frac{\alpha(\alpha-2)}{2}\|V_k\|_{\mathcal{B}_k}^2.
\end{eqnarray}
The assertion holds.
\end{proof}

In the rest of the paper, we use $F_k$ to denote $F(X_k)$.
By the definition of $l(k)$ (see \eqref{lk}), we have $F_{l(k)}=\max\limits_{\max\{k-m,0\}\leq j\leq k}F_j$.
We will use the notation
\begin{equation}\label{alpha}
\overline{\alpha}:=
\min\{1,\frac{(2-\sigma)\kappa_1}{2(\varrho M_2+\frac{1}{2}LM_1^2+L_hM_2)}\},
\end{equation}
where $L$ and $L_h$ are Lipschitz constants, $M_1$, $M_2$, $\sigma$, $\varrho$ and $\kappa_1$ are parameters in \eqref{eq5}, \eqref{eq6},
\eqref{nonmonotone}, \eqref{wang} and \eqref{kappa} respectively.

\begin{lemma}\label{lm42_line_search}
Suppose Assumption \ref{as_func} holds, and $\mathcal{B}_k$ satisfies \eqref{kan}. 
Let $\alpha_k$ be the stepsize of the $k$-th iteration of Algorithm \ref{algo1}.
Then $\alpha_k\geq\gamma\overline{\alpha}$, where $\overline{\alpha}$ is defined by \eqref{alpha} and $\gamma$ is the parameter of Algorithm \ref{algo1}.
Moreover,
\[
F_{k+1}-F_{l(k)}\leq-\frac{1}{2}\sigma\alpha_k\|V_k\|_{\mathcal{B}_k}^2.
\]
\end{lemma}

\begin{proof}
Since $\nabla f$ is Lipschitz continuous with constant $L$, for any $\alpha>0$, we have
\begin{eqnarray}
 &&f({\bf R}_{X_k}(\alpha V_k)) \nonumber\\
 &\leq& f(X_k)+\langle\nabla f(X_k),{\bf R}_{X_k}(\alpha V_k)-X_k\rangle+\frac{L}{2}\|{\bf R}_{X_k}(\alpha V_k)-X_k\|^2  \nonumber\\
 &\leq& f(X_k)+\langle\nabla f(X_k),{\bf R}_{X_k}(\alpha V_k)-(X_k+\alpha V_k)\rangle+\langle\nabla f(X_k),\alpha V_k\rangle+\frac{1}{2}LM_1^2\|\alpha V_k\|^2,\nonumber\\
&\leq& f(X_k)+\langle\nabla f(X_k),\alpha V_k\rangle+(\varrho M_2+\frac{1}{2}LM_1^2)\|\alpha V_k\|^2\nonumber\\
&=& f(X_k)+\langle\nabla f(X_k),\alpha V_k\rangle+c_1\|\alpha V_k\|^2, \label{eq34_1}
\end{eqnarray}
where $c_1:=\varrho M_2+\frac{1}{2}LM_1^2$, the second inequality follows from \eqref{eq5},
and the third inequality follows from \eqref{eq6} and \eqref{wang}.

By Assumption \ref{as_func} and \eqref{eq6}, it holds that
\begin{eqnarray}
\nonumber h({\bf R}_{X_k}(\alpha V_k))-h(X_k+\alpha V_k)&\leq& L_h\|{\bf R}_{X_k}(\alpha V_k)-X_k-\alpha V_k\|\\
&\leq& L_hM_2\|\alpha V_k\|^2.  \label{eq34_1_h}
\end{eqnarray}
Combining \eqref{eq34_1} and \eqref{eq34_1_h} yields
\begin{eqnarray}
&&F({\bf R}_{X_k}(\alpha V_k))\nonumber\\
&\leq& f(X_k)+\langle\nabla f(X_k),\alpha V_k\rangle+(c_1+L_hM_2)\|\alpha V_k\|^2 +h(X_k+\alpha V_k)\label{lele}\\
&=& f(X_k)+\phi_k(\alpha V_k) +c_2\|\alpha V_k\|^2 - \frac{1}{2}\|\alpha V_k\|_{\mathcal{B}_k}^2 \nonumber\\
&\leq& F_{l(k)}+\phi_k(\alpha V_k)-\phi_k(0)+(c_2\kappa_1^{-1} -\frac{1}{2})\|\alpha V_k\|_{\mathcal{B}_k}^2,\label{quan}
\end{eqnarray}
where $c_2:= c_1+L_hM_2=\varrho M_2+\frac{1}{2}LM_1^2+L_hM_2$.

For any $0<\alpha\leq 1$, by \eqref{lai} and \eqref{quan}, we have
\[
F({\bf R}_{X_k}(\alpha V_k))\leq F_{l(k)}+(c_2\kappa_1^{-1}-\alpha ^{-1})\alpha^2\|V_k\|_{\mathcal{B}_k}^2.
\]
Thus, if $0<\alpha\leq\overline{\alpha}$, it holds that
\begin{equation}\label{eq34_3}
F({\bf R}_{X_k}(\alpha V_k))-F_{l(k)}
\leq (c_2\kappa_1^{-1}\overline{\alpha}-1)\alpha\|V_k\|_{\mathcal{B}_k}^2
\leq -\frac{1}{2}\sigma\alpha\|V_k\|_{\mathcal{B}_k}^2.
\end{equation}
By steps 8-11 of Algorithm \ref{algo1}, we conclude that $\alpha_k\geq\gamma\overline{\alpha}$.
Substituting $\alpha=\alpha_k$ into \eqref{eq34_3} yields
\begin{eqnarray}
F_{k+1}-F_{l(k)} &=& F({\bf R}_{X_k}(\alpha_kV_k))-F_{l(k)} \leq -\frac{1}{2}\sigma\alpha_k\|V_k\|_{\mathcal{B}_k}^2.\label{eq34_2}
\end{eqnarray}
The proof is complete.
\end{proof}

Now we are in a position to present the main result of this section, the global
convergence of Algorithm \ref{algo1}.

\begin{theorem}\label{thm41}
Suppose Assumption \ref{as_func} holds, and $\mathcal{B}_k$ satisfies \eqref{kan}. 
Then all accumulation points of $\{X_k\}$ are stationary points of problem \eqref{eq1_prob}.
\end{theorem}

\begin{proof}
Let $X^*$ be an accumulation point of sequence $\{X_k\}$.
We need to prove that $X^*$ satisfies \eqref{eq2_opt}.
By Lemma \ref{lm42_line_search}, we have that
\begin{eqnarray}
\nonumber F_{l(k+1)} &=& \max_{0\leq j\leq \min\{m, k+1\}}F_{k+1-j}\\
\nonumber &=& \max\{F_{k+1}, \max_{0\leq j\leq \min\{m-1, k\}}F_{k-j}\}\\
\nonumber &\leq& \max\{F_{l(k)}-\frac{1}{2}\sigma\alpha_k\|V_k\|_{\mathcal{B}_k}^2, F_{l(k)}\},\ {\rm by} \eqref{eq34_2} \\
\nonumber &\leq& F_{l(k)},
\end{eqnarray}
which implies that $\{F_{l(k)}\}_k$ is a monotone nonincreasing sequence.
Since $\mathcal{M}$ is compact, $F_k$ is bounded below.
Thus, there exists a scalar $F^*$ such that
\begin{equation}\label{eq32_6}
\lim_{k\to\infty}F_{l(k)}=F^*.
\end{equation}

By \eqref{kan} and \eqref{eq34_2}, we have
\begin{eqnarray}
\nonumber F_{l(k)}&\leq& F_{l(l(k)-1)}-\frac{1}{2}\sigma\alpha_{l(k)-1}\|V_{l(k)-1}\|_{\mathcal{B}_{l(k)-1}}^2\\
\nonumber&\leq&F_{l(l(k)-1)}-\frac{1}{2}\sigma\kappa_1\alpha_{l(k)-1}\|V_{l(k)-1}\|^2.
\end{eqnarray}
From Lemma \ref{lm42_line_search}, it follows that $\alpha_k\geq\gamma\overline{\alpha}$ for all $k$,
Thus, we must have
\[
\lim_{k\to\infty}V_{l(k)-1}=0.
\]
Combining it with (\ref{eq32_6}) yields that
\begin{eqnarray}
\nonumber \lim_{k\to\infty}F_{l(k)-1}&=&\lim_{k\to\infty}F({\bf R}_{X_{l(k)-1}}(\alpha_{l(k)-1}V_{l(k)-1}))\\
\nonumber &=&\lim_{k\to\infty}F(X_{l(k)})=F^*.
\end{eqnarray}

For all $1\leq j\leq m$, we can prove by induction that
\begin{equation}\label{eq32_7}
\lim_{k\to\infty}V_{l(k)-j}=0, \
{\rm and} \
\lim_{k\to\infty}F_{l(k)-j}=F^*.
\end{equation}
The proof is similar to that of \cite[Theorem 1]{wright2008}, and so we omit it.

For any $k$, there exists an integer $1\leq j(k)\leq m$ such that $k=l(k+m)-j(k)$,
which together with \eqref{eq32_7} implies
\begin{equation}\label{eq32_8}
\lim_{k\to\infty}V_k=\lim_{k\to\infty}V_{l(k+m)-j(k)}=0,
\end{equation}
and
\begin{equation}\label{eq32_9}
\lim_{k\to\infty}F_k=\lim_{k\to\infty}F_{l(k+m)-j(k)}=F^*.
\end{equation}
By \eqref{eq32_8} and \eqref{eq4_sub_opt}, we can deduce that $X^*$ satisfies \eqref{eq2_opt},
which completes the proof.
\end{proof}

In the following, we introduce the definition of an $\epsilon$-stationary point of the problem \eqref{eq1_prob}.

\begin{definition}[$\epsilon$-stationary point \cite{mashiqian2020}]
Given $\epsilon>0$ and $X_k$ generated by Algorithm \ref{algo1}, we say that $X_k\in\mathcal{M}$ is an $\epsilon$-stationary point of \eqref{eq1_prob} if the solution $V_k$ to \eqref{eq3_subprob} satisfies $\|V_k\|_{F}\leq\epsilon$.
\end{definition}

In Algorithm \ref{algo1}, we use $\|V_k\|_{F}\leq\epsilon$ as our stopping criterion.
Similar to \cite[Theorem 5.5]{mashiqian2020}, we give the iteration complexity analysis of Algorithm \ref{algo1}.

\begin{corollary}\label{cor43}
Algorithm \ref{algo1} will find an $\epsilon$-stationary point in at most 
\[
(m+1)\left\lceil \frac{2(F(X_0)-F^*)}{\sigma\kappa_1\gamma\overline{\alpha}\epsilon^2}\right\rceil,
\]
iterations, where $m$, $\sigma$, $\kappa_1$, $\overline{\alpha}$ are defined in \eqref{nonmonotone}, \eqref{kappa}, \eqref{alpha} respectively
and $F^*$ is the optimal value of \eqref{eq1_prob}.
\end{corollary}

\begin{proof}
Similar to the proof of Theorem 3.2 in \cite{dyh2002}, we can obtain that for $j\geq0$,
\begin{eqnarray}
\nonumber 
& &
F(X_{l((j+1)(m+1))}) - F(X_{l(j(m+1))}) 
\\
\nonumber
&\leq& 
\max_{1\leq i\leq m}\left\{
-\frac{1}{2}\sigma\alpha_{j(m+1)+i}\|V_{j(m+1)+i}\|_{\mathcal{B}_{j(m+1)+i}}^2\right\}\\
\label{eq43_1} 
&\leq& 
\max_{1\leq i\leq m}
\left\{-\frac{1}{2}\sigma\kappa_1\alpha_{j(m+1)+i}\|V_{j(m+1)+i}\|_{F}^2\right\}.
\end{eqnarray}

Given $K>0$, suppose that after $K(m+1)$ iterations, Algorithm \ref{algo1} does not terminate, which means that $\|V_k\|_{\mathcal{B}_k}^2>\kappa_1\epsilon^2$ for any $0\leq k\leq K(m+1)-1$.
It follows that
\begin{eqnarray}
\nonumber 
F(X_0)-F^*
&\geq& 
F(X_{l(0)})-F(X_{l(K(m+1))})\\
\nonumber 
&\geq& 
\sum_{j=0}^{K-1}\min_{1\leq i\leq m}
\left\{\frac{1}{2}\sigma\kappa_1\alpha_{j(m+1)+i}\|V_{j(m+1)+i}\|_{F}^2\right\}
> 
\frac{1}{2}\sigma\kappa_1\epsilon^2 K\gamma\overline{\alpha},
\end{eqnarray}
where the second inequality follows from \eqref{eq43_1}. Thus, Algorithm \ref{algo1} will return an $\epsilon$-stationary point in at most $(m+1)\lceil 2(F(X_0)-F^*)/(\sigma\kappa_1\gamma\overline{\alpha}\epsilon^2)\rceil$ iterations.
\end{proof}

\subsection{Locally Linear Convergence}\label{sec42}

The objective of this subsection is to show that Algorithm \ref{algo1} has a local linear convergence rate around the nondegenerate local minimum point.
Let $\{X_k\}$ be the sequence of iterates generated by Algorithm \ref{algo1} and $\overline{X}^*$ be any accumulation point of $\{X_k\}$. By
\eqref{eq32_9}, we know that
\begin{equation}\label{eq4_fstar}
F(\overline{X}^*)=F^*,
\end{equation}
where $F^*$ is the scalar in \eqref{eq32_9}.

%

We need the following assumption before presenting our main results.


\begin{assumption}\label{as43}
The function $f$ is twice continuously differentiable. The sequence $\{X_k\}$ has an accumulation point $X^*$ such that
\begin{equation}\label{eq4_hess}
\lambda_{\min}({\rm Hess} (f\circ{\bf R}_{X^*})(0_{X^*}))\geq\widetilde{\eta},
\end{equation}
where $\widetilde{\eta}>5L_hM_2$.
\end{assumption}


The constant $M_2$ is defined in \eqref{eq6}.
We should point out Assumption \ref{as43} is not a strong condition.
In some typical applications, $h(X)=\mu\|X\|_1$.
We can see that if 
\[
\mu<\frac{1}{5M_2}\lambda_{\min}({\rm Hess} (f\circ{\bf R}_{X^*})(0_{X^*})),
\]
then \eqref{eq4_hess} will be satisfied .

\begin{definition}($\varsigma$-strongly convex function \cite[Definition 2.1.3]{nesterov})\label{def_strongly_convex}
The function $g(x)$ is said to be a $\varsigma$-strongly convex function if $g(x)-\frac{1}{2}\varsigma\|x\|^2$ is convex,
where $\varsigma$ is called the convexity parameter of $g$.
\end{definition}
By the definition, if $g(x)$~is strongly convex with parameter $\varsigma$, it is easy to prove
\begin{eqnarray}\label{di}
g(x)-g(x^*)\geq\frac{\varsigma}{2}\|x-x^*\|^2\quad\forall x,
\end{eqnarray}
where $x^*$ is the unique minimizer of $g$. We also use the following property of strongly convex functions.
If $g$ is twice continuously differentiable, then
\begin{eqnarray}\label{lao}
\textrm{$\lambda_{\min}(\nabla^2g(x))\geq\varsigma,\quad\forall x$ $\Longleftrightarrow$ $g$~is strongly convex with parameter $\varsigma$.}
\end{eqnarray}
For a proof, see \cite{nesterov}.

To establish the main results, we need some preparing results.

\begin{lemma}\label{lm43}
Suppose Assumptions \ref{as_func} and \ref{as43} hold. Let $X^*$ be the accumulation point satisfying \eqref{eq4_hess}.
Then there exists a neighbourhood $\mathcal{U}_{X^*}$ of $X^*$ and $\epsilon>0$ such that:
\begin{enumerate}
\item[\rm(1)] For all $X\in \mathcal{U}_{X^*}$, $f\circ{\bf R}_X$ is a convex function on the set $\{\xi\in{\rm T}_X\mathcal{M} : \|\xi\|<\epsilon\}$.
\item[\rm(2)]
For all $X\in \mathcal{U}_{X^*}$,
\begin{equation}\label{eq4_eta}
F(X)-F(X^*)\geq\eta\|{\bf R}_{X}^{-1}(X^*)\|^2,
\end{equation}
for some $\eta>L_hM_2$.
\end{enumerate}
\end{lemma}

\begin{proof}
$\rm(1)$. Since $f$ is twice continuously differentiable and ${\bf R}$ is smooth,
$\lambda_{\min}({\rm Hess} (f\circ{\bf R}))$ is a continuous function of ${\rm T}\mathcal{M}$,
which together with \eqref{eq4_hess} implies that
there exists a neighbourhood $\mathcal{U}_{X^*}$ of $X^*$ and $\epsilon>0$ such that
\begin{equation}\label{eq4_hess2}
\lambda_{\min}({\rm Hess} (f\circ{\bf R}_{X})(\xi_{X}))>\frac{9}{10}\widetilde{\eta},
\end{equation}
for all $X\in\mathcal{U}_{X^*}$ and all $\xi_{X}\in\{\xi\in{\rm T}_X\mathcal{M} : \|\xi\|<\epsilon\}$.
Thus $f\circ{\bf R}_X$ is a convex function on $\{\xi\in{\rm T}_X\mathcal{M} : \|\xi\|<\epsilon\}$.


$\rm(2)$.
By \eqref{lao} and \eqref{eq4_hess2}, we know that $f({\bf R}_{X^*}(\xi))$ is $9\widetilde{\eta}/10$-strongly convex on $\mathcal{D}_{X^*}$ where $\mathcal{D}_{X^*}:=\{\xi\in{\rm T}_{X^*}\mathcal{M} : \|\xi\|<\epsilon\}$.

Define $\Gamma:{\rm T}_{X^*}\mathcal{M}\rightarrow\mathbb{R}$ by $\Gamma(\xi)=f({\bf R}_{X^*}(\xi))+h(X^*+\xi)$.
Then $\Gamma$ is also $9\widetilde{\eta}/10$-strongly convex on $\mathcal{D}_{X^*}$.
By \eqref{r}, we have
\begin{eqnarray*}
\partial\Gamma(0_{X^*})&=&{\rm Proj}_{{\rm T}_{X^*}\mathcal{M}}\big({\rm D}{\bf R}_{X^*}(0_{X^*})\nabla f(X^*)+\partial h(X^*)\big)\\
&=&{\rm grad} f(X^*)+{\rm Proj}_{{\rm T}_{X^*}\mathcal{M}}(\partial h(X^*)),
\end{eqnarray*}
which together with \eqref{eq2_opt} implies $0\in \partial\Gamma(0_{X^*})$.
Thus $0_{X^*}$ is the unique minimizer of $\Gamma$ in $\mathcal{D}_{X^*}$.
By \eqref{di}, we have
\begin{equation}\label{t1}
\Gamma(\xi)-F(X^*)=\Gamma(\xi)-\Gamma(0_{X^*})\geq\frac{9}{20}\widetilde{\eta}\|\xi\|^2\quad\forall\xi\in\mathcal{D}_{X^*}.
\end{equation}
Let $\varphi(\xi)=h({\bf R}_{X^*}(\xi))-h(X^*+\xi)$.
From \eqref{eq6} and Assumption \ref{as_func}, it holds that
\begin{equation}\label{t2}
\vert\varphi(\xi)\vert\leq L_hM_2\|\xi\|^2\quad\forall\xi\in\mathcal{D}_{X^*}.
\end{equation}
By \eqref{t1} and \eqref{t2}, we can deduce that
\begin{eqnarray}
F({\bf R}_{X^*}(\xi))-F(X^*) &=&
 f({\bf R}_{X^*}(\xi))+ h({\bf R}_{X^*}(\xi))-F(X^*)\nonumber\\
&=& \Gamma(\xi)+ \varphi(\xi)-\Gamma(0_{X^*})\nonumber\\
&\geq& (\frac{9}{20}\widetilde{\eta}-L_hM_2)\|\xi\|^2.\label{eq4_3_f}
\end{eqnarray}
Substituting $\xi={\bf R}_{X^*}^{-1}(X)$ into \eqref{eq4_3_f} yields
\begin{equation}\label{eq4_strong_F}
F(X)-F(X^*)\geq\frac{1}{4}\widetilde{\eta}\|{\bf R}_{X^*}^{-1}(X)\|^2.
\end{equation}

By \cite[Lemma 6]{ring2012}, for any $\varepsilon>0$, there exists a neighbourhood $\mathcal{U}_{X^*}$ of $X^*$ and $\varepsilon^{\prime}>0$ such that for all $X\in\mathcal{U}_{X^*}$ and $V,W\in{\rm T}_X\mathcal{M}$ with $\|V\|,\|W\| <\varepsilon^{\prime}$,
\begin{equation}\label{eq4_hw}
(1-\varepsilon)\|V-W\|  \leq {\rm dist}({\bf R}_X(V),{\bf R}_X(W)) \leq (1+\varepsilon)\|V-W\| .
\end{equation}
Assume that $\mathcal{U}_{X^*}$ is small enough such that $\|{\bf R}_{X^*}^{-1}(X)\|<\varepsilon^{\prime}$ and $\|{\bf R}_{X}^{-1}(X^*)\|<\varepsilon^{\prime}$ for any $X\in\mathcal{U}_{X^*}$.
By \eqref{eq4_hw}, we have
\begin{equation}\label{qin}
(1-\varepsilon)\|{\bf R}_{X}^{-1}(X^*)\|  \leq \|X-X^*\| \leq (1+\varepsilon)\|{\bf R}_{X^*}^{-1}(X)\|,
\end{equation}
where the first inequality follows from substituting $V=0_X$ and $W={\bf R}_{X}^{-1}(X^*)$ into \eqref{eq4_hw}, and the second inequality from $V={\bf R}_{X^*}^{-1}(X)$ and $W=0_{X^*}$ in \eqref{eq4_hw}.
We can choose $\varepsilon$ satisfying $\varepsilon<(\sqrt{5}-2)^2$.  By \eqref{eq4_strong_F} and \eqref{qin}, we have
\begin{equation*}
F(X)-F(X^*)\geq\frac{(1-\varepsilon)^2}{4(1+\varepsilon)^2}\widetilde{\eta}\|{\bf R}_{X}^{-1}(X^*)\|^2\geq\frac{1}{5}\widetilde{\eta}\|{\bf R}_{X}^{-1}(X^*)\|^2\quad\forall X\in\mathcal{U}_{X^*},
\end{equation*}
which together with $\widetilde{\eta}>5L_hM_2$ implies \eqref{eq4_eta}.
\end{proof}

Under the condition of Assumption \ref{as43}, from the following result,
we know that the sequence $\{X_k\}$ has only one accumulation point $X^*$, which is of course the limit of $\{X_k\}$.


\begin{theorem}\label{thm4_new}
Suppose Assumptions \ref{as_func} and \ref{as43} hold, and $X^*$ is the accumulation point satisfying \eqref{eq4_hess}. Then,
$X_k$ converges to $X^*$.
\end{theorem}

\begin{proof}
Let $\mathcal{D}_{X^*}:=\{\xi\in{\rm T}_{X^*}\mathcal{M} : \|\xi\|<\epsilon\}$ be the neighbourhood defined in Lemma \ref{lm43}. 
Define $\overline{\mathcal{U}}_{X^*}:=\{{\bf R}_{X^*}(\xi): \xi\in \mathcal{D}_{X^*}\}$.
By \eqref{eq4_eta}, $X^*$ is the unique minimizer of $F$ in $\overline{\mathcal{U}}_{X^*}$,
which together with \eqref{eq4_fstar} implies that $X^*$ is an isolated accumulation point of $\{X_k\}$.
Since $V_k\rightarrow0$ by \eqref{eq32_8}, we have 
\[
\|X_{k+1}-X_k\|=\|{\bf R}_{X_{k}}(\alpha_kV_k)-X_k\|
\leq M_1\alpha_k\|V_k\|
\rightarrow0.
\]
By \cite[Lemma (4.10)]{more1983}, we can obtain that $X_k\rightarrow X^*$.
\end{proof}

From Lemma \ref{lm42_line_search}, we know that the stepsize $\alpha_k$ satisfies
$\gamma\overline{\alpha}<\alpha_k\leq 1$ for all $k\geq 0$, where $\overline{\alpha}$ is defined in \eqref{alpha}.
We will use this fact to prove the local linear convergence of Algorithm \ref{algo1}. 
Our proof is based on a technique used in \cite[Theorem 3.4]{huangyakui2015}.

\begin{theorem}\label{thm42}
Suppose Assumptions \ref{as_func} and \ref{as43} hold, $X^*$ is the accumulation point satisfying \eqref{eq4_hess}, and $\mathcal{B}_k$ satisfies \eqref{kan}.
Then there exists an integer $K$, $\mu>0$ and $\tau\in(0,1)$ such that
\begin{equation}\label{huo}
F(X_k)-F(X^*)\leq \mu\tau^{k-K}\big(F(X_{l(K)})-F(X^*)\big), \ {\rm for \ all}\ k> K.
\end{equation}
\end{theorem}

\begin{proof}
Let $\mathcal{U}_{X^*}$ be the neighbourhood of $X^*$ and $\epsilon>0$ be the constant such that all statements of Lemma \ref{lm43} hold.
By Theorem \ref{thm4_new}, there exists an integer $K>0$ such that $X_k\in\mathcal{U}_{X^*}$ for all $k\geq K$.
From $X_k\rightarrow X^*$, it follows that $\|{\bf R}_{X_k}^{-1}(X^*)\|\rightarrow0$.
Without loss of generality, assume that $\|{\bf R}_{X_k}^{-1}(X^*)\|<\epsilon$ for all $k\geq K$.

We separate our proof into three parts.

Part (1). By \eqref{lele}, we have
\begin{equation}\label{eq4_1}
F_{k+1} \leq f(X_k)+\langle\nabla f(X_k),\alpha_kV_k\rangle+ c_2\|\alpha_kV_k\|^2 +h(X_k+\alpha_kV_k),
\end{equation}
where $c_2:=\varrho M_2+\frac{1}{2}LM_1^2+L_hM_2$.
Since $h$ is convex and $\alpha_k\in(0,1]$, it holds that
\begin{eqnarray}
\nonumber h(X_k+\alpha_kV_k) &\leq& \alpha_kh(X_k+V_k) + (1-\alpha_k)h(X_k).
\end{eqnarray}
Combining it with \eqref{eq4_1} yields
\begin{equation}\label{eq4_2}
F_{k+1} \leq (1-\alpha_k)F_k+\alpha_k(f(X_k)+\phi_{k}(V_k))+ (c_2-\frac{\kappa_1}{2\alpha_k})\|\alpha_kV_k\|^2.
\end{equation}

From Lemma \ref{lm43}, we know that $f\circ{\bf R}_X$ is convex on the set $\{V\in{\rm T}_X\mathcal{M} : \|V\|<\epsilon\}$.
For $V\in{\rm T}_X\mathcal{M}$, it holds that ${\rm grad}(f\circ{\bf R}_X)(0_X)={\rm Proj}_{{\rm T}_{X}\mathcal{M}}\nabla f(X)$.
If $\|V\|<\epsilon$, then
\begin{eqnarray}\label{yuan}
f({\bf R}_X(V))-f(X) \geq\langle \nabla f(X) ,V\rangle.
\end{eqnarray}
By the definition of $\phi_{k}$ (see \eqref{eq3_subprob_func}) and $V_k=\arg\min_{V\in{\rm T}_{X_k}\mathcal{M}}\phi_k(V)$, we have
\begin{eqnarray}
\nonumber f(X_k)+\phi_{k}(V_k) &=& \min_{V\in T_{X_k}\mathcal{M}} \{ f(X_k)+\langle\nabla f(X_k), V\rangle+ \frac{1}{2}\|V\|_{\mathcal{B}_k}^2+h(X_k+V)\}.
\end{eqnarray}
By combining it with \eqref{yuan}, for all $k\geq K$ and $\theta\in[0,1]$, we have
\begin{eqnarray}
&&f(X_k)+\phi_{k}(V_k)\nonumber\\
&\leq& \min_{V\in T_{X_k}\mathcal{M},\|V\|<\epsilon} \{f({\bf R}_{X_k}(V))+ \frac{1}{2}\|V\|_{\mathcal{B}_k}^2+h(X_k+V)\}\nonumber\\
&\leq& \theta f(X^*)+(1-\theta) F_k+ \frac{1}{2}\theta^2\kappa_2\|{\bf R}_{X_k}^{-1}(X^*)\|^2 +\theta h(X_k+{\bf R}_{X_k}^{-1}(X^*))\nonumber\\
&\leq& \theta f(X^*)+(1-\theta) F_k+ \frac{1}{2}\theta^2\kappa_2\|{\bf R}_{X_k}^{-1}(X^*)\|^2 +\theta h(X^*)+\theta L_hM_2\|{\bf R}_{X_k}^{-1}(X^*)\|^2\nonumber\\
&\leq& \theta F(X^*)+(1-\theta) F_k+ (\frac{1}{2}\theta^2\kappa_2+ \theta L_hM_2 )\|{\bf R}_{X_k}^{-1}(X^*)\|^2.\label{qu}
\end{eqnarray}
Denote $\Upsilon:=(\frac{1}{2}\theta^2\kappa_2+ \theta L_hM_2 )\|{\bf R}_{X_k}^{-1}(X^*)\|^2$. From \eqref{qu} and \eqref{eq4_2}, it follows that
\begin{eqnarray}
&& F_{k+1} \nonumber\\
&\leq& (1-\alpha_k)F_{k}+\alpha_k\big(\theta F(X^*)+(1-\theta) F_{k}+
\Upsilon\big) + (c_2-\frac{\kappa_1}{2\alpha_k})\|\alpha_kV_k\|^2\nonumber\\
&\leq& F_{l(k)}+\alpha_k\big( \theta(F(X^*)-F_{l(k)})+ \Upsilon\big)+c_2\|\alpha_kV_k\|^2.\label{eq4_3}
\end{eqnarray}
By \eqref{eq4_eta}, it holds that
\begin{eqnarray}\label{zheng}
F_{l(k)}-F(X^*)\geq F_{k}-F(X^*)\geq\eta\|{\bf R}_{X_k}^{-1}(X^*)\|^2.
\end{eqnarray}
Since $\alpha_k\leq1$, by \eqref{zheng} and \eqref{eq4_3}, we can deduce that
\begin{equation}\label{lei}
F_{k+1} \leq
F_{l(k)}+
\alpha_k
\left[
\left(F_{l(k)}-F(X^*)\right)(\frac{1}{2\eta}\kappa_2\theta^2-(1-\frac{1}{\eta}L_hM_2)\theta)
+c_2\|V_k\|^2
\right].
\end{equation}

Part (2). Now we prove that there exists a $\nu\in(0,1)$ such that for all $k>K$,
\begin{equation}\label{eq4_5}
F_{k+1}-F(X^*)\leq\nu(F_{l(k)}-F(X^*)).
\end{equation}
Let $\omega$ be a positive real number such that
\[
\omega<\min\{
\frac{2}{\sigma\gamma\overline{\alpha}},
\frac{\kappa_2}{2\eta c_2},
\frac{(\eta-L_hM_2)^2}{2\eta \kappa_2c_2} \}.
\]
Next, we consider two cases of the value of $\|V_k\|^2$.

\begin{enumerate}
\item[\rm(i)] $\|V_k\|^2\geq\omega(F_{l(k)}-F(X^*))$.
By \eqref{eq34_2}, we have
\[
\frac{2}{\sigma\gamma\overline{\alpha}}(F_{l(k)}-F_{k+1})
\geq\frac{2}{\sigma\alpha_k}(F_{l(k)}-F_{k+1})
\geq\|V_k\|^2\geq\omega (F_{l(k)}-F(X^*)).
\]
Thus, 
\[
F_{k+1}-F(X^*)\leq(1-\frac{\sigma\gamma\overline{\alpha}\omega}{2})(F_{l(k)}-F(X^*)),
\]
which implies \eqref{eq4_5}.
\item[\rm(ii)] $\|V_k\|^2<\omega(F_{l(k)}-F(X^*))$. Combining it with \eqref{lei} yields
\[
F_{k+1}
\leq
F_{l(k)}
+\alpha_k
\left(F_{l(k)}-F(X^*\right)(\frac{1}{2\eta}\kappa_2\theta^2-(1-\frac{1}{\eta}L_hM_2)\theta +c_2\omega).
\]
Denote $r_{k+1}:= F_{k+1}-F(X^*)$. Then, we have
\begin{equation}
\label{eq4_7}
r_{k+1}
\leq
\left[1+\alpha_k\left(\frac{1}{2\eta}\kappa_2\theta^2-(1-\frac{1}{\eta}L_hM_2)\theta+c_2\omega\right)\right]
\cdot r_{l(k)}.
\end{equation}
Define
$q(\theta)
:=
\frac{1}{2\eta}\kappa_2\theta^2-(1-\frac{1}{\eta}L_hM_2)\theta+c_2\omega$.
Let $\theta_{\min}:=\mathop{\arg\min}\limits_{0\leq\theta\leq1}q(\theta)$. Then
\[
\theta_{\min}=\min\{1,(\eta-L_hM_2)/\kappa_2\}.
\]

Consider the following two cases:
\begin{enumerate}
\item[\rm(a)] If $\theta_{\min}=1$, then
\[
\frac{1}{2\eta}\kappa_2\leq \frac{1}{2}(1-\frac{1}{\eta}L_hM_2),
\]
which together with $\omega<\kappa_2/(2\eta c_2)$ implies
\[
q(\theta_{\min})
=
\frac{1}{2\eta}\kappa_2-(1-\frac{1}{\eta}L_hM_2)+c_2\omega
\leq
-\frac{1}{2\eta}\kappa_2+c_2\omega
<0.
\]

\item[\rm(b)] Otherwise, $\theta_{\min}=(\eta-L_hM_2)/\kappa_2<1$.
By $\omega<(\eta-L_hM_2)^2/(2\eta \kappa_2c_2)$, we have
\[
q(\theta_{\min})
=
-\frac{1}{2\eta \kappa_2}(\eta-L_hM_2)^2+c_2\omega
<0.
\]
\end{enumerate}
In either case $\rm(a)$ or $\rm(b)$, substituting $\theta_{\min}$ into \eqref{eq4_7},  we can see that \eqref{eq4_5} holds.
\end{enumerate}

Part (3). For any $k> K$, there exists an integer $i\geq 1$ such that $(i-1)m<k-K\leq im$ where $m$ is the memory size parameter of the nonmonotone line search in \eqref{nonmonotone}.
By using \eqref{eq4_5} recursively, we have
\begin{eqnarray}
\nonumber r_k &\leq&r_{l(k)} \leq \nu r_{l(l(k)-1)}\leq \nu r_{l(k-m)}\leq\dots\\
\nonumber &\leq&\nu^{i-1}r_{l(k-(i-1)m)}\leq\nu^{i-1}r_{l(K)}\leq\nu^{(k-K)/m-1}r_{l(K)}.
\end{eqnarray}

Then \eqref{huo} follows from the above inequality by taking
\[
\mu:=\frac{1}{\nu}, \ \tau:=\nu^{1/m}.
\]
The proof is complete.
\end{proof}

\begin{corollary}\label{cor41}
Suppose the same assumptions hold as in Theorem \ref{thm42}.
Then there exists an integer $K$ and a constant $C_K>0$ such that
\begin{equation}\label{xi}
\|X_k - X^*\| \leq C_K\sqrt{\tau}^{k}, \ {\rm for \ all}\ k> K,
\end{equation}
where $\tau\in(0,1)$ is as in Theorem \ref{thm42}.
\end{corollary}

\begin{proof}
By \eqref{eq5}, \eqref{eq4_strong_F} and \eqref{huo}, for any $k\geq K$, we have
\begin{eqnarray}
\nonumber \|X_k - X^*\| &\leq& M_1\|{\bf R}_{X^*}^{-1}(X_k)\|\\
\nonumber &\leq& 2M_1\big(\frac{1}{\widetilde{\eta}}(F_k-F^*)\big)^{1/2}\\
\nonumber &\leq&  2M_1\big(\frac{1}{\widetilde{\eta}}\mu\tau^{-K}(F_{l(K)}-F^*)\big)^{1/2}\sqrt{\tau}^{k}.
\end{eqnarray}
Taking $C_K:=2M_1\big(\mu\tau^{-K}(F_{l(K)}-F^*)/\widetilde{\eta}\big)^{1/2}$ in the above inequality yields \eqref{xi}.
\end{proof}

\begin{remark}
The constant $\tau$ in \eqref{xi} only depends on $f$ and $h$,
while $K$ and $C_K$ depend on more factors.
Different initial point $X_0$ and stepsize $\alpha_k$ may lead to different $K$ and $C_K$.
\end{remark}

\section{Numerical Experiments}\label{sec_num_exp}

In this section, we report our numerical experiments
comparing our method with Riemannian proximal gradient methods, including ManPG and ManPG-Ada in \cite{mashiqian2020}, and NLS-ManPG which equips ManPG with the nonmonotone line search strategy.
Our objective is to show the efficiency of the proximal quasi-Newton method for composite optimization problems over the Stiefel Manifold.

Our test problems include the compressed modes (CM) problem, sparse principle component analysis (Sparse PCA), 
and the joint diagonalization problem with a regularization term. 
All of these experiments\footnote{Our MATLAB code is available at https://github.com/QinsiWang2022/ManPQN.} were conducted in MATLAB R2018b on a PC using Windows 10 (64bit) system with Intel Core i5 CPU (2.3GHz) and 8GB memory.

For the stopping criterion, we terminate our algorithm when $\|V_k\|^2\leq10^{-8}nr$, where $V_k\in\mathbb{R}^{n\times r}$ is defined by \eqref{duoduo},
or the algorithm reaches the maximum iteration number 30000.
For other parameters, the maximum iteration number of the inner loop is set to be $100$.
In the implementation of ManPQN, we set $m=10$ and $p=5$ (for $m$ and $p$, see \eqref{nonmonotone} and \eqref{eq3_lbfgs_hk}).
The parameters used in ManPG, ManPG-Ada and NLS-ManPG are set to be the default values in \cite{mashiqian2020}.
For all the problems, we use the singular value decomposition (SVD) as the retraction mapping in ManPQN, ManPG, ManPG-Ada and NLS-ManPG.

We report the numerical results obtained by solving randomly generated instances.
Specifically, we randomly generate 50 instances and record the averaged numerical performance of
these instances.
Numerical results are shown in several figures and tables.
In each figure, \textbf{CPU} denotes the CPU time in seconds,
\textbf{Iter} represents the number of iterations, \textbf{Sparsity} denotes the percentage of zeros in the local minimum $X^*$.
In each table, the total number of line search steps and the averaged iteration number of the adaptive regularized semismooth Newton (ASSN) method are reported.

\subsection{CM Problem}\label{sec51}

The compressed modes (CM) problem aims to find sparse solutions of systems of equations in physics,
including the Schr\"odinger equation in quantum mechanics.
The CM problem can be written as
\begin{equation}\label{eq5_1}
\min_{X\in\mathcal{M}}\mathrm{tr}(X^THX)+\mu\|X\|_1,
\end{equation}
where $H$ is the discretized Schr\"odinger operator.
For details of the CM problem, the reader is referred to \cite{cm2013}.

We can observe from Figures \ref{fig51_1}-\ref{fig51_3} and Tables \ref{tb51_1}-\ref{tb51_3} that the ManPQN method outperforms ManPG, ManPG-Ada and NLS-ManPG, which demonstrates the efficiency of the quasi-Newton strategy used in ManPQN.
ManPQN requires less computational time and less iterations 
than ManPG related methods, especially when $n$ and $r$ are large.
From these results, we can see that the quasi-Newton technique can accelerate the proximal gradient method for composite optimization problems over the Stiefel manifold.
ManPG related methods can achieve a solution with slightly better sparsity than ManPQN.
The reason for this is that we use an
approximate quasi-Newton strategy in our method (see \eqref{duoduo} and \eqref{hanwu}).

\begin{figure}[H]
	\centering  
	\subfigbottomskip=2pt 
	\subfigcapskip=-5pt 
	\subfigure[CPU]{
		\includegraphics[width=0.5\linewidth]{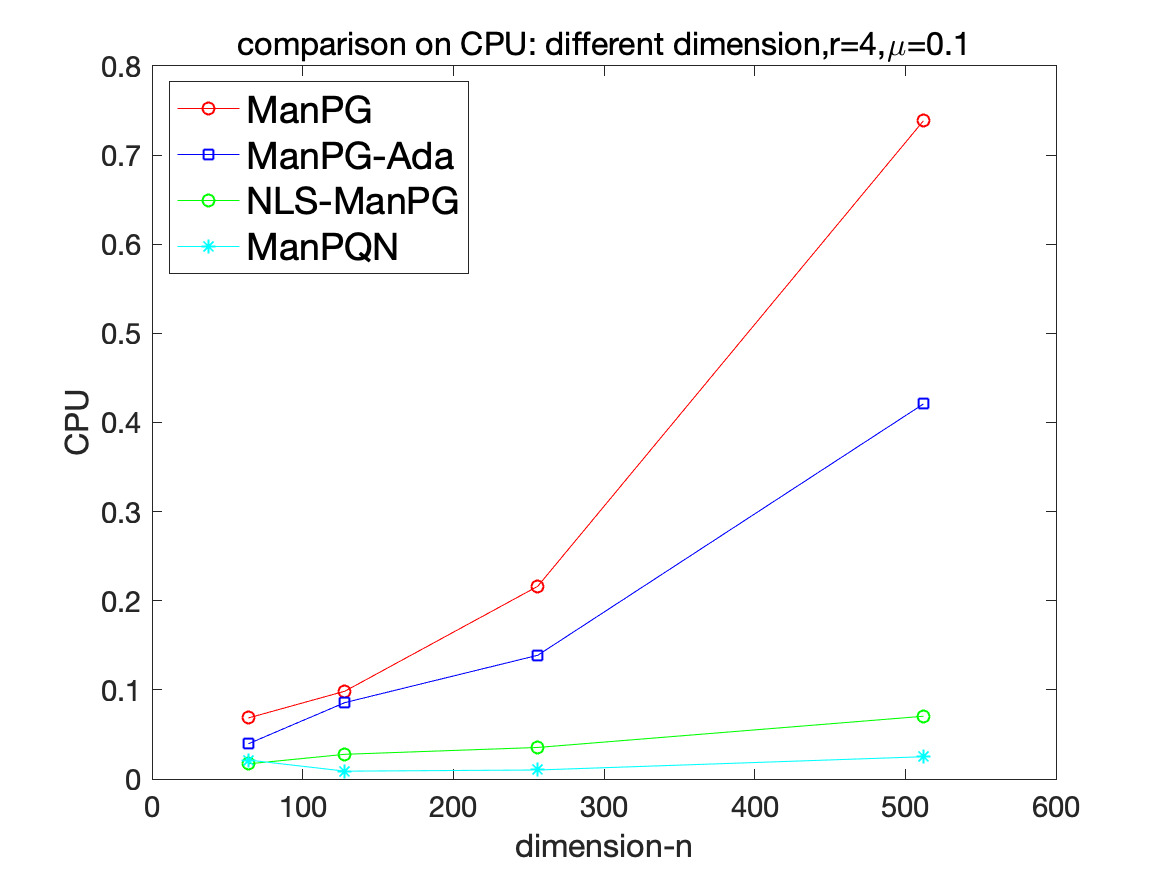}}
		\hspace{-5mm}
	\subfigure[Iter]{
		\includegraphics[width=0.5\linewidth]{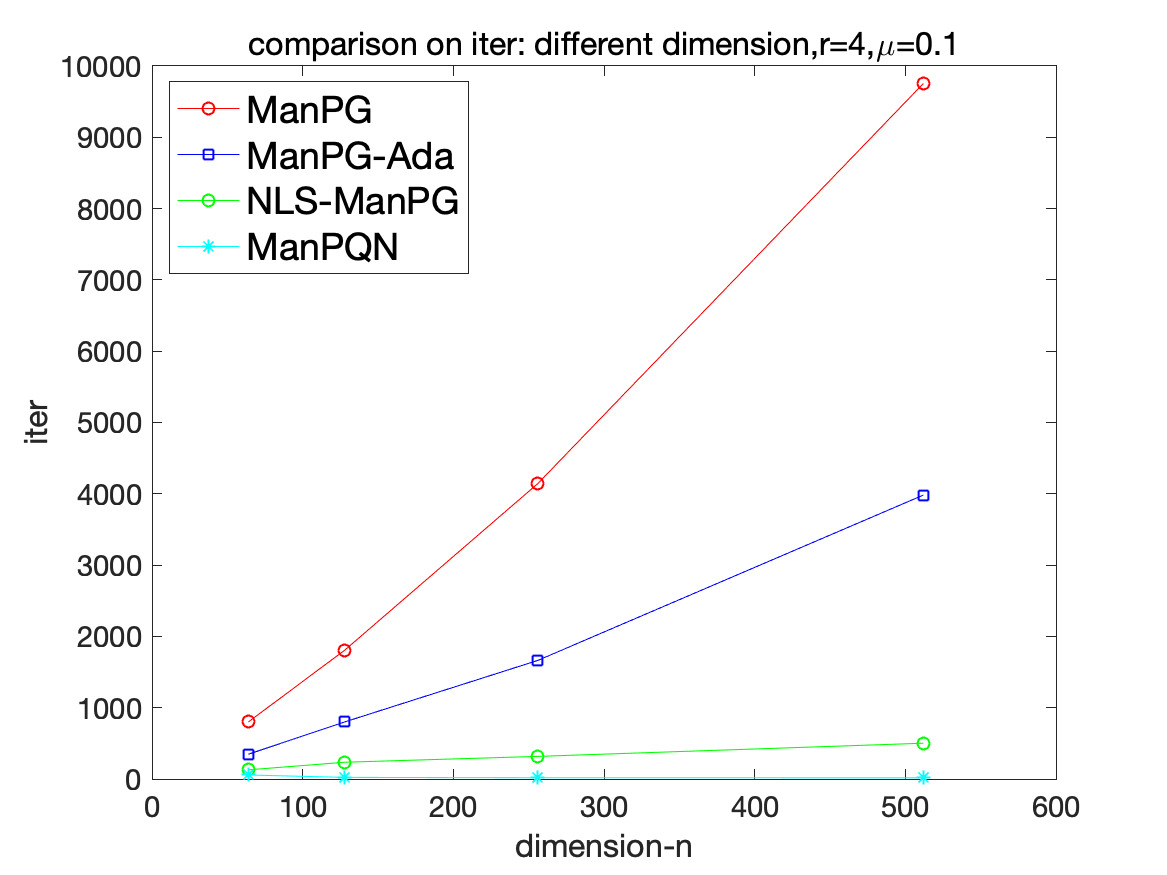}}
	\caption{Comparison on CM problem, different $n=\{64,128,256,512\}$ with $r= 4$ and $\mu=0.1$.}\label{fig51_1}
\end{figure}

\begin{figure}[H]
	\centering  
	\subfigbottomskip=2pt 
	\subfigcapskip=-5pt 
	\subfigure[CPU]{
		\includegraphics[width=0.5\linewidth]{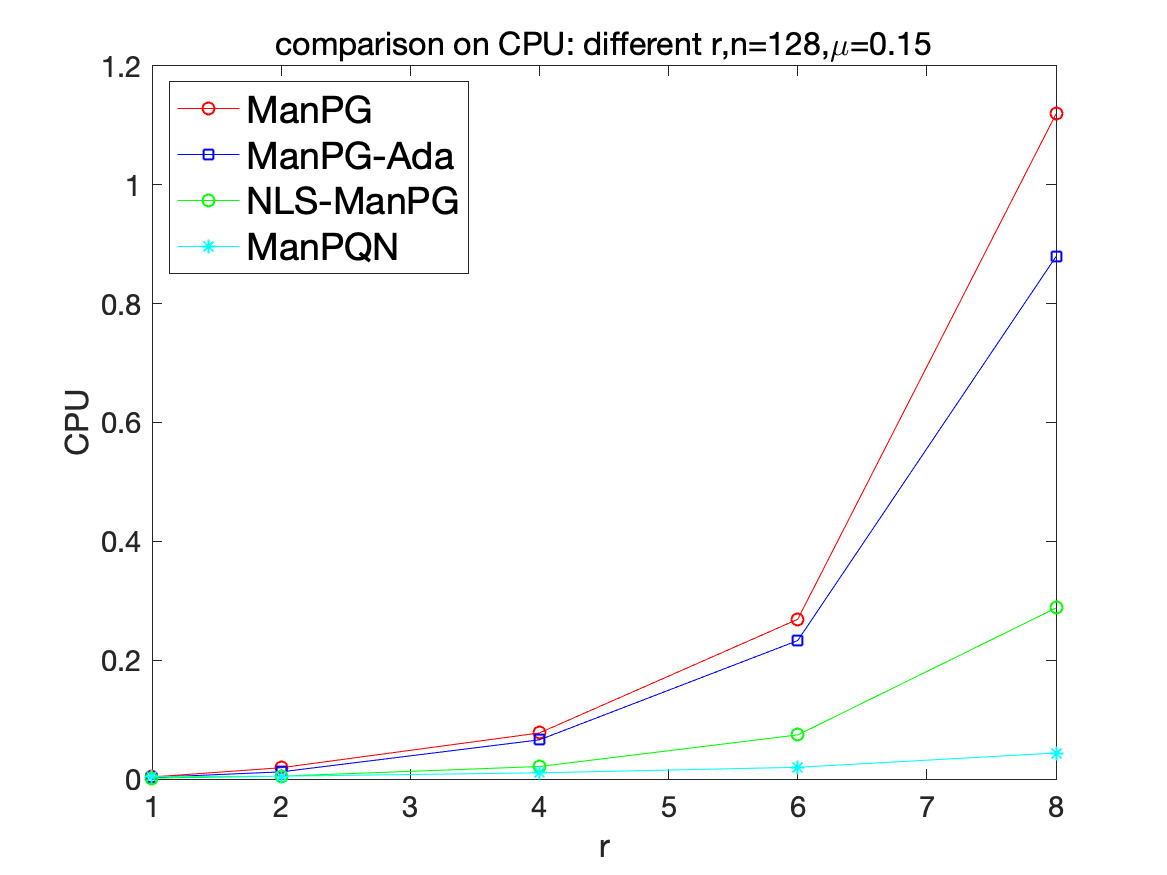}}
		\hspace{-5mm}
	\subfigure[Iter]{
		\includegraphics[width=0.5\linewidth]{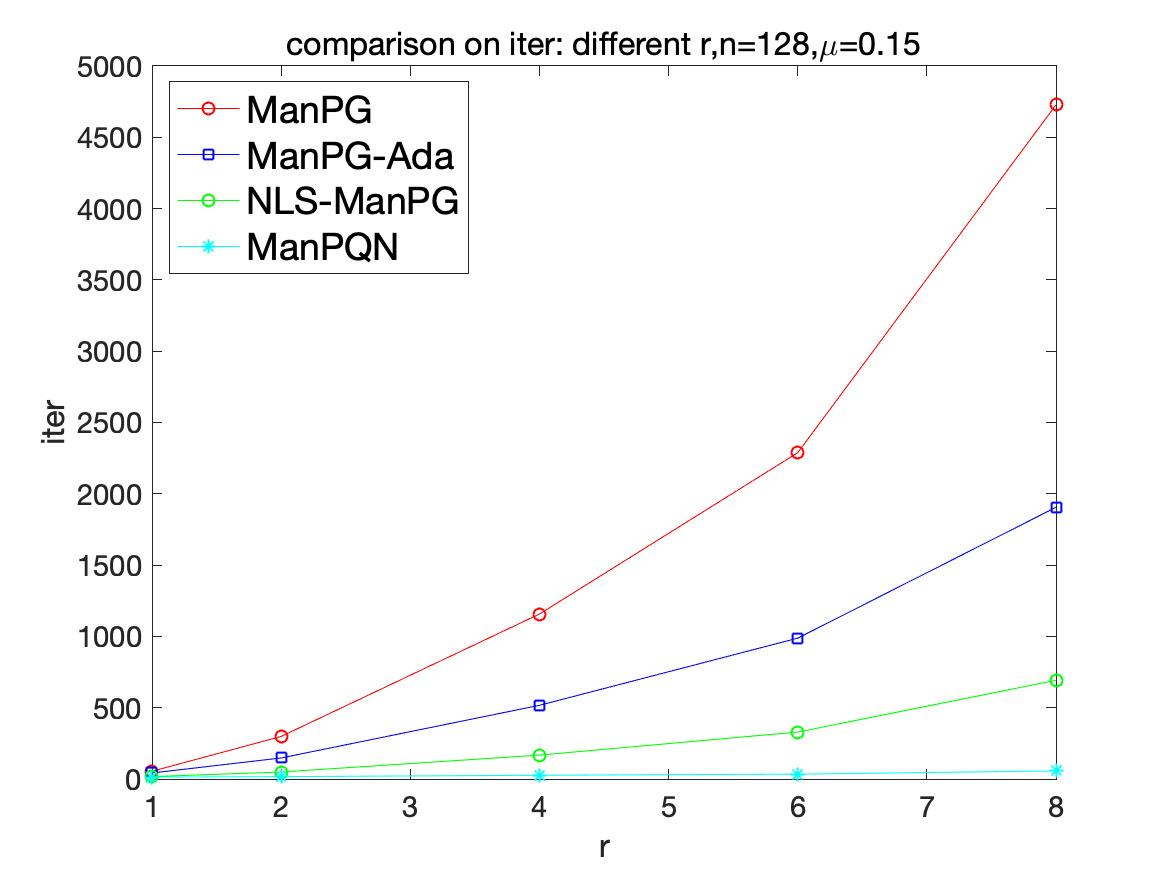}}
	\caption{Comparison on CM problem, different $r=\{1,2,4,6,8\}$ with $n= 128$ and $\mu=0.15$.}\label{fig51_2}
\end{figure}

\begin{figure}[H]
	\centering  
	\subfigbottomskip=2pt 
	\subfigcapskip=-5pt 
	\subfigure[CPU]{
		\includegraphics[width=0.5\linewidth]{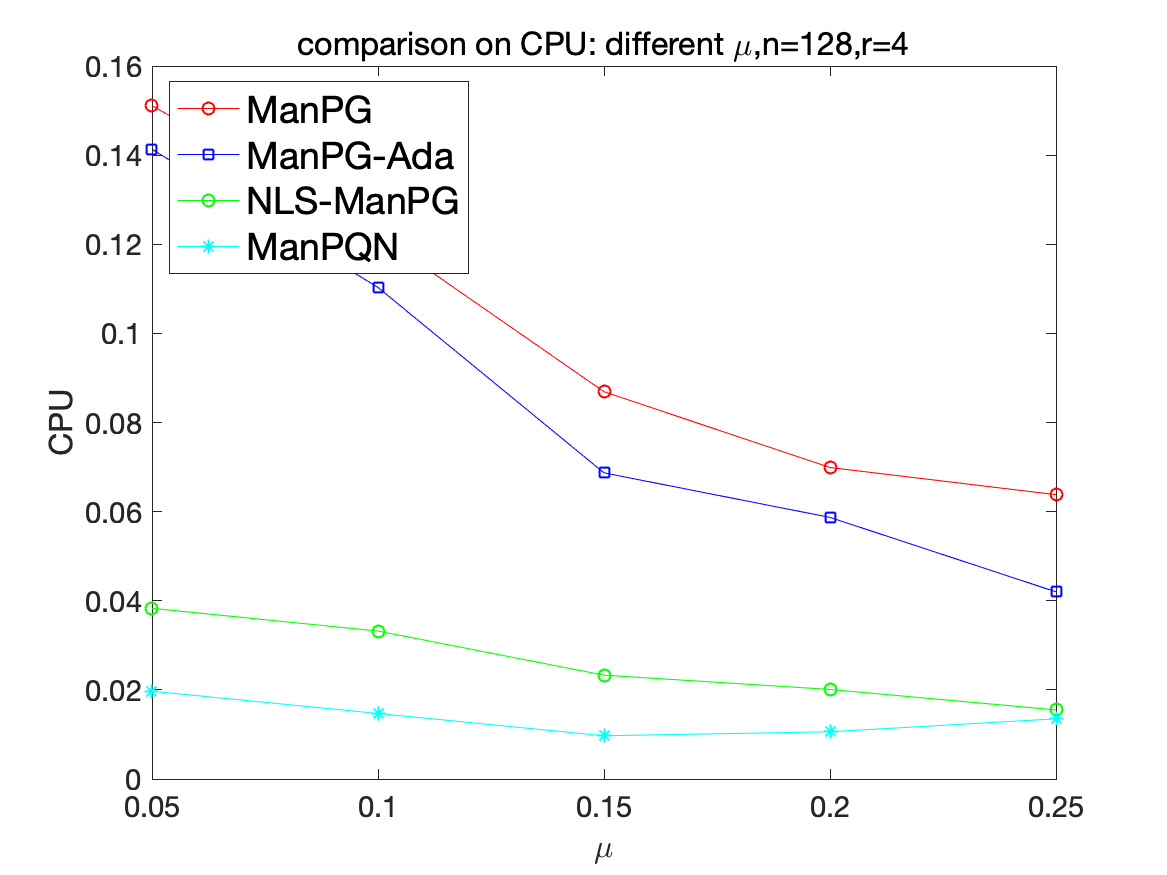}}
		\hspace{-5mm}
	\subfigure[Iter]{
		\includegraphics[width=0.5\linewidth]{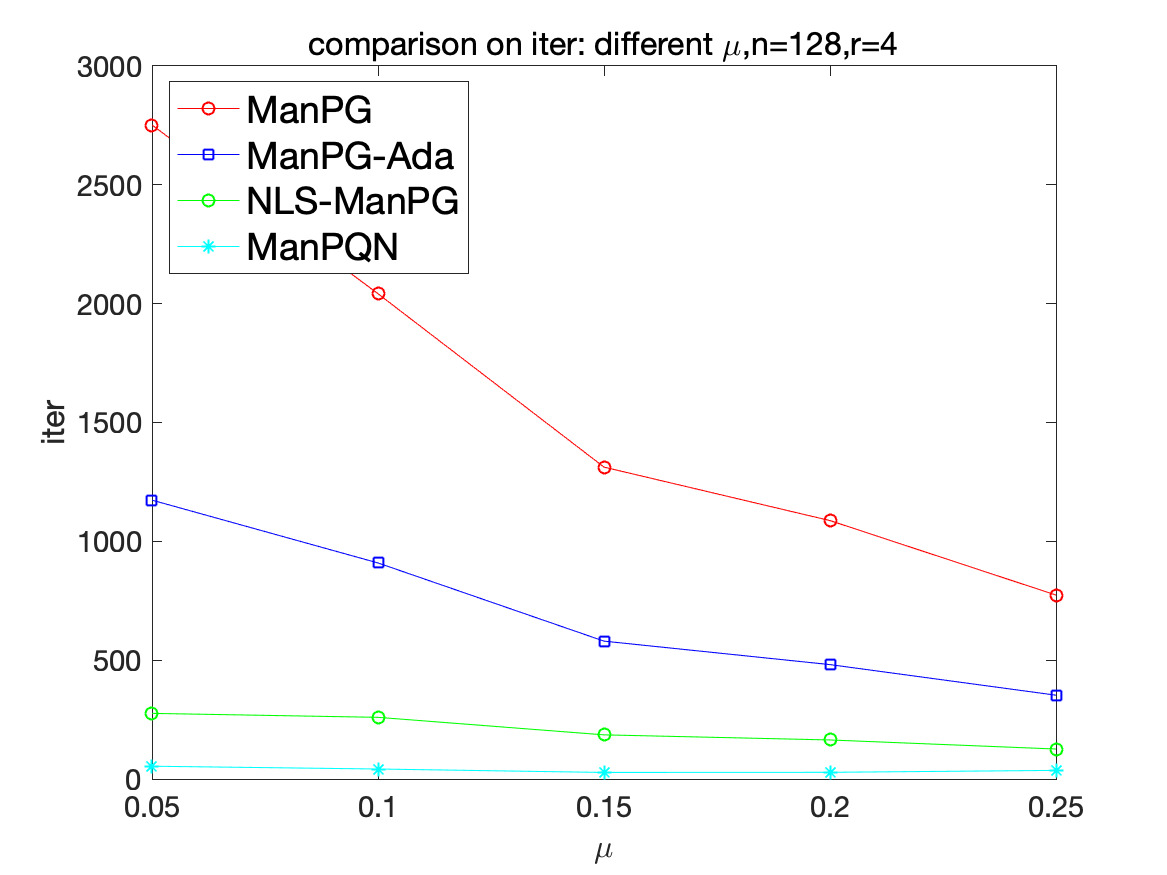}}
	\caption{Comparison on CM problem, different $\mu=\{0.05,0.10,0.15,0.20,0.25\}$ with $n= 128$ and $r=4$.}\label{fig51_3}
\end{figure}

\begin{table}[htbp]\centering
\caption{Comparison on CM problem, different $n=\{64,128,256,512\}$ with $r= 4$ and $\mu=0.1$.}\label{tb51_1}
\begin{tabular}{ccccccc}
\midrule
$n=64$ & Iter & $F(X^*)$  & sparsity & CPU time & \# line-search &  SSN iters  \\
\midrule
ManPG       &   800.74  & 1.424  &    0.82  &   0.0685   &   130.18   &   0.97  \\ 
ManPG-Ada  &   347.80  & 1.424  &    0.82  &   0.0396   &   130.72   &   1.26  \\ 
NLS-ManPG   &   128.74   & 1.424  &    0.82  &   0.0170   &   0.16   &   1.57  \\ 
ManPQN&   56.32    & 1.432  &    0.80  &   0.0213   &   116.20   &   5.91  \\ 
\midrule
$n=128$ &   \\
\midrule
ManPG       &   1808.54  & 1.885  &    0.83  &   0.0985   &   86.98   &   0.53  \\ 
ManPG-Ada  &   801.16  & 1.885  &    0.83  &   0.0857   &   566.30   &   1.07  \\ 
NLS-ManPG   &   235.20   & 1.885  &    0.83  &   0.0277   &   10.12   &   1.47  \\ 
ManPQN&   22.52    & 1.890  &    0.81  &   0.0088   &   18.76   &   4.44  \\ 
\midrule
$n=256$ &   \\
\midrule
ManPG       &   4141.94  & 2.489  &    0.85  &   0.2162   &   72.14   &   0.45  \\ 
ManPG-Ada  &   1662.62  & 2.489  &    0.85  &   0.1388   &   1161.84   &   0.62  \\ 
NLS-ManPG   &   317.06   & 2.489 &    0.85  &   0.0354   &   19.06   &   1.33  \\ 
ManPQN&   17.60    & 2.497  &    0.84  &   0.0101   &   19.00   &   4.33  \\ 
\midrule
$n=512$ &   \\
\midrule
ManPG       &   9755.60  & 3.286  &    0.87  &   0.7385   &   50.66   &   0.16  \\ 
ManPG-Ada  &   3983.06  & 3.286  &    0.87  &   0.4208   &   2623.98   &   0.25  \\ 
NLS-ManPG   &   501.92   & 3.286  &    0.87  &   0.0704   &   25.06   &   0.92  \\ 
ManPQN&   16.54    & 3.293  &    0.86  &   0.0250   &   15.58   &   3.28  \\ 
\midrule
\end{tabular}
\end{table}

The total number of line-search steps and the averaged iteration number of the ASSN method are shown in Tables \ref{tb51_1}-\ref{tb51_3}.
ManPQN and NLS-ManPG need less line-search steps than the other methods since they use the nonmonotone line search technique.
From Tables \ref{tb51_1}-\ref{tb51_3}, we can see that the ASSN method in ManPQN needs more iterations than that of the other three methods.
The reason for this is that the matrix $({\rm diag}B_k)^{-1})$ is involved in $\mathcal{G}(\overline{\rm vec}(\Lambda_l))$ and therefore the condition number of $\mathcal{G}(\overline{\rm vec}(\Lambda_l))$ becomes larger (see \eqref{eq3_jac} and \eqref{biao}).

\begin{table}[htbp]\centering
\caption{Comparison on CM problem, different $r=\{1,2,4,6,8\}$ with $n= 128$ and $\mu=0.15$.}\label{tb51_2}
\begin{tabular}{ccccccc}
\midrule
$r=1$ & Iter & $F(X^*)$  & sparsity & CPU time & \# line-search &  SSN iters  \\
\midrule
ManPG       &   54.12  & 0.6513 &    0.87  &   0.0034   &   0.00   &   0.93  \\ 
ManPG-Ada  &   42.20  & 0.6513  &    0.87  &   0.0027   &   0.00   &   1.04  \\ 
NLS-ManPG   &   16.94   & 0.6513  &    0.87  &   0.0015   &   0.08   &   1.12  \\ 
ManPQN&   12.78    & 0.6603  &    0.86  &   0.0039   &   4.94   &   2.29  \\
\midrule
$r=2$ &   \\
\midrule
ManPG       &   298.62  & 1.302  &    0.86  &   0.0190   &   15.04   &   0.91  \\ 
ManPG-Ada  &   147.52  & 1.302  &    0.86  &   0.0120   &   94.92   &   1.06  \\ 
NLS-ManPG   &   49.04   & 1.302  &    0.86  &   0.0048   &   1.46   &   1.24  \\ 
ManPQN&   16.58    & 1.303  &    0.86  &   0.0047   &   5.24   &   2.71  \\ 
\midrule
$r=4$ &   \\
\midrule
ManPG       &   1156.72  & 2.605  &    0.86  &   0.0775   &   95.56   &   0.82  \\ 
ManPG-Ada  &   516.68  & 2.605  &    0.86  &   0.0659   &   371.62   &   1.25  \\ 
NLS-ManPG   &   167.60   & 2.605  &    0.86  &   0.0210   &   5.56   &   1.65  \\ 
ManPQN&   26.26    & 2.610  &    0.85  &   0.0104   &   16.86   &   4.33  \\
\midrule
$r=6$ &   \\
\midrule
ManPG       &   2287.98  & 3.909  &    0.85  &   0.2688   &   163.14   &   0.60  \\ 
ManPG-Ada  &   987.78  & 3.909  &    0.85  &   0.2330   &   709.94   &   1.57  \\ 
NLS-ManPG   &   329.22   & 3.909  &    0.85  &   0.0742   &   8.98   &   1.79  \\ 
ManPQN&   34.08    & 3.920  &    0.84  &   0.0197   &   29.38   &   4.32  \\ 
\midrule
$r=8$ &   \\
\midrule
ManPG       &   4727.60  & 5.214  &    0.84  &   1.1193   &   220.42   &   1.24  \\ 
ManPG-Ada  &   1905.44  & 5.214  &    0.84  &   0.8793   &   1265.02   &   2.82  \\ 
NLS-ManPG   &   692.14   & 5.214  &    0.84  &   0.2878   &   15.86   &   2.87  \\ 
ManPQN&   56.56    & 5.248  &    0.82  &   0.0437   &   51.76   &   5.96  \\ 
\midrule
\end{tabular}
\end{table}

\begin{table}[htbp]\centering
\caption{Comparison on CM problem, different $\mu=\{0.05,0.10,0.15,0.20,0.25\}$ with $n= 128$ and $r=4$.}\label{tb51_3}
\begin{tabular}{ccccccc}
\midrule
$\mu=0.05$ & Iter & $F(X^*)$  & sparsity & CPU time & \# line-search &  SSN iters  \\
\midrule
ManPG       &   2751.68  & 1.083  &    0.76  &   0.1511   &   57.08   &   0.28  \\ 
ManPG-Ada  &   1173.30  & 1.083  &    0.76  &   0.1413   &   710.14   &   0.85  \\ 
NLS-ManPG   &   276.12   & 1.083  &    0.76  &   0.0383   &   14.76   &   1.37  \\ 
ManPQN&   53.78    & 1.111  &    0.73  &   0.0197   &   92.40   &   5.35  \\ 
\midrule
$\mu=0.10$ &   \\
\midrule
ManPG       &   2041.74  & 1.885  &    0.83  &   0.1221   &   79.06   &   0.52  \\ 
ManPG-Ada  &   909.02  & 1.885  &    0.83  &   0.1103   &   703.42   &   1.10  \\ 
NLS-ManPG   &   259.48   & 1.885  &    0.83  &   0.0332   &   11.30   &   1.47  \\ 
ManPQN&   42.04    & 1.902  &    0.81  &   0.0147   &   37.10   &   4.27  \\ 
\midrule
$\mu=0.15$ &   \\
\midrule
ManPG       &   1312.30  & 2.605  &    0.86  &   0.0869   &   130.72   &   0.77  \\ 
ManPG-Ada  &   580.08  & 2.605  &    0.86  &   0.0687   &   424.32   &   1.28  \\ 
NLS-ManPG   &   185.76   & 2.605  &    0.86  &   0.0233   &   5.74   &   1.63  \\ 
ManPQN&   27.54    & 2.610  &    0.85  &   0.0097   &   18.36   &   4.15  \\
\midrule
$\mu=0.20$ &   \\
\midrule
ManPG       &   1087.28  & 3.278  &    0.88  &   0.0699   &   90.14   &   0.88  \\ 
ManPG-Ada  &   481.32  & 3.278  &    0.88  &   0.0587   &   365.34   &   1.35  \\ 
NLS-ManPG   &   164.26   & 3.278  &    0.88  &   0.0201   &   3.88   &   1.80  \\ 
ManPQN&   28.12    & 3.280  &    0.87  &   0.0106   &   23.40   &   4.91  \\ 
\midrule
$\mu=0.25$ &   \\
\midrule
ManPG       &   773.86  & 3.916  &    0.89  &   0.0638   &   86.08   &   0.55  \\ 
ManPG-Ada  &   352.60  & 3.916  &    0.89  &   0.0420   &   211.18   &   1.00  \\ 
NLS-ManPG   &   126.10   & 3.916  &    0.89  &   0.0155   &   1.70   &   1.22  \\ 
ManPQN&   36.24    & 3.920  &    0.88  &   0.0135   &   21.82   &   3.34  \\ 
\midrule
\end{tabular}
\end{table}

\subsection{Sparse PCA}\label{sec52}

The sparse PCA model can be formulated as
\begin{equation}\label{eq5_2}
\min_{X\in\mathcal{M}}-\mathrm{tr}(X^TA^TAX)+\mu\|X\|_1,
\end{equation}
where $A\in\mathbb{R}^{m\times n}$.
We apply ManPQN and ManPG related algorithms to \eqref{eq5_2} and compare their performance.
The matrix $A^TA$ is generated by the following two ways.
In Section \ref{sec52_1}, $A$ is generated by normal distribution.
In Section \ref{sec52_2}, the matrix $A$ is chosen from real matrices in ``UF Sparse Matrix Collection'' \cite{ufmat_web}.

\subsubsection{Random generated sparse PCA problem}\label{sec52_1}

From Figures \ref{fig52_1}-\ref{fig52_3} and Tables \ref{tb52_1}-\ref{tb52_3}, 
we can see that ManPQN shows better performance than ManPG, ManPG-Ada and NLS-ManPG in terms of CPU time and iteration number, especially when $n$ and $r$ are large. 
In some cases, ManPG related methods can achieve a solution, the sparsity of which is a little better than that of ManPQN.
Taking CPU time, the total number of line search steps and the averaged SSN iteration number in Tables \ref{tb52_1}-\ref{tb52_3} into account, we can deduce that the quasi-Newton strategy can have a significant effect on accelerating our method.

Next, we investigate the convergence of $F(X_k)$ generated by ManPQN, ManPG and ManPG-Ada. We select six cases with different $n$ and $r$, and plot numerical results of these three algorithms in Figures \ref{fig52_4}-\ref{fig52_6}. We can see that when $X_k$ is close to $X^*$, $F(X_k)$ converges to $F(X^*)$ approximately at a linear rate, which matches our theoretical results in Theorem \ref{thm42}.

\begin{figure}[H]
	\centering  
	\subfigbottomskip=2pt 
	\subfigcapskip=-5pt 
	\subfigure[CPU]{
		\includegraphics[width=0.33\linewidth]{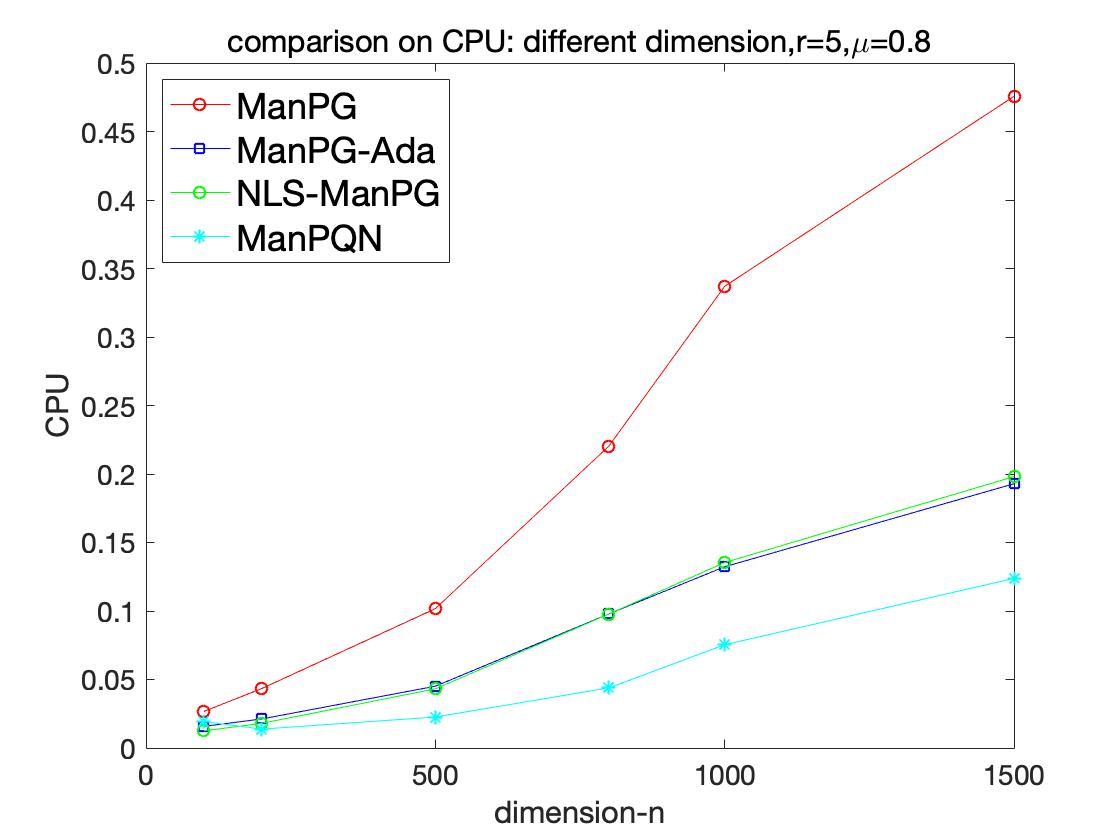}}
		\hspace{-5mm}
	\subfigure[Iter]{
		\includegraphics[width=0.33\linewidth]{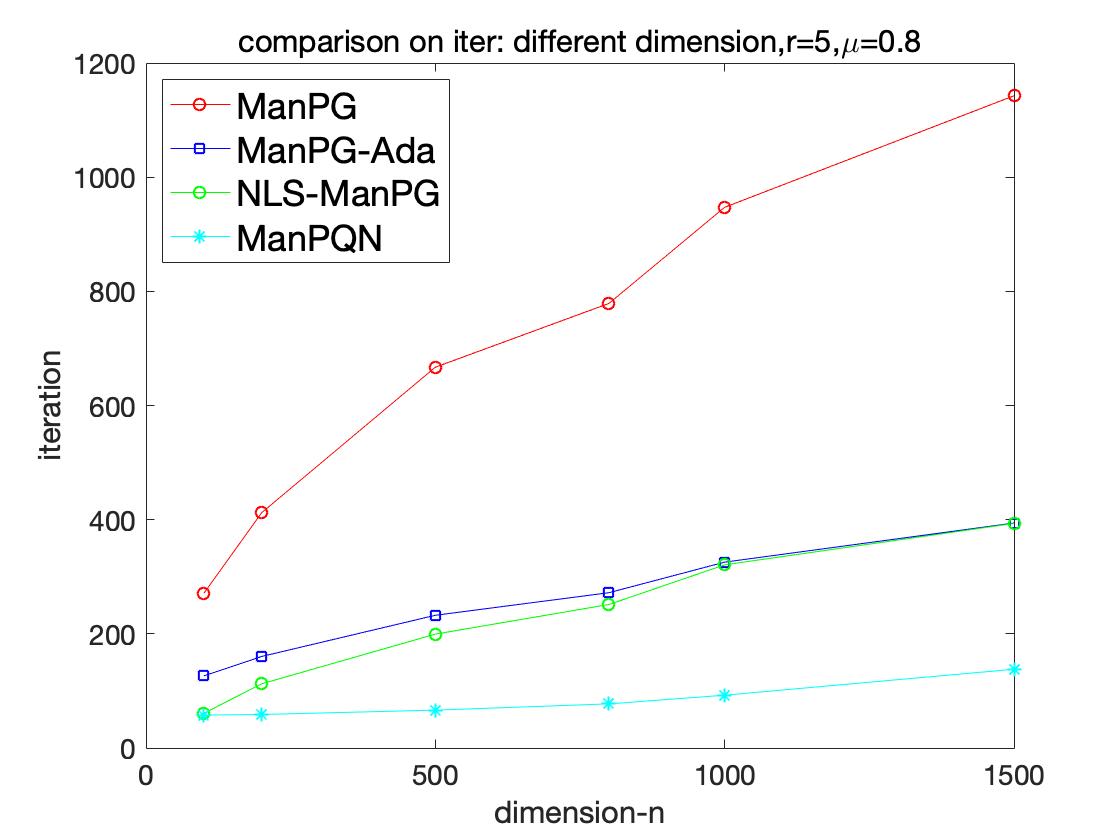}}
		\hspace{-5mm}
	\subfigure[Sparsity]{
		\includegraphics[width=0.33\linewidth]{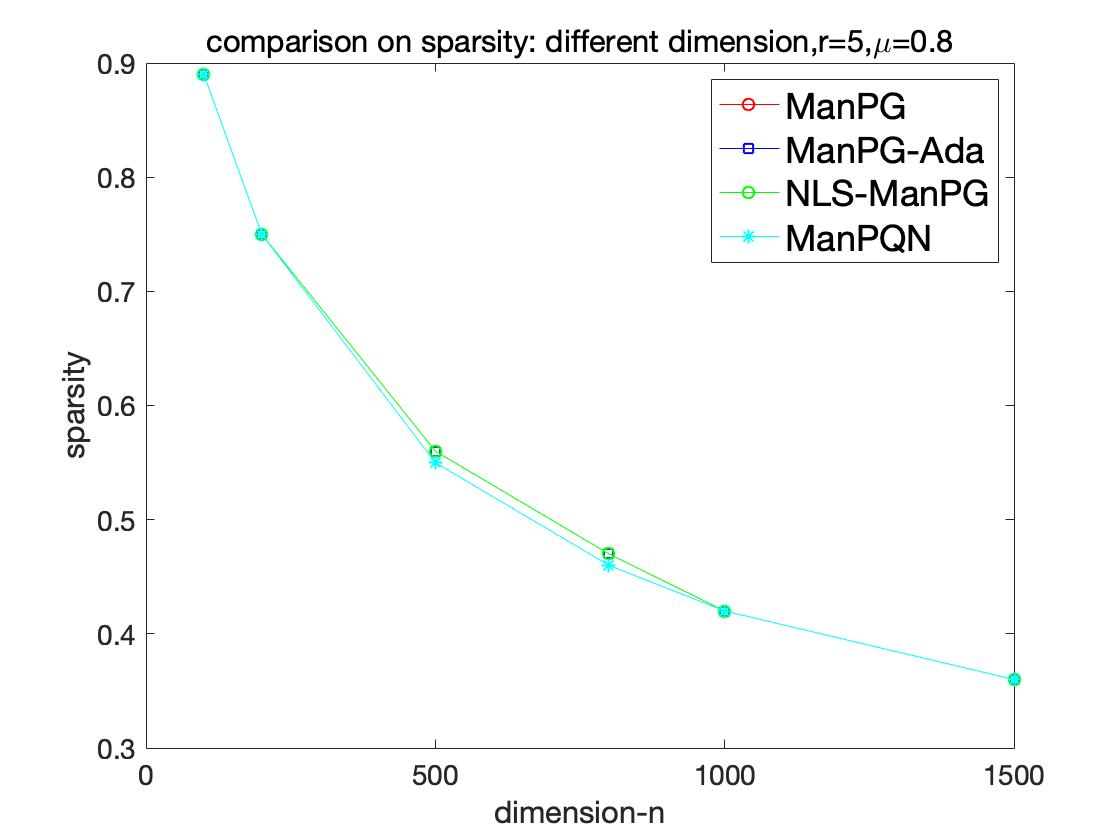}}
	\caption{Comparison on Sparse PCA problem, different $n=\{100,200, 500, 800, 1000, 1500\}$ with $r= 5$ and $\mu=0.8$.}\label{fig52_1}
\end{figure}

\begin{figure}[H]
	\centering  
	\subfigbottomskip=2pt 
	\subfigcapskip=-5pt 
	\subfigure[CPU]{
		\includegraphics[width=0.34\linewidth]{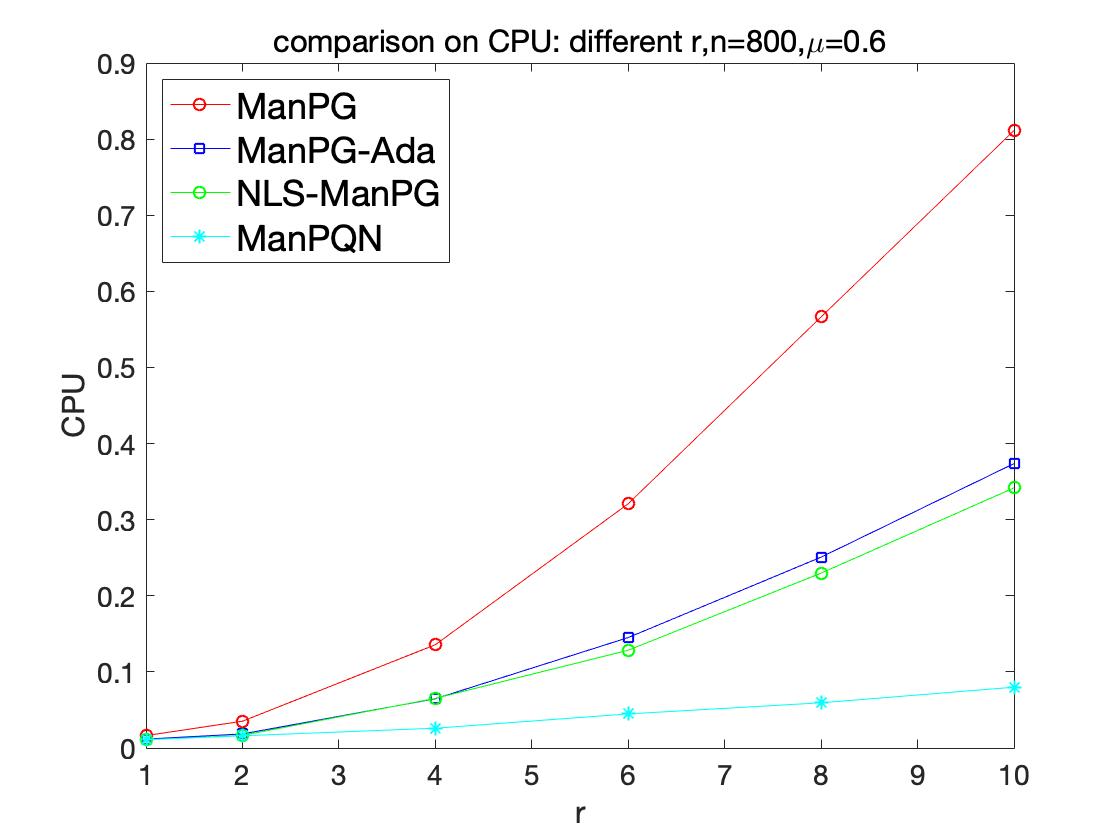}}
		\hspace{-5mm}
	\subfigure[Iter]{
		\includegraphics[width=0.34\linewidth]{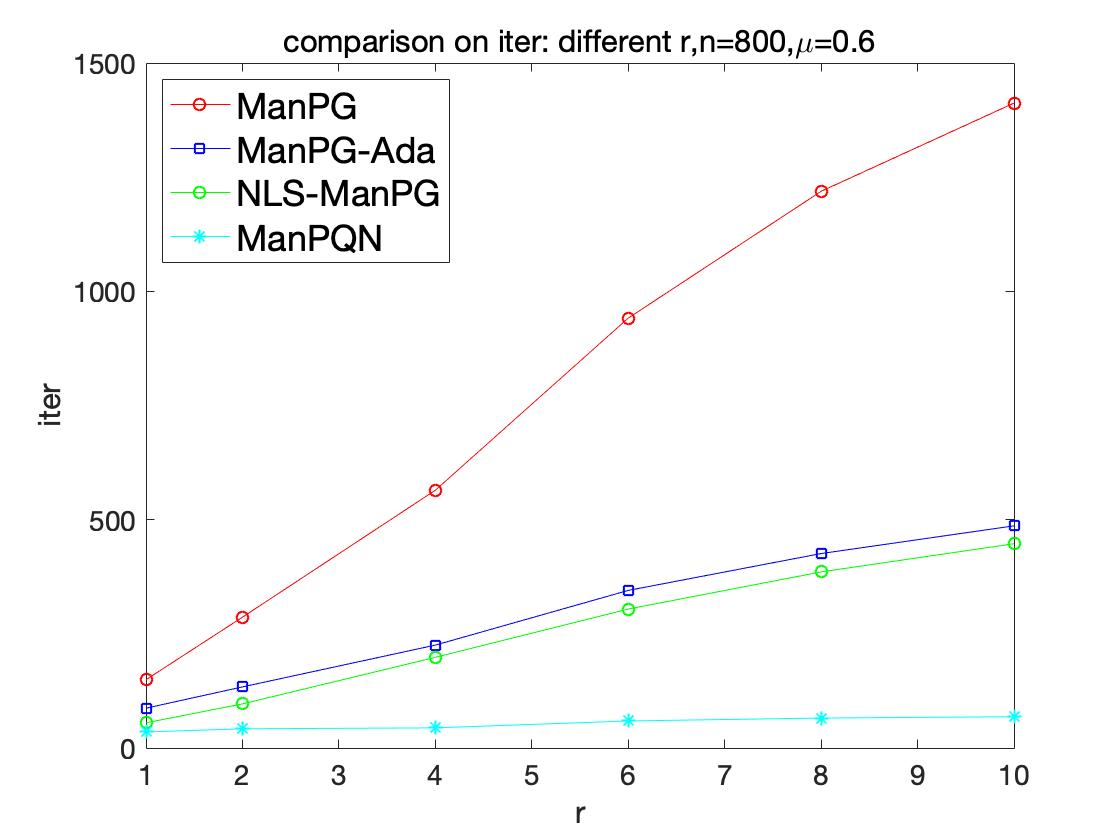}}
		\hspace{-5mm}
	\subfigure[Sparsity]{
		\includegraphics[width=0.34\linewidth]{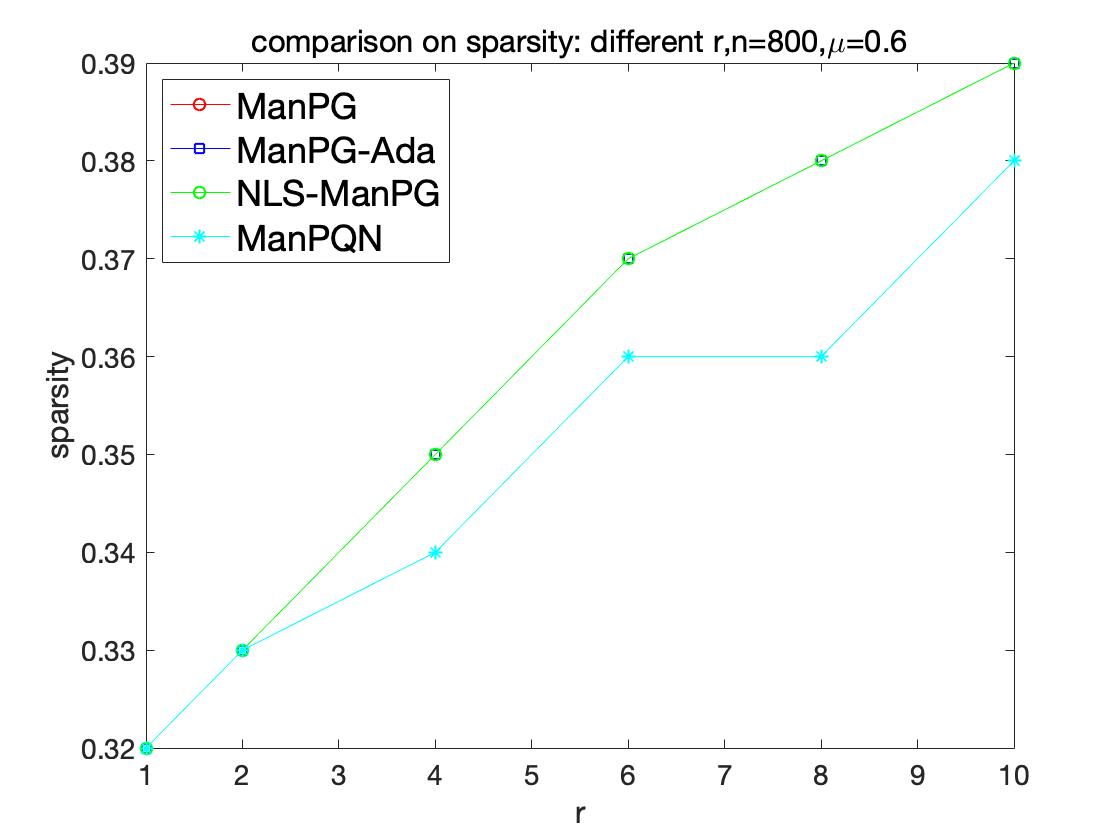}}
	\caption{Comparison on Sparse PCA problem, different $r=\{1,2,4,8,10\}$ with $n= 800$ and $\mu=0.6$.}\label{fig52_2}
\end{figure}

\begin{figure}[H]
	\centering  
	\subfigbottomskip=2pt 
	\subfigcapskip=-5pt 
	\subfigure[CPU]{
		\includegraphics[width=0.34\linewidth]{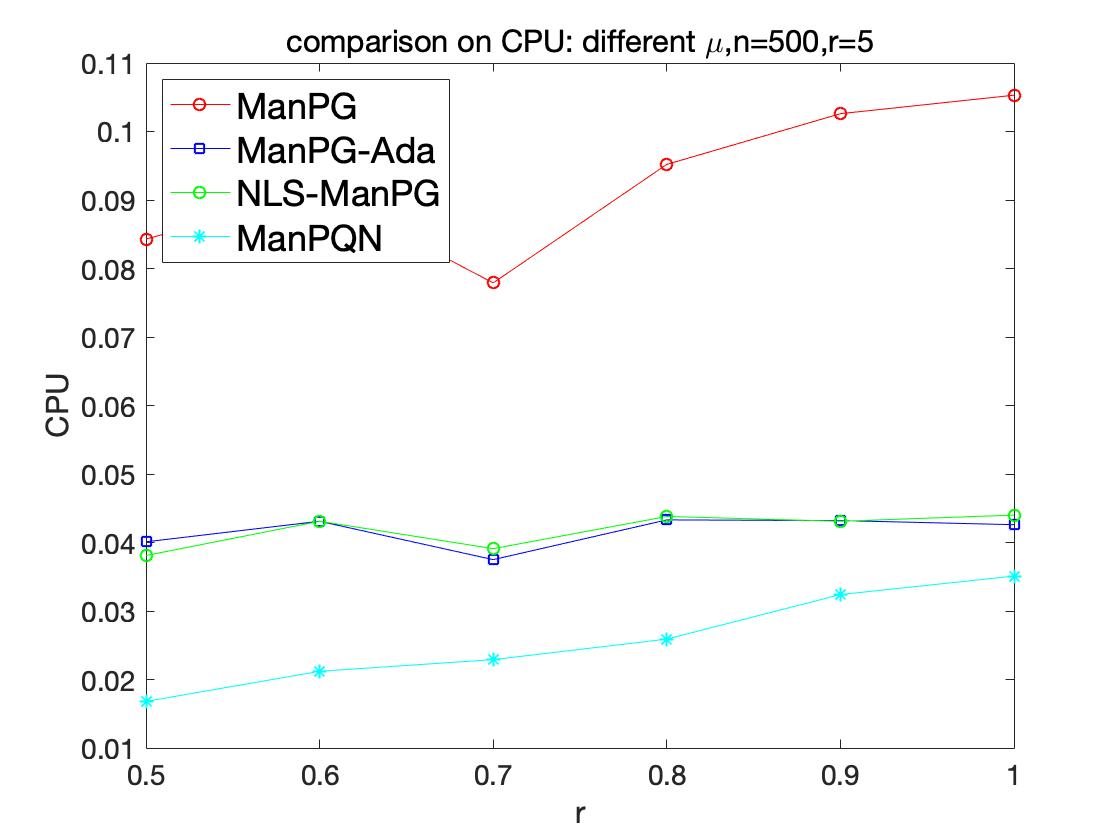}}
		\hspace{-5mm}
	\subfigure[Iter]{
		\includegraphics[width=0.34\linewidth]{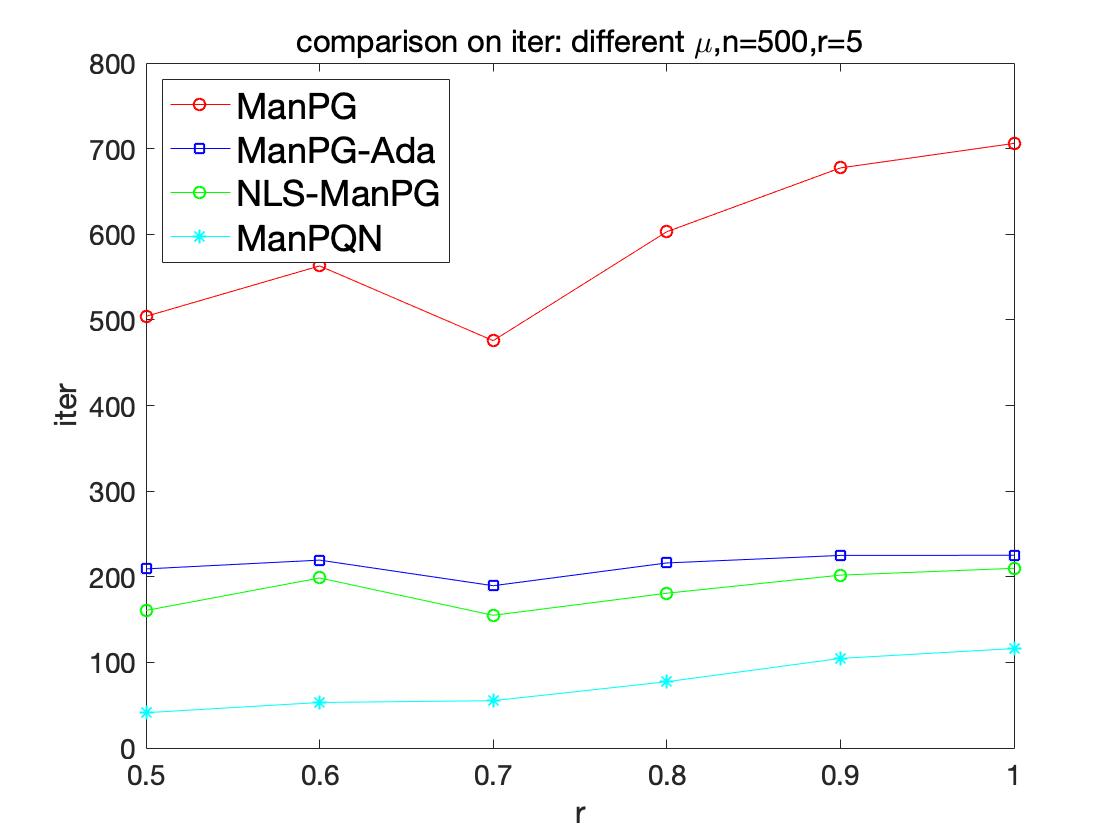}}
		\hspace{-5mm}
	\subfigure[Sparsity]{
		\includegraphics[width=0.34\linewidth]{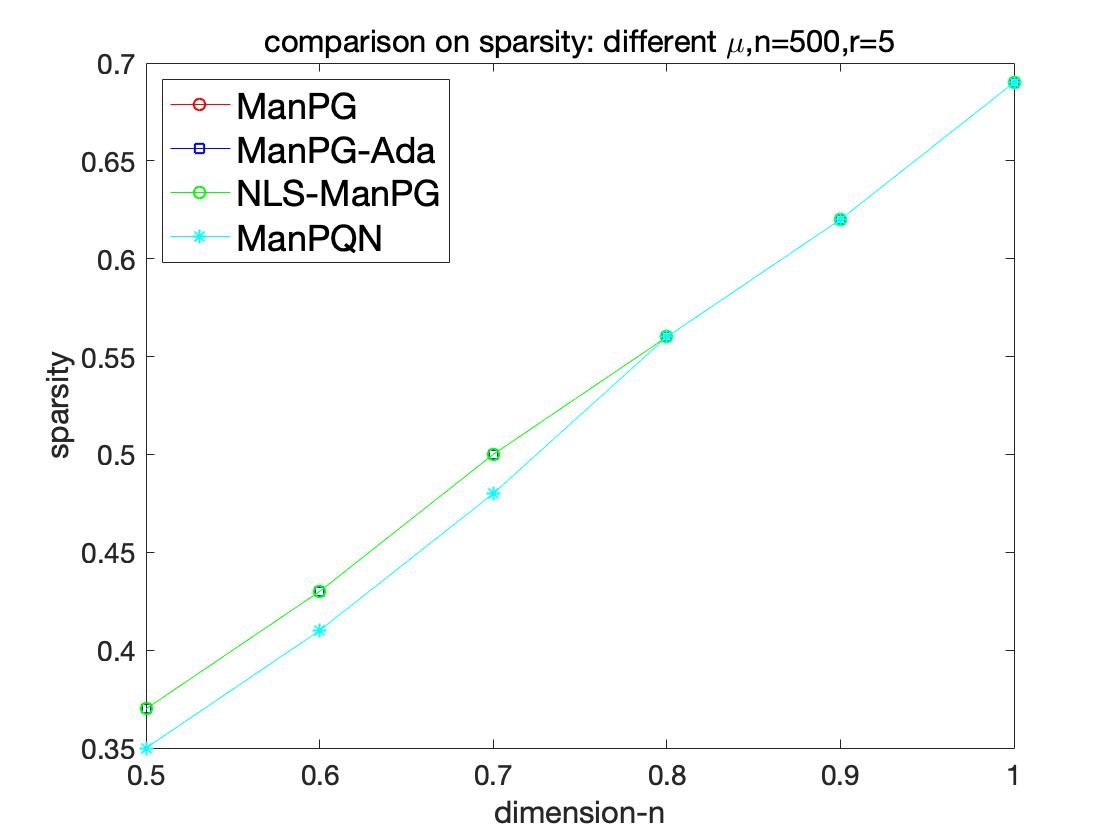}}
	\caption{Comparison on Sparse PCA problem, different $\mu=\{0.5,0.6,0.7, 0.8,0.9,1.0\}$ with $n= 500$ and $r=5$.}\label{fig52_3}
\end{figure}

\begin{figure}[H]
	\centering  
	\subfigbottomskip=2pt 
	\subfigcapskip=-5pt 
	\subfigure[$r=4$]{
		\includegraphics[width=0.5\linewidth]{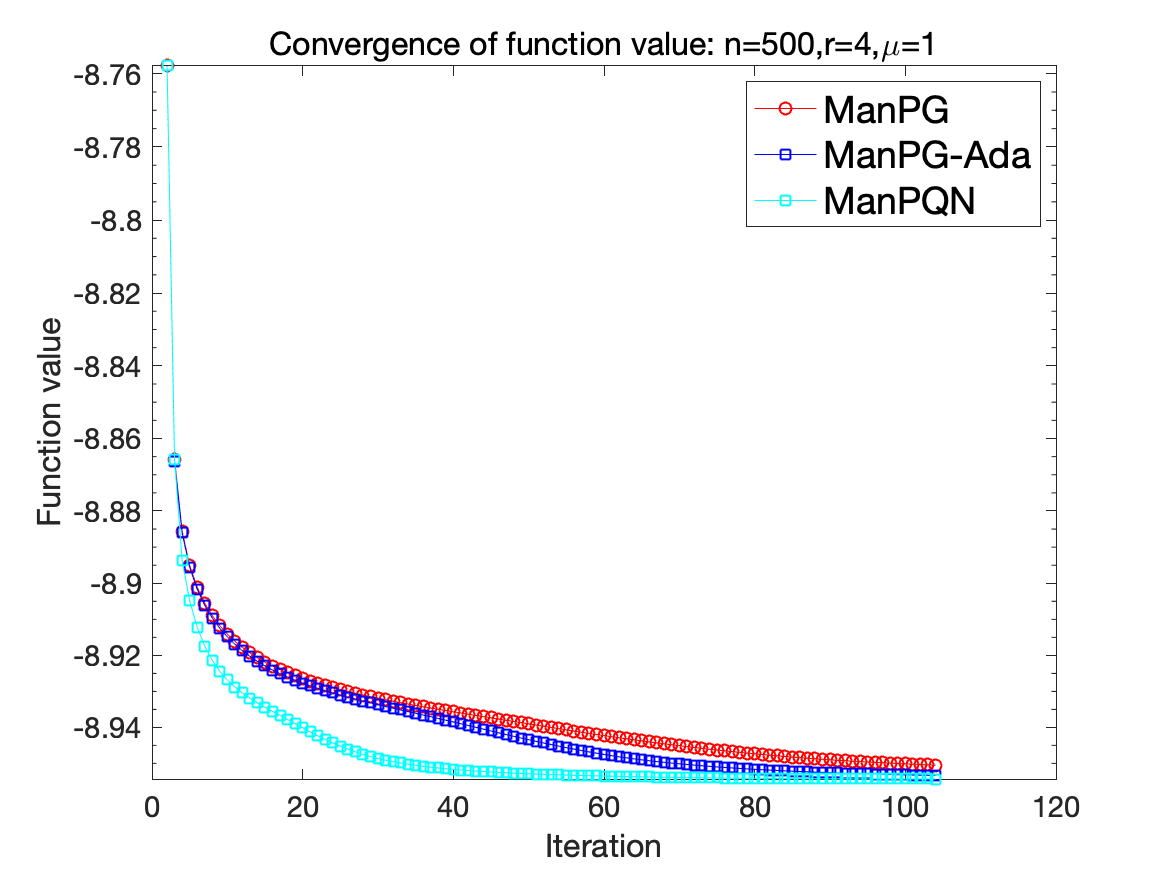}}
		\hspace{-5mm}
	\subfigure[$r=10$]{
		\includegraphics[width=0.5\linewidth]{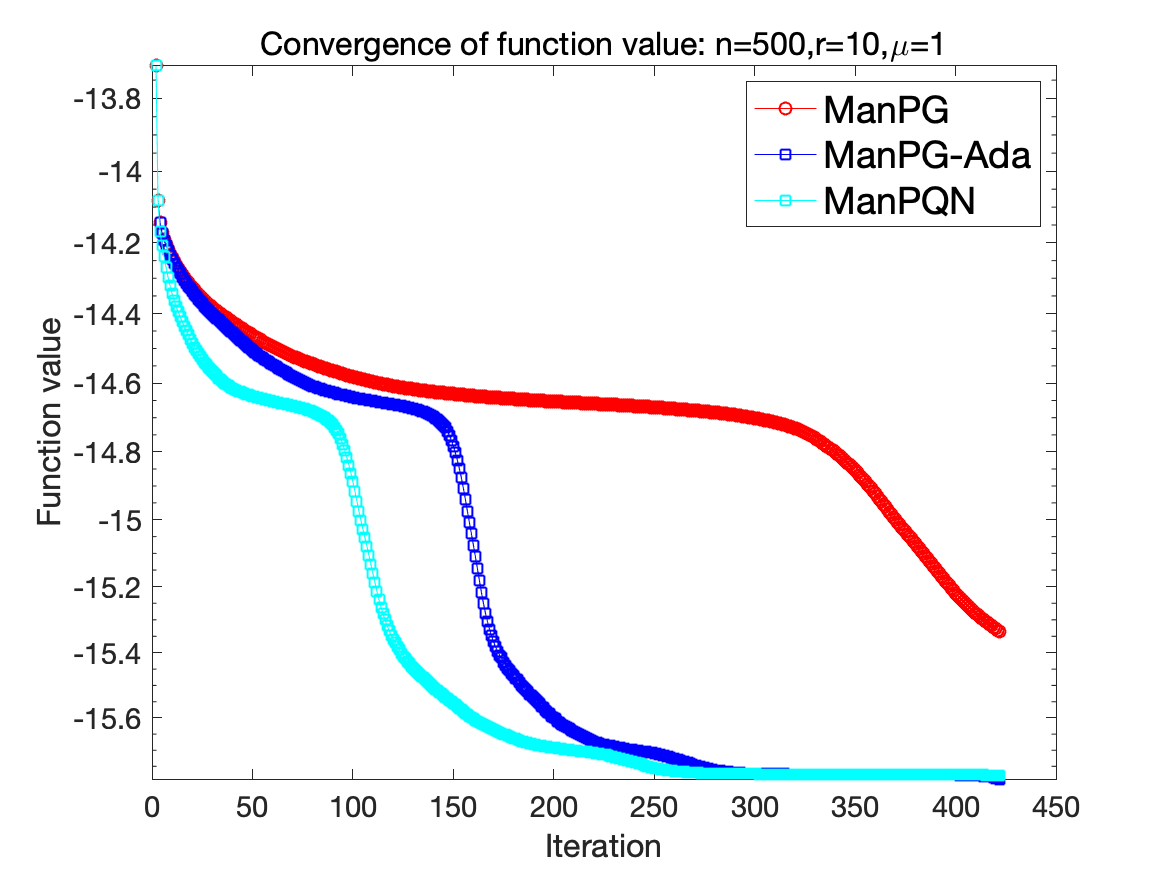}}
	\caption{Convergence of function value on Sparse PCA problem with $n=500, \mu=1.0$.}\label{fig52_4}
\end{figure} 

\begin{figure}[H]
	\centering  
	\subfigbottomskip=2pt 
	\subfigcapskip=-5pt 
	\subfigure[$r=4$]{
		\includegraphics[width=0.5\linewidth]{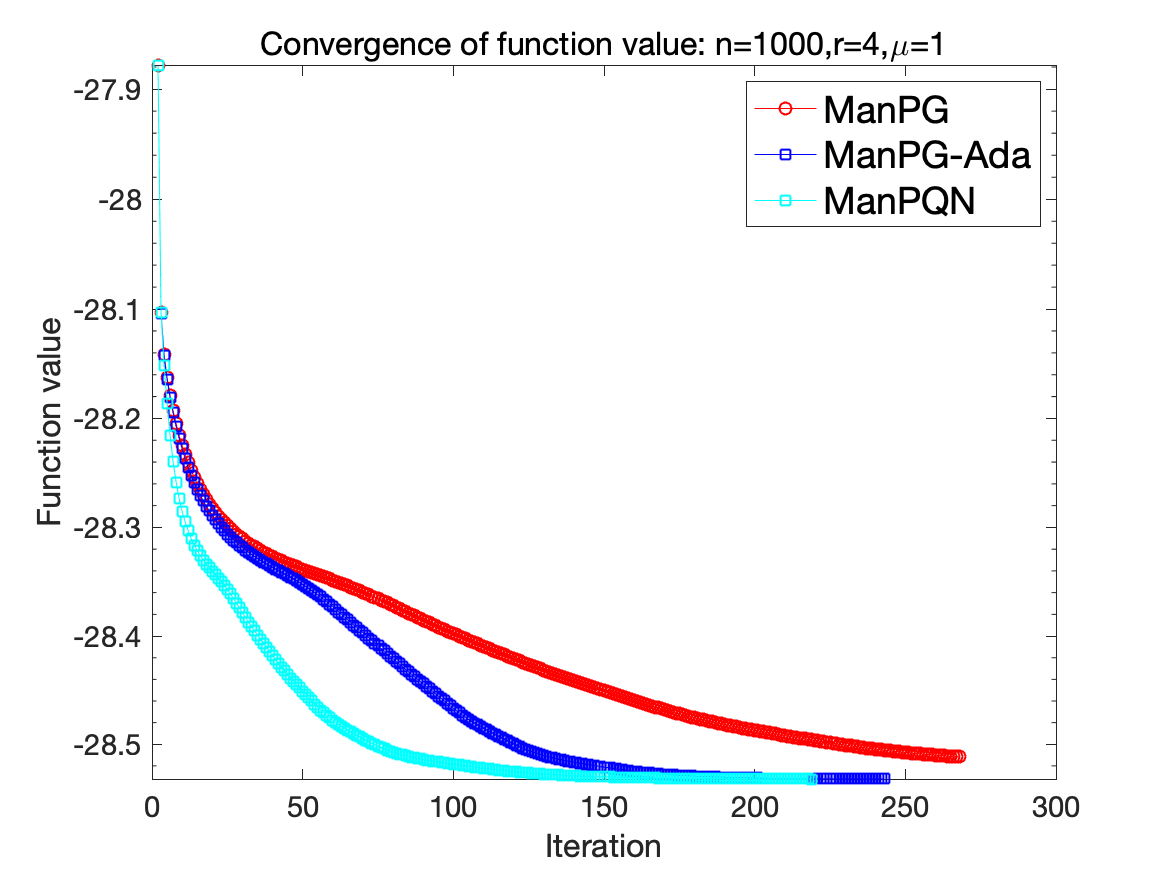}}
		\hspace{-5mm}
	\subfigure[$r=10$]{
		\includegraphics[width=0.5\linewidth]{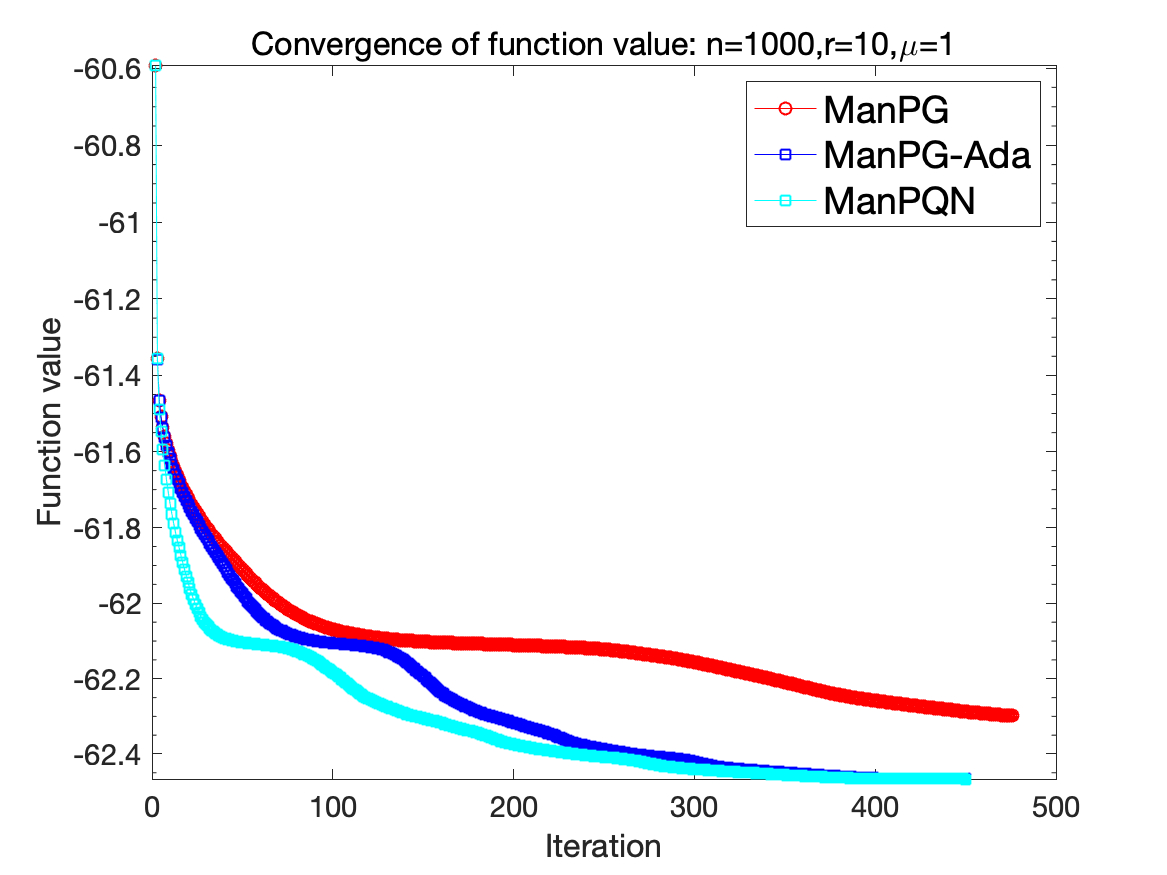}}	
	\caption{Convergence of function value on Sparse PCA problem with $n=1000,\mu=1.0$.}\label{fig52_5}
\end{figure}

\begin{figure}[H]
	\centering  
	\subfigbottomskip=2pt 
	\subfigcapskip=-5pt 
	\subfigure[$r=4$]{
		\includegraphics[width=0.5\linewidth]{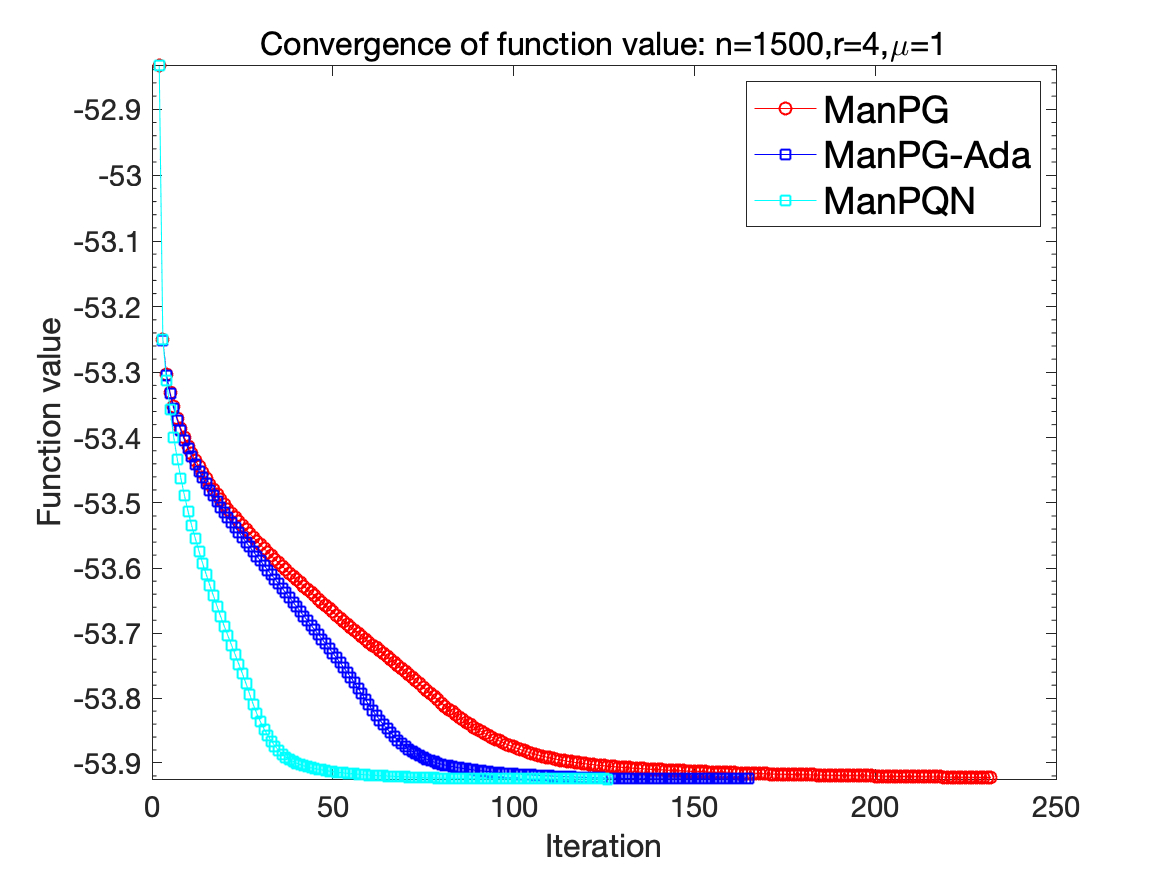}}
		\hspace{-5mm}
	\subfigure[$r=10$]{
		\includegraphics[width=0.5\linewidth]{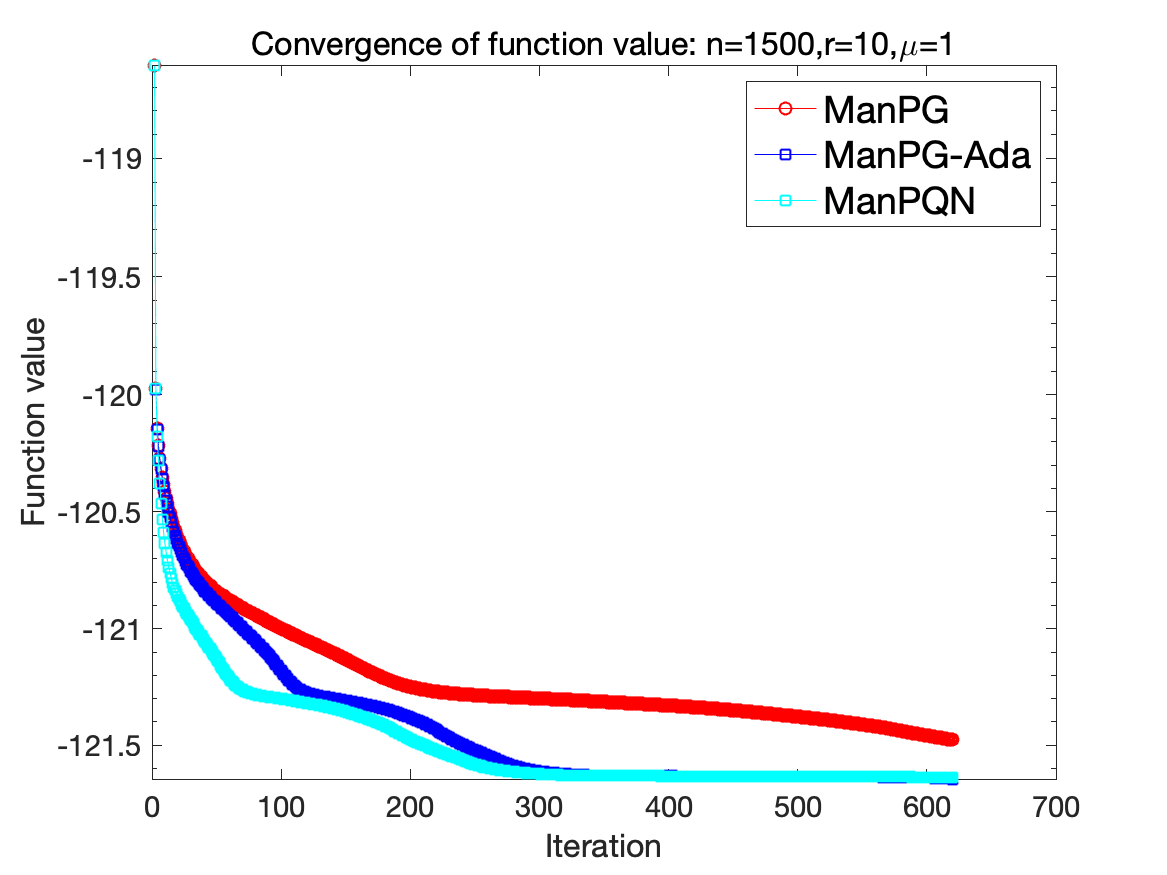}}	
	\caption{Convergence of function value on Sparse PCA problem with $n=1500,\mu=1.0$.}\label{fig52_6}
\end{figure}

\begin{table}[htbp]\centering
\caption{Comparison on Sparse PCA problem, different $n=\{100,200,500,800,1000,1500\}$ with $r= 5$ and $\mu=0.8$.}\label{tb52_1}
\begin{tabular}{ccccccc}
\midrule
$n=100$ & Iter & $F(X^*)$  & sparsity & CPU time & \# line-search &  SSN iters  \\
\midrule
ManPG       &   271.02  & -2.285  &    0.89  &   0.0267   &   1.86   &   1.20  \\ 
ManPG-Ada  &   126.40  & -2.285  &    0.89  &   0.0158   &   0.38   &   1.49  \\ 
NLS-ManPG   &   60.96  & -2.285  &    0.89  &   0.0125   &   0.00   &   2.34  \\ 
ManPQN     &   57.64  & -2.274  &    0.89  &   0.0193   &   23.94   &   2.68  \\ 
\midrule
$n=200$ &   \\
\midrule
ManPG       &   412.32  & -5.449  &    0.75  &   0.0434   &   1.94   &   1.07  \\ 
ManPG-Ada  &   160.38  & -5.449  &    0.75  &   0.0212   &   7.22   &   1.33  \\ 
NLS-ManPG   &   112.48  & -5.449  &    0.75  &   0.0180   &   0.00   &   1.50  \\ 
ManPQN     &   58.72  & -5.428  &    0.75  &   0.0138   &   37.92   &   2.03  \\ 
\midrule
$n=500$ &   \\
\midrule
ManPG       &   667.08  & -20.30  &    0.56  &   0.1017   &   0.00   &   1.02  \\ 
ManPG-Ada  &   232.50  & -20.30  &    0.56  &   0.0451   &   48.86   &   1.22  \\ 
NLS-ManPG   &   199.50  & -20.30  &    0.56  &   0.0432   &   0.00   &   1.38  \\ 
ManPQN     &   66.34  & -20.21  &    0.55  &   0.0226   &   41.76   &   1.96  \\ 
\midrule
$n=800$ &   \\
\midrule
ManPG       &   778.82  & -39.16  &    0.47  &   0.2204   &   0.00   &   1.02  \\ 
ManPG-Ada  &   272.06  & -39.16  &    0.47  &   0.0979   &   70.66   &   1.19  \\ 
NLS-ManPG   &   251.44  & -39.16  &    0.47  &   0.0978   &   0.00   &   1.30  \\ 
ManPQN     &   77.32  & -39.07  &    0.46  &   0.0439   &   50.14   &   1.95  \\ 
\midrule
$n=1000$ &   \\
\midrule
ManPG       &   947.26  & -53.23  &    0.42  &   0.3370   &   0.00   &   1.01  \\ 
ManPG-Ada  &   325.62  & -53.23  &    0.42  &   0.1323   &   102.98   &   1.17  \\ 
NLS-ManPG   &   320.90  & -53.23  &    0.42  &   0.1354   &   0.00   &   1.24  \\ 
ManPQN     &   92.50  & -53.17  &    0.42  &   0.0754   &   68.24   &   1.92  \\ 
\midrule
$n=1500$ &   \\
\midrule
ManPG       &   1142.60  & -89.60  &    0.36  &   0.4757   &   19.04   &   1.01  \\ 
ManPG-Ada  &   394.10  & -89.60  &    0.36  &   0.1929   &   150.80   &   1.12  \\ 
NLS-ManPG   &   393.40  & -89.60  &    0.36  &   0.1981   &   0.00   &   1.20  \\ 
ManPQN     &   137.94  & -89.50  &    0.36  &   0.1238   &   171.00   &   1.87  \\ 
\midrule
\end{tabular}
\end{table}

\begin{table}[htbp]\centering
\caption{Comparison on Sparse PCA problem, different $r=\{1,2,4,8,10\}$ with $n= 800$ and $\mu=0.6$.}\label{tb52_2}
\begin{tabular}{ccccccc}
\midrule
$r=1$ & Iter & $F(X^*)$  & sparsity & CPU time & \# line-search &  SSN iters  \\
\midrule
ManPG       &   149.56  & -12.21  &    0.32  &   0.0163   &   0.00   &   0.78  \\ 
ManPG-Ada  &   87.18  & -12.21  &    0.32  &   0.0117   &   2.00   &   0.93  \\ 
NLS-ManPG   &   55.10  & -12.21  &    0.33  &   0.0109   &   0.00   &   1.03  \\ 
ManPQN     &   35.16  & -12.21  &    0.32  &   0.0134   &   34.04   &   1.26  \\ 
\midrule
$r=2$ &   \\
\midrule
ManPG       &   286.48  & -24.01  &    0.33  &   0.0352   &   93.48   &   0.99  \\ 
ManPG-Ada  &   133.80  & -24.01  &    0.33  &   0.0185   &   14.92   &   1.06  \\ 
NLS-ManPG   &   96.58  & -24.01  &    0.33  &   0.0166   &   0.02   &   1.11  \\ 
ManPQN     &   42.04  & -23.98  &    0.33  &   0.0160   &   39.32   &   1.63  \\ 
\midrule
$r=4$ &   \\
\midrule
ManPG       &   564.46  & -45.83  &    0.35  &   0.1357   &   4.52   &   1.02  \\ 
ManPG-Ada  &   225.32  & -45.83  &    0.35  &   0.0647   &   56.36   &   1.07  \\ 
NLS-ManPG   &   198.54  & -45.83  &    0.35  &   0.0652   &   0.00   &   1.16  \\ 
ManPQN    &   44.10  & -45.77  &    0.34  &   0.0260   &   31.94   &   1.83  \\ 
\midrule
$r=6$ &   \\
\midrule
ManPG       &   941.32  & -66.18  &    0.37  &   0.3209   &   0.00   &   1.01  \\ 
ManPG-Ada  &   345.08  & -66.18  &    0.37  &   0.1451   &   120.30   &   1.19  \\ 
NLS-ManPG   &   304.30  & -66.18  &    0.37  &   0.1285   &   0.00   &   1.28  \\ 
ManPQN     &   59.18  & -66.03  &    0.36  &   0.0448   &   39.76   &   1.96  \\ 
\midrule
$r=8$ &   \\
\midrule
ManPG       &   1219.04  & -85.57  &    0.38  &   0.5664   &   0.00   &   1.01  \\ 
ManPG-Ada  &   425.80  & -85.57  &    0.38  &   0.2508   &   171.44   &   1.32  \\ 
NLS-ManPG   &   385.58  & -85.57  &    0.38  &   0.2298   &   0.00   &   1.37  \\ 
ManPQN     &   65.56  & -85.39  &    0.36  &   0.0594   &   40.82   &   2.01  \\ 
\midrule
$r=10$ &   \\
\midrule
ManPG       &   1412.26  & -103.59  &    0.39  &   0.8113   &   0.00   &   1.02  \\ 
ManPG-Ada  &   486.46  & -103.59  &    0.39  &   0.3737   &   206.70   &   1.44  \\ 
NLS-ManPG   &   447.34  & -103.59  &    0.39  &   0.3418   &   0.00   &   1.41  \\ 
ManPQN     &   68.52  & -103.40  &    0.38  &   0.0799   &   42.38   &   2.10  \\ 
\midrule
\end{tabular}
\end{table}

\begin{table}[htbp]\centering
\caption{Comparison on Sparse PCA problem, different $\mu=\{0.5,0.6,0.7, 0.8,0.9,1.0\}$ with $n= 500$ and $r=5$.}\label{tb52_3}
\begin{tabular}{ccccccc}
\midrule
$\mu=0.5$ & Iter & $F(X^*)$  & sparsity & CPU time & \# line-search &  SSN iters  \\
\midrule
ManPG       &   504.12  & -39.41  &    0.37  &   0.0843   &   0.00   &   1.02  \\ 
ManPG-Ada  &   209.14  & -39.41  &    0.37  &   0.0401   &   43.44   &   1.13  \\ 
NLS-ManPG   &   160.62  & -39.41  &    0.37  &   0.0381   &   0.00   &   1.27  \\ 
ManPQN     &   41.24  & -39.34  &    0.35  &   0.0168   &   26.60   &   1.87  \\ 
\midrule
$\mu=0.6$ & \\
\midrule
ManPG       &   563.02  & -32.69  &    0.43  &   0.0912   &   0.00   &   1.02  \\ 
ManPG-Ada  &   219.42  & -32.69  &    0.43  &   0.0431   &   49.30   &   1.15  \\ 
NLS-ManPG   &   198.58  & -32.69  &    0.43  &   0.0431   &   0.00   &   1.25  \\ 
ManPQN      &   53.16  & -32.63  &    0.41  &   0.0212   &   33.72   &   1.90  \\ 
\midrule
$\mu=0.7$ & \\
\midrule
ManPG       &   475.74  & -26.30  &    0.50  &   0.0779   &   0.00   &   1.02  \\ 
ManPG-Ada  &   189.52  & -26.30  &    0.50  &   0.0375   &   25.36   &   1.14  \\ 
NLS-ManPG   &   154.96  & -26.30  &    0.50  &   0.0391   &   0.00   &   1.28  \\ 
ManPQN      &   55.26  & -26.26  &    0.48  &   0.0229   &   34.32   &   1.96  \\ 
\midrule
$\mu=0.8$ & \\
\midrule
ManPG       &   602.94  & -20.36  &    0.56  &   0.0952   &   1.40   &   1.02  \\ 
ManPG-Ada  &   216.12  & -20.36  &    0.56  &   0.0433   &   39.86   &   1.24  \\ 
NLS-ManPG   &   180.76  & --20.36  &    0.56  &   0.0438   &   0.00   &   1.40  \\ 
ManPQN      &   77.40  & -20.32  &    0.56  &   0.0259   &   53.34   &   1.96  \\ 
\midrule
$\mu=0.9$ & \\
\midrule
ManPG       &   677.48  & -14.63  &    0.62  &   0.1026   &   0.14   &   1.03  \\ 
ManPG-Ada  &   224.90  & -14.63  &    0.62  &   0.0432   &   38.54   &   1.27  \\ 
NLS-ManPG   &   201.88  & -14.63  &    0.62  &   0.0431   &   0.00   &   1.41  \\ 
ManPQN      &   104.62  & -14.59  &    0.62  &   0.0324   &   69.30   &   1.96  \\ 
\midrule
$\mu=1.0$ & \\
\midrule
ManPG       &   706.12  & -9.525  &    0.69  &   0.1053   &   0.28   &   1.04  \\ 
ManPG-Ada  &   225.10  & -9.525  &    0.69  &   0.0426   &   32.26   &   1.30  \\ 
NLS-ManPG   &   209.84  & -9.525  &    0.69  &   0.0440   &   0.02   &   1.39  \\ 
ManPQN      &   116.16  & -9.485 &    0.69  &   0.0351   &   109.80   &   1.93  \\ 
\midrule
\end{tabular}
\end{table}

\subsubsection{Sparse PCA problem from ``UF Sparse Matrix Collection''}\label{sec52_2}

To further evaluate the performance of ManPQN and ManPG related algorithms, we apply them to \eqref{eq5_2}, where $A\in\mathbb{R}^{m\times n}$ is chosen from real matrices in \cite{ufmat_web}. 
We choose four matrices named ``lpi$\_$klein1'', ``bcsstk22'', ``lp$\_$fit1d'' and ``fidap003'' 
whose dimensions are $(m,n)=(54,108)$, $(m,n)=(138,138)$, $(m,n)=(24,1049)$ and $(m,n)=(1821,1821)$ respectively. 
We set $r=4$, $\mu=0.2$.
Experiments are repeated for 50 times with different random initial point. 
The numerical results are given in Table \ref{tb52_4}.

Since $A^TA$ generated by these four matrices are ill-conditioned, all algorithms need more iteration numbers and CPU time to converge. 
In Table \ref{tb52_4}, we can observe that ManPQN and NLS-ManPG show much better performance than ManPG and ManPG-Ada in terms of iteration numbers and CPU time. 
Only for matrix ``lpi$\_$klein1'', ManPQN needs slightly more CPU time than NLS-ManPG.
For the other three matrices, ManPQN outperforms NLS-ManPG.

\begin{table}[htbp]\centering
\caption{Comparison on Sparse PCA problem with different matrices from \cite{ufmat_web}.}\label{tb52_4}
\begin{tabular}{ccccccc}
\midrule
``lpi$\_$klein1'' & Iter & $F(X^*)$  & sparsity & CPU time & \# line-search &  SSN iters  \\
\midrule
ManPG       &   2252.00  & -38.77  &    0.94  &   0.1684   &   0.00   &   1.01  \\ 
ManPG-Ada  &   919.00  & -38.77  &    0.94  &   0.0840   &   533.00   &   1.10  \\ 
NLS-ManPG   &   223.00  & -38.77  &    0.94  &   0.0305   &   170.00   &   1.74  \\ 
ManPQN      &   137.00  & -38.76  &    0.93  &   0.0318   &   58.00   &   1.27  \\ 
\midrule
``bcsstk20'' &   \\
\midrule
ManPG       &   11276.00  & -4823.74  &    0.82  &   0.8836   &   157.06   &   1.00  \\ 
ManPG-Ada  &   4685.58  & -4823.74  &    0.82  &   0.4489   &   2985.08   &   1.11  \\ 
NLS-ManPG   &   830.22  & -4823.74  &    0.82  &   0.1073   &   833.80   &   1.59  \\ 
ManPQN      &   137.18  & -4823.71  &    0.80  &   0.0322   &   0.00   &   1.00  \\ 
\midrule
``lp$\_$fit1d'' & \\
\midrule
ManPG       &   9596.26  & -2480.94  &    0.08  &   2.9966   &   0.00   &   1.45  \\ 
ManPG-Ada &   4182.68  & -2480.94  &    0.08  &   1.7876   &   2488.80   &   1.97  \\ 
NLS-ManPG   &   1024.54  & -2480.94  &    0.08  &   0.5133   &   850.58   &   2.31  \\ 
ManPQN      &   210.16  & -2480.64  &    0.05  &   0.1115   &   7.22   &   2.17  \\ 
\midrule
``fidap003'' & \\
\midrule
ManPG       &   4986.52  & -6764.09  &    0.89  &   2.1449   &   0.00   &   1.06  \\ 
ManPG-Ada &   2169.84  & -6764.09  &    0.89  &   1.1988   &   1249.16   &   1.19  \\ 
NLS-ManPG   &   478.88  & -6764.09  &    0.89  &   0.4152   &   449.62   &   1.93  \\ 
ManPQN      &   95.58  & -6764.08  &    0.87  &   0.1705   &   1.14   &   3.07  \\ 
\midrule
\end{tabular}
\end{table}

\subsection{Joint Diagonalization Problem with a Regularization Term}

The joint diagonalization problem with a regularization term on the Stiefel manifold \cite{jd2014} can be written in the following formulation:
\begin{equation}\label{eq5_4}
\min_{X\in\mathcal{M}}-\sum_{l=1}^N\|\mathrm{diag}(X^TA_lX)\|_F^2+\mu\|X\|_1,
\end{equation}
where $A_1,A_2,\dots,A_N\in\mathbb{R}^{n\times n}$ are $N$ real symmetric matrices.

Let $f(X):=-\sum_{l=1}^N\|\mathrm{diag}(X^TA_lX)\|_F^2$. 
We can deduce that the Euclidean gradient of $f$ is 
\[
\nabla f(X)=-4\sum_{l=1}^NA_lX\mathrm{diag}(X^TA_lX),
\]
and the Riemannian gradient of $f$ is 
\[
\mathrm{grad} f(X)=-4\sum_{l=1}^N(A_lX\mathrm{diag}(X^TA_lX)
-X\mathrm{sym}(X^TA_lX\mathrm{diag}(X^TA_lX)))
,
\]
where $\mathrm{sym}(Y):=(Y^T+Y)/2$.

For numerical experiments of \eqref{eq5_4}, we generate $N$ randomly chosen $n\times n$ diagonal matrices $\Lambda_1,\Lambda_2,\dots,\Lambda_N$ and a randomly chosen $n\times n$ orthogonal matrix $P$. Then, let $A_1,A_2,\dots,A_N$ be computed by $A_i=P^T\Lambda_iP$ for $i=1,2,\dots,N$.
We set $N=5$ in the following experiments.

The numerical results of each algorithm are displayed in Figures \ref{fig54_1}-\ref{fig54_3} and Tables \ref{tb54_1}-\ref{tb54_3}. 
Since all the algorithms need large number of line search steps, we report the average iteration numbers of line search in Tables \ref{tb54_1}-\ref{tb54_3}. 
Figures \ref{fig54_1}-\ref{fig54_3} show that ManPQN needs much less iterations and time than ManPG and ManPG-Ada.
In most cases, ManPQN outperforms NLS-ManPG in terms of CPU time, especially when $n$ and $r$ are large.

\begin{figure}[H]
	\centering  
	\subfigbottomskip=2pt 
	\subfigcapskip=-5pt
	\subfigure[CPU]{
		\includegraphics[width=0.5\linewidth]{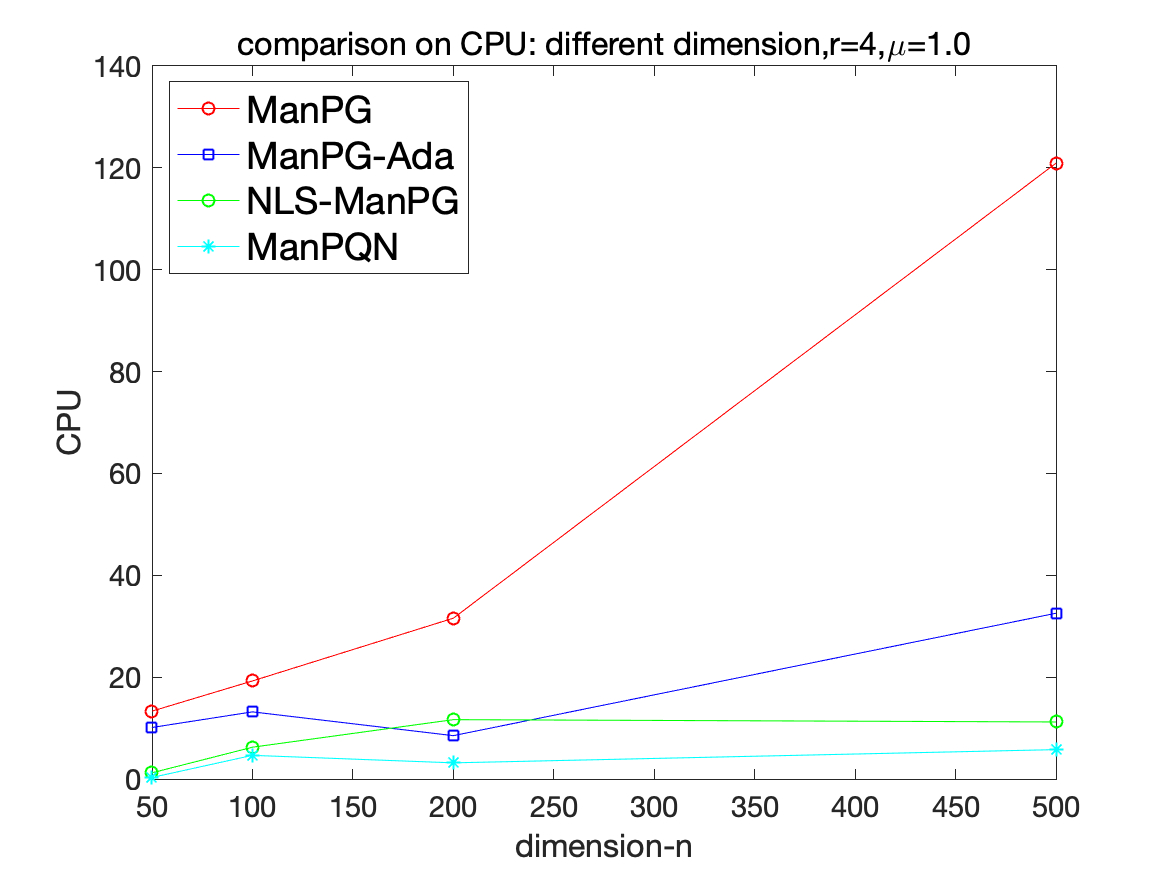}}
		\hspace{-5mm}
	\subfigure[Iter]{
		\includegraphics[width=0.5\linewidth]{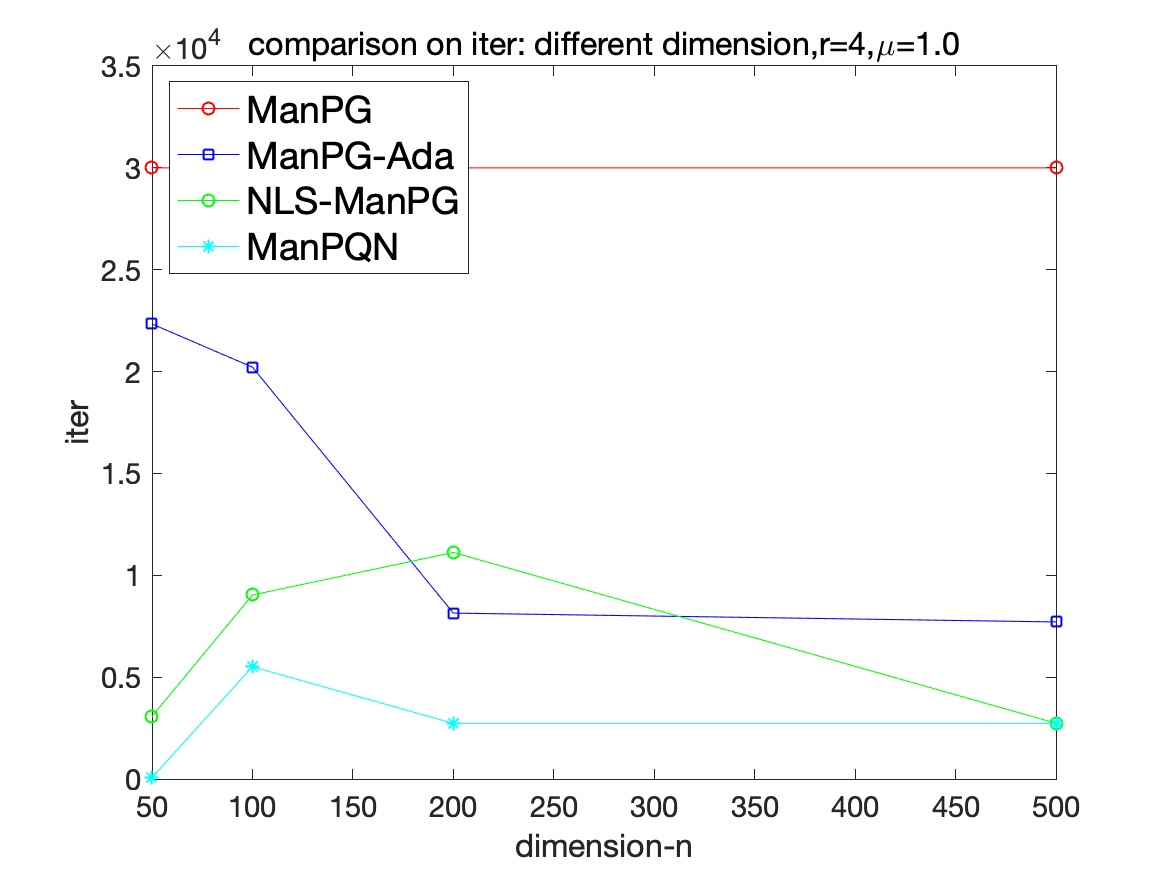}}
	\caption{Comparison on problem \eqref{eq5_4}, different $n=\{50,100,200,500\}$ with $r=4$ and $\mu=1.0$.}\label{fig54_1}
\end{figure}

\begin{figure}[H]
	\centering  
	\subfigbottomskip=2pt 
	\subfigcapskip=-5pt 
	\subfigure[CPU]{
		\includegraphics[width=0.5\linewidth]{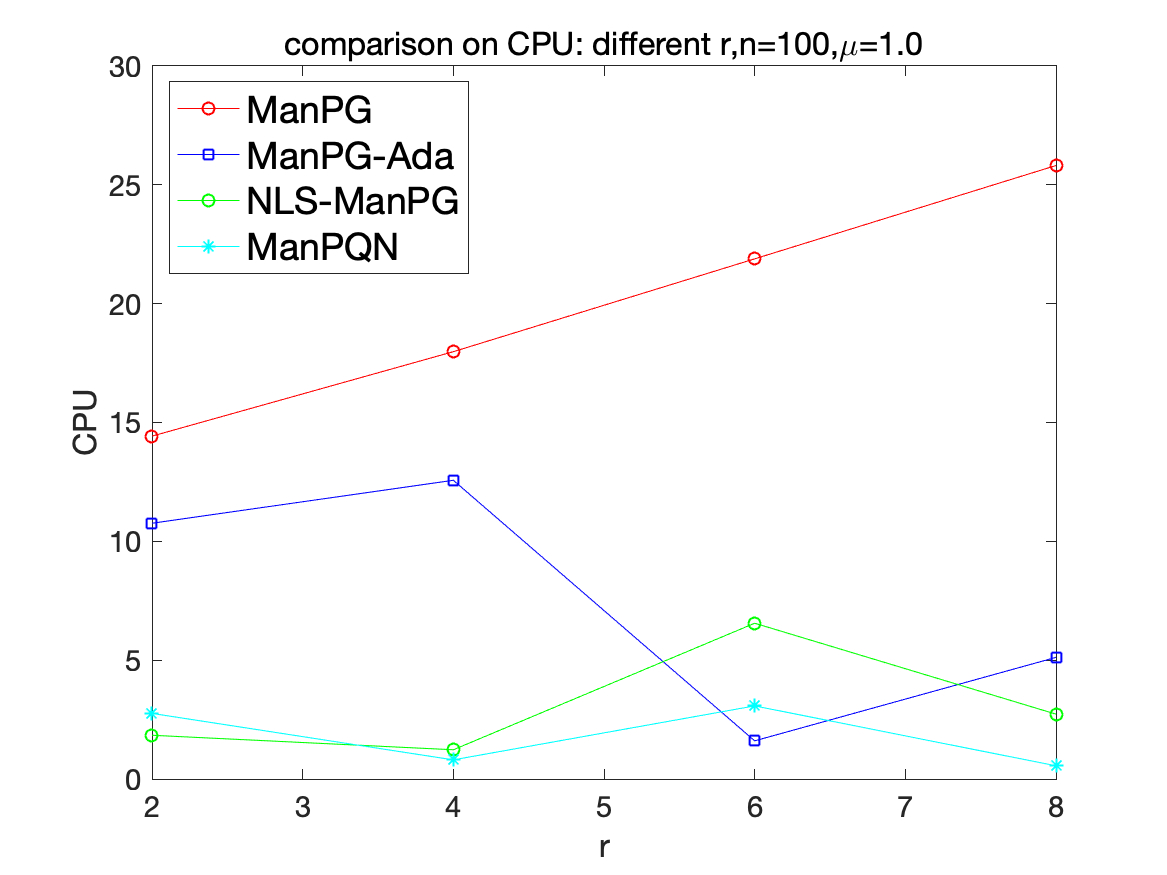}}
		\hspace{-5mm}
	\subfigure[Iter]{
		\includegraphics[width=0.5\linewidth]{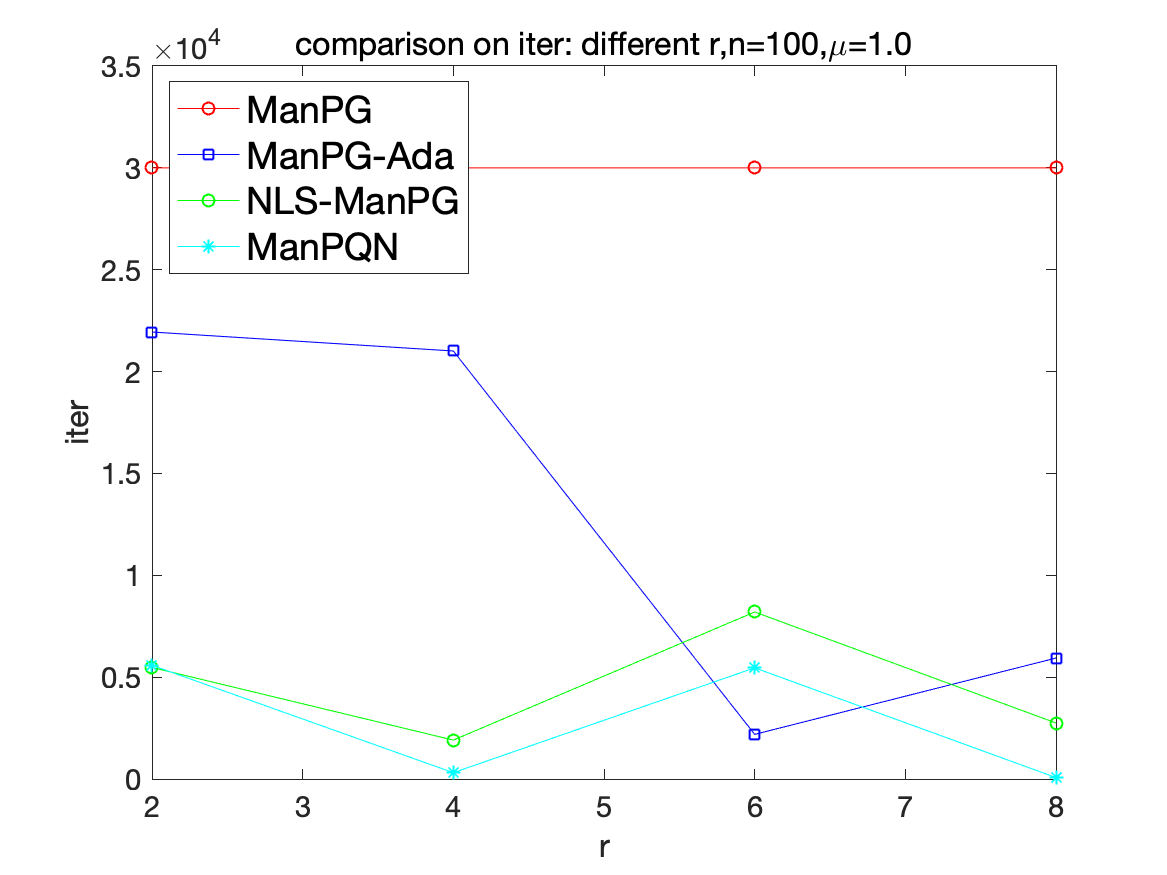}}
	\caption{Comparison on problem \eqref{eq5_4}, different $r=\{2,4,6,8\}$ with $n=100$ and $\mu=1.0$.}\label{fig54_2}
\end{figure}

\begin{figure}[H]
	\centering  
	\subfigbottomskip=2pt 
	\subfigcapskip=-5pt 
	\subfigure[CPU]{
		\includegraphics[width=0.5\linewidth]{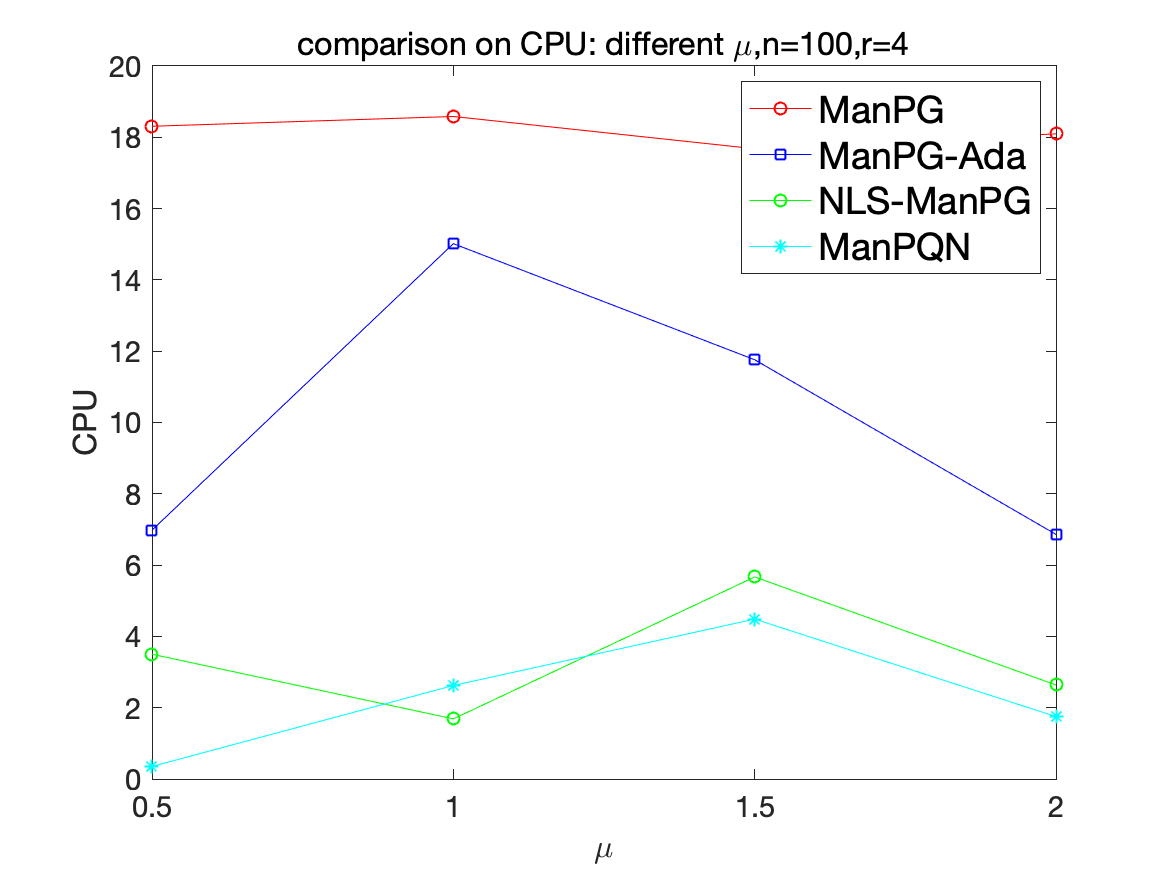}}
		\hspace{-5mm}
	\subfigure[Iter]{
		\includegraphics[width=0.5\linewidth]{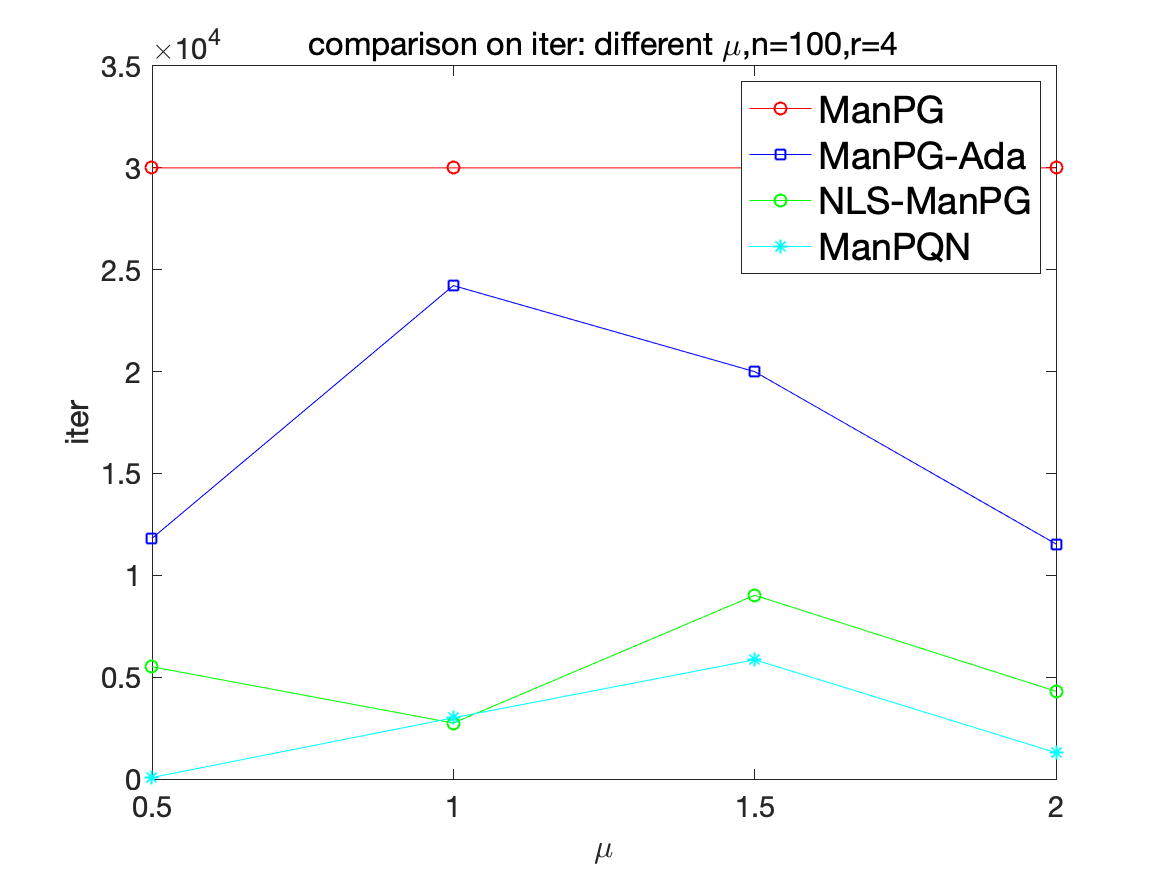}}
	\caption{Comparison on problem \eqref{eq5_4}, different $\mu=\{0.5,1.0,1.5,2.0\}$ with $n=100$ and $r=4$.}\label{fig54_3}
\end{figure}

\begin{table}[htbp]\centering
\caption{Comparison on problem \eqref{eq5_4}, different $n=\{50,100,200,500\}$ with $r=4$ and $\mu=1.0$.}\label{tb54_1}
\begin{tabular}{cccccc}
\midrule
$n=50$ & Iter & $F(X^*)$  & CPU time & line-search &  SSN iters  \\
\midrule
ManPG      &   30001.00  & -72.804  &   13.3257   &   13.99   &   1.00  \\ 
ManPG-Ada &   22335.45  & -72.784  &    10.1168   &   13.98   &   1.01  \\ 
NLS-ManPG  &   3073.09   & -72.821  &    1.2158   &   12.09   &   2.91  \\ 
ManPQN     &   68.82    & -72.825  &     0.3008   &   9.48   &   12.54  \\ 
\midrule
$n=100$ &   \\
\midrule
ManPG      &   30001.00  & -126.85  &     19.2554   &   13.99   &   1.01  \\ 
ManPG-Ada &   20215.64  & -126.85  &     13.1823   &   13.98   &   1.68  \\ 
NLS-ManPG  &   9032.55   & -126.87  &    6.2430   &   13.97   &   4.04  \\ 
ManPQN     &   5509.45    & -126.81  &    4.6449   &   12.73   &   14.62  \\ 
\midrule
$n=200$ &   \\
\midrule
ManPG      &   30001.00  & -148.11  &     31.5747   &   13.99   &   1.00  \\ 
ManPG-Ada &   8143.82  & -148.11  &    8.5102   &   13.97   &   2.44  \\ 
NLS-ManPG  &   11119.36   & -148.11  &     11.6534   &   13.9   &   3.22  \\ 
ManPQN     &   2730.18    & -148.17  &       3.1700   &   13.5   &   13.32  \\ 
\midrule
$n=500$ &   \\
\midrule
ManPG      &   30001.00  & -222.75  &   120.9160   &   13.99  &   1.00  \\ 
ManPG-Ada &   7709.36  & -222.77  &    32.5464   &   13.98   &   3.50  \\ 
NLS-ManPG  &   2736.82   & -222.79  &    11.1977   &   13.95   &   4.70  \\ 
ManPQN     &   2729.18    & -222.93  &    5.7707   &   5.54   &   22.69  \\ 
\midrule
\end{tabular}
\end{table}

\begin{table}[htbp]\centering
\caption{Comparison on problem \eqref{eq5_4}, different $r=\{2,4,6,8\}$ with $n=100$ and $\mu=1.0$.}\label{tb54_2}
\begin{tabular}{cccccc}
\midrule
$r=2$ & Iter & $F(X^*)$  & CPU time & line-search &  SSN iters  \\
\midrule
ManPG      &   30001.00  & -69.79  &   14.4263   &   13.99   &   0.92  \\ 
ManPG-Ada &   21941.64  & -69.79  &  10.7613   &   13.98   &   1.22  \\ 
NLS-ManPG  &   5474.55   & -69.80  &   1.8400   &   8.81   &   3.24  \\ 
ManPQN     &   5580.91    & -69.80  &   2.7605   &   8.73   &   16.92  \\
\midrule
$r=4$ &   \\
\midrule
ManPG      &   30001.00  & -119.37  &     17.9799   &   13.99   &   1.00  \\ 
ManPG-Ada &   21010.00  & -119.38  &    12.5684   &  13.98  &   1.80  \\ 
NLS-ManPG  &   1909.36   & -119.39  &    1.2302   &   12.56  &   4.35  \\ 
ManPQN     &   310.09    & -119.42  &    0.8054   &   12.44   &   18.71  \\ 
\midrule
$r=6$ &   \\
\midrule
ManPG      &   30001.00  & -122.92  &     21.8858   &   13.99   &   1.00  \\ 
ManPG-Ada &   2190.82  & -122.92 &     1.6054   &   13.98   &   1.47  \\ 
NLS-ManPG  &   8198.64   & -122.91  &    6.5546   &   13.95   &   4.10  \\ 
ManPQN     &   5456.36    & -123.00  &     3.0761   &   5.89   &   10.95  \\
\midrule
$r=8$ &   \\
\midrule
ManPG      &   30001.00  & -142.23  &     25.8129   &   13.99   &   1.00  \\ 
ManPG-Ada &   5973.36  & -142.29  &      5.1177   &   13.97   &   2.67  \\ 
NLS-ManPG  &   2747.73   & -142.30  &    2.7289   &   13.94   &   5.02  \\ 
ManPQN     &   70.00    & -142.54  &     0.5549   &   9.25   &   17.60  \\ 
\midrule
\end{tabular}
\end{table}

\begin{table}[htbp]\centering
\caption{Comparison on problem \eqref{eq5_4}, different $\mu=\{0.5,1.0,1.5,2.0\}$ with $n=100$ and $r=4$.}\label{tb54_3}
\begin{tabular}{cccccc}
\midrule
$\mu=0.5$ & Iter & $F(X^*)$  & CPU time & line-search &  SSN iters  \\
\midrule
ManPG      &   30001.00  & -134.85  &    18.3098   &   13.99   &   1.00  \\ 
ManPG-Ada &   11813.55  & -134.86  &     6.9803   &   13.96   &   1.84  \\ 
NLS-ManPG  &   5509.45   & -134.86  &    3.5001   &   13.92   &   3.74  \\ 
ManPQN     &   74.27    & -134.87  &    0.3515   &   8.62   &   13.36  \\ 
\midrule
$\mu=1.0$ &   \\
\midrule
ManPG      &   30001.00  & -129.57  &   18.5839   &   13.99   &   0.98  \\ 
ManPG-Ada &   24221.09  & -129.57  &    15.0166   &   13.99   &   1.24  \\ 
NLS-ManPG  &   2745.73   & -129.59  &   1.6866   &   13.90   &   3.84  \\ 
ManPQN     &   3003.18    & -129.58  &    2.6224   &   13.76   &   12.40  \\ 
\midrule
$\mu=1.5$ &   \\
\midrule
ManPG      &   30001.00  & -112.51  &    17.6801   &   13.99   &   1.00  \\ 
ManPG-Ada &   19994.13  & -112.51  &     11.7601   &   13.98   &   1.11  \\ 
NLS-ManPG  &   9018.80   & -112.51  &     5.6677   &   13.88   &   3.62  \\ 
ManPQN     &   5854.17    & -112.53  &     4.4855   &   13.5   &   12.66  \\ 
\midrule
$\mu=2.0$ &   \\
\midrule
ManPG      &   30001.00  & -76.63  &     18.0976   &   13.99   &   1.10  \\ 
ManPG-Ada &   11534.86  & -76.64  &     6.8606   &   13.98   &   1.35  \\ 
NLS-ManPG  &   4306.00   & -76.65  &      2.6373   &   13.91   &   4.01  \\ 
ManPQN     &   1292.71    & -76.71  &      1.7506   &   12.65   &   16.84  \\ 
\midrule
\end{tabular}
\end{table}

\section{Conclusion and Future Work}\label{sec6}

In this paper, we present a proximal quasi-Newton algorithm, named ManPQN,
for the composite optimization problem \eqref{eq1_prob} over the Stiefel manifold.
The ManPQN method finds the descent direction $V_k$ by solving a subproblem,
which is formed by replacing the term $\|V\|^2/2t$ in the subproblem of ManPG by $\frac{1}{2}\|V\|_{\mathcal{B}_k}^2$,
where $\mathcal{B}_k$ is a symmetric linear operator on ${\rm T}_{X_k}\mathcal{M}$.
We also use several techniques to accelerate the speed of the ManPQN algorithm.
The most important technique is that we use the linear operator $\mathbf{B}_k$,
defined by \eqref{hanwu},
to approximate the linear operator $\mathcal{B}_k$.
To guarantee the positive definiteness of $\mathbf{B}_k$, a damped LBFGS method is used to update $\mathbf{B}_k$.
Moreover, we use a nonmonotone line search technique to improve the performance of the ManPQN method.
Numerical results demonstrate that the ManPQN method is an effective method.
Under some mild conditions, we establish the global convergence of ManPQN.
If the Hessian operator of the objective function is positive definite at the local minimum,
the local linear convergence of ManPQN is also proved.

The main cost of the ManPQN lies in the step of solving the subproblem \eqref{eq3_subprob}.
Stimulated by the work in \cite{mashiqian2020}, we use the adaptive semismooth Newton (ASSN) method to get the solution of \eqref{eq3_subprob}.
But the total cost of the ASSN method is excessive for large $n$ and $r$.
How to reduce the computational cost of the ASSN method will be one of the topics of our future work.
We will also investigate other techniques to accelerate the ManPG method.
This is another topic of our future work.

\bmhead{Acknowledgments}

The work of Wei Hong Yang was supported by the National Natural Science Foundation of China grant 11971118.
The authors are grateful to the associate editor and the two anonymous referees
for their valuable comments and suggestions.
\bibliographystyle{plain}
\nocite{proxmanifold2005}

\bibliography{ManPQN}


\begin{thebibliography}{43}
\ifx \bisbn   \undefined \def \bisbn  #1{ISBN #1}\fi
\ifx \binits  \undefined \def \binits#1{#1}\fi
\ifx \bauthor  \undefined \def \bauthor#1{#1}\fi
\ifx \batitle  \undefined \def \batitle#1{#1}\fi
\ifx \bjtitle  \undefined \def \bjtitle#1{#1}\fi
\ifx \bvolume  \undefined \def \bvolume#1{\textbf{#1}}\fi
\ifx \byear  \undefined \def \byear#1{#1}\fi
\ifx \bissue  \undefined \def \bissue#1{#1}\fi
\ifx \bfpage  \undefined \def \bfpage#1{#1}\fi
\ifx \blpage  \undefined \def \blpage #1{#1}\fi
\ifx \burl  \undefined \def \burl#1{\textsf{#1}}\fi
\ifx \doiurl  \undefined \def \doiurl#1{\url{https://doi.org/#1}}\fi
\ifx \betal  \undefined \def \betal{\textit{et al.}}\fi
\ifx \binstitute  \undefined \def \binstitute#1{#1}\fi
\ifx \binstitutionaled  \undefined \def \binstitutionaled#1{#1}\fi
\ifx \bctitle  \undefined \def \bctitle#1{#1}\fi
\ifx \beditor  \undefined \def \beditor#1{#1}\fi
\ifx \bpublisher  \undefined \def \bpublisher#1{#1}\fi
\ifx \bbtitle  \undefined \def \bbtitle#1{#1}\fi
\ifx \bedition  \undefined \def \bedition#1{#1}\fi
\ifx \bseriesno  \undefined \def \bseriesno#1{#1}\fi
\ifx \blocation  \undefined \def \blocation#1{#1}\fi
\ifx \bsertitle  \undefined \def \bsertitle#1{#1}\fi
\ifx \bsnm \undefined \def \bsnm#1{#1}\fi
\ifx \bsuffix \undefined \def \bsuffix#1{#1}\fi
\ifx \bparticle \undefined \def \bparticle#1{#1}\fi
\ifx \barticle \undefined \def \barticle#1{#1}\fi
\bibcommenthead
\ifx \bconfdate \undefined \def \bconfdate #1{#1}\fi
\ifx \botherref \undefined \def \botherref #1{#1}\fi
\ifx \url \undefined \def \url#1{\textsf{#1}}\fi
\ifx \bchapter \undefined \def \bchapter#1{#1}\fi
\ifx \bbook \undefined \def \bbook#1{#1}\fi
\ifx \bcomment \undefined \def \bcomment#1{#1}\fi
\ifx \oauthor \undefined \def \oauthor#1{#1}\fi
\ifx \citeauthoryear \undefined \def \citeauthoryear#1{#1}\fi
\ifx \endbibitem  \undefined \def \endbibitem {}\fi
\ifx \bconflocation  \undefined \def \bconflocation#1{#1}\fi
\ifx \arxivurl  \undefined \def \arxivurl#1{\textsf{#1}}\fi
\csname PreBibitemsHook\endcsname

\bibitem[\protect\citeauthoryear{Absil and Hosseini}{2019}]{absil2017}
\begin{barticle}
\bauthor{\bsnm{Absil}, \binits{P.-A.}},
\bauthor{\bsnm{Hosseini}, \binits{S.}}:
\batitle{A collection of nonsmooth {Riemannian} optimization problems}.
\bjtitle{International Series of Numerical Mathematics}
\bvolume{170},
\bfpage{1}--\blpage{15}
(\byear{2019}).
\doiurl{10.1007/978-3-030-11370-4_1}
\end{barticle}
\endbibitem

\bibitem[\protect\citeauthoryear{Absil et~al.}{2008}]{absil2008}
\begin{bbook}
\bauthor{\bsnm{Absil}, \binits{P.-A.}},
\bauthor{\bsnm{Mahony}, \binits{R.}},
\bauthor{\bsnm{Sepulchre}, \binits{R.}}:
\bbtitle{Optimization Algorithms on Matrix Manifolds}.
\bpublisher{Princeton University Press},
\blocation{Princeton, NJ}
(\byear{2008})
\end{bbook}
\endbibitem

\bibitem[\protect\citeauthoryear{Beck and Teboulle}{2009}]{fista}
\begin{barticle}
\bauthor{\bsnm{Beck}, \binits{A.}},
\bauthor{\bsnm{Teboulle}, \binits{M.}}:
\batitle{A fast iterative shrinkage-threshold algorithm for linear inverse
  problems}.
\bjtitle{SIAM J. Imaging Sci.}
\bvolume{2},
\bfpage{183}--\blpage{202}
(\byear{2009}).
\doiurl{10.1137/080716542}
\end{barticle}
\endbibitem

\bibitem[\protect\citeauthoryear{Bento et~al.}{2016}]{bento2016}
\begin{barticle}
\bauthor{\bsnm{Bento}, \binits{G.}},
\bauthor{\bsnm{Cruz~Neto}, \binits{J.}},
\bauthor{\bsnm{Oliveira}, \binits{P.}}:
\batitle{A new approach to the proximal point method: Convergence on general
  {Riemannian} manifolds}.
\bjtitle{Journal of Optimization Theory and Applications}
\bvolume{168},
\bfpage{743}--\blpage{755}
(\byear{2016}).
\doiurl{10.1007/s10957-015-0861-2}
\end{barticle}
\endbibitem

\bibitem[\protect\citeauthoryear{Boumal et~al.}{2016}]{boumal2016}
\begin{barticle}
\bauthor{\bsnm{Boumal}, \binits{N.}},
\bauthor{\bsnm{Absil}, \binits{P.-A.}},
\bauthor{\bsnm{Cartis}, \binits{C.}}:
\batitle{Global rates of convergence for nonconvex optimization on manifolds}.
\bjtitle{IMA Journal of Numerical Analysis}
\bvolume{39}(\bissue{1}),
\bfpage{1}--\blpage{33}
(\byear{2016}).
\doiurl{10.1093/imanum/drx080}
\end{barticle}
\endbibitem

\bibitem[\protect\citeauthoryear{Chen et~al.}{2020}]{mashiqian2020}
\begin{barticle}
\bauthor{\bsnm{Chen}, \binits{S.}},
\bauthor{\bsnm{Ma}, \binits{S.}},
\bauthor{\bsnm{So}, \binits{A.M.-C.}},
\bauthor{\bsnm{Zhang}, \binits{T.}}:
\batitle{Proximal gradient method for nonsmooth optimization over the {Stiefel}
  manifold}.
\bjtitle{SIAM Journal on Optimization}
\bvolume{30}(\bissue{1}),
\bfpage{210}--\blpage{239}
(\byear{2020}).
\doiurl{10.1137/18M122457X}
\end{barticle}
\endbibitem

\bibitem[\protect\citeauthoryear{Dai}{2002}]{dyh2002}
\begin{barticle}
\bauthor{\bsnm{Dai}, \binits{Y.-H.}}:
\batitle{A nonmonotone conjugate gradient algorithm for unconstrained
  optimization}.
\bjtitle{Journal of Systems Science and Complexity}
\bvolume{15}(\bissue{1}),
\bfpage{139}--\blpage{145}
(\byear{2002})
\end{barticle}
\endbibitem

\bibitem[\protect\citeauthoryear{Davis and Hu}{2011}]{ufmat_web}
\begin{barticle}
\bauthor{\bsnm{Davis}, \binits{T.A.}},
\bauthor{\bsnm{Hu}, \binits{Y.}}:
\batitle{The university of florida sparse matrix collection}.
\bjtitle{ACM Transactions on Mathematical Software (TOMS)}
\bvolume{38}(\bissue{1}),
\bfpage{1}--\blpage{25}
(\byear{2011})
\end{barticle}
\endbibitem

\bibitem[\protect\citeauthoryear{Ferreira and Oliveira}{1998}]{subgradient1998}
\begin{barticle}
\bauthor{\bsnm{Ferreira}, \binits{O.}},
\bauthor{\bsnm{Oliveira}, \binits{P.}}:
\batitle{Subgradient algorithm on {Riemannian} manifolds}.
\bjtitle{Journal of Optimization Theory and Applications}
\bvolume{97},
\bfpage{93}--\blpage{104}
(\byear{1998}).
\doiurl{10.1023/A:1022675100677}
\end{barticle}
\endbibitem

\bibitem[\protect\citeauthoryear{Ferreira and
  Oliveira}{2002}]{ppaonmanifold2002}
\begin{barticle}
\bauthor{\bsnm{Ferreira}, \binits{O.}},
\bauthor{\bsnm{Oliveira}, \binits{P.}}:
\batitle{Proximal point algorithm on {Riemannian} manifolds}.
\bjtitle{Optimization}
\bvolume{51},
\bfpage{257}--\blpage{270}
(\byear{2002}).
\doiurl{10.1080/02331930290019413}
\end{barticle}
\endbibitem

\bibitem[\protect\citeauthoryear{Fukushima and Qi}{1996}]{pn1996fukushima}
\begin{barticle}
\bauthor{\bsnm{Fukushima}, \binits{M.}},
\bauthor{\bsnm{Qi}, \binits{L.}}:
\batitle{A globally and superlinearly convergent algorithm for nonsmooth convex
  minimization}.
\bjtitle{SIAM Journal on Optimization}
\bvolume{6}(\bissue{4}),
\bfpage{1106}--\blpage{1120}
(\byear{1996}).
\doiurl{10.1137/S1052623494278839}
\end{barticle}
\endbibitem

\bibitem[\protect\citeauthoryear{Gao et~al.}{2019}]{yyx2019}
\begin{barticle}
\bauthor{\bsnm{Gao}, \binits{B.}},
\bauthor{\bsnm{Liu}, \binits{X.}},
\bauthor{\bsnm{Yuan}, \binits{Y.-X.}}:
\batitle{Parallelizable algorithms for optimization problems with orthogonality
  constraints}.
\bjtitle{SIAM Journal on Scientific Computing}
\bvolume{41}(\bissue{3}),
\bfpage{1949}--\blpage{1983}
(\byear{2019})
{\href{https://arxiv.org/abs/https://doi.org/10.1137/18M1221679}{{https://doi.org/10.1137/18M1221679}}}.
\doiurl{10.1137/18M1221679}
\end{barticle}
\endbibitem

\bibitem[\protect\citeauthoryear{Grippo et~al.}{1986}]{gll1986}
\begin{barticle}
\bauthor{\bsnm{Grippo}, \binits{L.}},
\bauthor{\bsnm{Lampariello}, \binits{F.}},
\bauthor{\bsnm{Lucidi}, \binits{S.}}:
\batitle{A nonmonotone line search technique for {Newton’s} method}.
\bjtitle{SIAM Journal on Numerical Analysis}
\bvolume{23}(\bissue{4}),
\bfpage{707}--\blpage{716}
(\byear{1986}).
\doiurl{10.1137/0723046}
\end{barticle}
\endbibitem

\bibitem[\protect\citeauthoryear{Grohs and Hosseini}{2016}]{subgradient2015}
\begin{barticle}
\bauthor{\bsnm{Grohs}, \binits{P.}},
\bauthor{\bsnm{Hosseini}, \binits{S.}}:
\batitle{$\varepsilon$-subgradient algorithms for locally {Lipschitz} functions
  on {Riemannian} manifolds}.
\bjtitle{Advances in Computational Mathematics}
\bvolume{42}(\bissue{2}),
\bfpage{333}--\blpage{360}
(\byear{2016}).
\doiurl{10.1007/s10444-015-9426-z}
\end{barticle}
\endbibitem

\bibitem[\protect\citeauthoryear{Hosseini and Grohs}{2016}]{subgradient_tr2015}
\begin{barticle}
\bauthor{\bsnm{Hosseini}, \binits{S.}},
\bauthor{\bsnm{Grohs}, \binits{P.}}:
\batitle{Nonsmooth trust region algorithms for locally {Lipschitz} functions on
  {Riemannian} manifolds}.
\bjtitle{IMA Journal of Numerical Analysis}
\bvolume{36}(\bissue{3}),
\bfpage{1167}--\blpage{1192}
(\byear{2016}).
\doiurl{10.1093/imanum/drv043}
\end{barticle}
\endbibitem

\bibitem[\protect\citeauthoryear{Hu et~al.}{2018}]{wzw2018}
\begin{barticle}
\bauthor{\bsnm{Hu}, \binits{J.}},
\bauthor{\bsnm{Milzarek}, \binits{A.}},
\bauthor{\bsnm{Wen}, \binits{Z.}},
\bauthor{\bsnm{Yuan}, \binits{Y.-X.}}:
\batitle{Adaptive quadratically regularized {Newton} method for {Riemannian}
  optimization}.
\bjtitle{SIAM Journal on Matrix Analysis and Applications}
\bvolume{39}(\bissue{3}),
\bfpage{1181}--\blpage{1207}
(\byear{2018}).
\doiurl{10.1137/17M1142478}
\end{barticle}
\endbibitem

\bibitem[\protect\citeauthoryear{Huang and Wei}{2022a}]{huang2022}
\begin{barticle}
\bauthor{\bsnm{Huang}, \binits{W.}},
\bauthor{\bsnm{Wei}, \binits{K.}}:
\batitle{An extension of fast iterative shrinkage-thresholding to {Riemannian}
  optimization for sparse principal component analysis}.
\bjtitle{Numerical Linear Algebra with Applications}
\bvolume{29},
\bfpage{2409}
(\byear{2022}a).
\doiurl{10.1002/nla.2409}
\end{barticle}
\endbibitem

\bibitem[\protect\citeauthoryear{Huang and Wei}{2022b}]{huang2021riemannian}
\begin{barticle}
\bauthor{\bsnm{Huang}, \binits{W.}},
\bauthor{\bsnm{Wei}, \binits{K.}}:
\batitle{{Riemannian} proximal gradient methods}.
\bjtitle{Mathematical Programming}
\bvolume{194},
\bfpage{371}--\blpage{413}
(\byear{2022}b).
\doiurl{10.1007/s10107-021-01632-3}
\end{barticle}
\endbibitem

\bibitem[\protect\citeauthoryear{Huang and Liu}{2015}]{huangyakui2015}
\begin{barticle}
\bauthor{\bsnm{Huang}, \binits{Y.}},
\bauthor{\bsnm{Liu}, \binits{H.}}:
\batitle{On the rate of convergence of projected {Barzilai–Borwein} methods}.
\bjtitle{Optimization Methods and Software}
\bvolume{30}(\bissue{4}),
\bfpage{880}--\blpage{892}
(\byear{2015}).
\doiurl{10.1080/10556788.2015.1004064}
\end{barticle}
\endbibitem

\bibitem[\protect\citeauthoryear{Lee et~al.}{2014}]{pn2014}
\begin{barticle}
\bauthor{\bsnm{Lee}, \binits{J.D.}},
\bauthor{\bsnm{Sun}, \binits{Y.}},
\bauthor{\bsnm{Saunders}, \binits{M.A.}}:
\batitle{Proximal {Newton}-type methods for minimizing composite functions}.
\bjtitle{SIAM Journal on Optimization}
\bvolume{24}(\bissue{3}),
\bfpage{1420}--\blpage{1443}
(\byear{2014}).
\doiurl{10.1137/130921428}
\end{barticle}
\endbibitem

\bibitem[\protect\citeauthoryear{Li et~al.}{2022}]{diagonal2022}
\begin{barticle}
\bauthor{\bsnm{Li}, \binits{D.}},
\bauthor{\bsnm{Wang}, \binits{X.}},
\bauthor{\bsnm{Huang}, \binits{J.}}:
\batitle{Diagonal {BFGS} updates and applications to the limited memory {BFGS}
  method}.
\bjtitle{Computational Optimization and Applications}
\bvolume{81},
\bfpage{829}--\blpage{856}
(\byear{2022}).
\doiurl{10.1007/s10589-022-00353-3}
\end{barticle}
\endbibitem

\bibitem[\protect\citeauthoryear{Liu et~al.}{2019}]{2order_boudness2016}
\begin{barticle}
\bauthor{\bsnm{Liu}, \binits{H.}},
\bauthor{\bsnm{Wu}, \binits{W.}},
\bauthor{\bsnm{So}, \binits{A.M.-C.}}:
\batitle{Quadratic optimization with orthogonality constraints: Explicit
  {{\L}ojasiewicz} exponent and linear convergence of retraction-based
  line-search and stochastic variance-reduced gradient methods}.
\bjtitle{Mathematical Programming}
\bvolume{178},
\bfpage{215}--\blpage{262}
(\byear{2019}).
\doiurl{10.1007/s10107-018-1285-1}
\end{barticle}
\endbibitem

\bibitem[\protect\citeauthoryear{Mordukhovich et~al.}{2022}]{pn2022}
\begin{barticle}
\bauthor{\bsnm{Mordukhovich}, \binits{B.}},
\bauthor{\bsnm{Yuan}, \binits{X.}},
\bauthor{\bsnm{Zeng}, \binits{S.}},
\bauthor{\bsnm{Zhang}, \binits{J.}}:
\batitle{A globally convergent proximal {Newton}-type method in nonsmooth
  convex optimization}.
\bjtitle{Mathematical Programming}
\bvolume{198},
\bfpage{899}--\blpage{936}
(\byear{2022}).
\doiurl{10.1007/s10107-022-01797-5}
\end{barticle}
\endbibitem

\bibitem[\protect\citeauthoryear{Moré and Sorensen}{1983}]{more1983}
\begin{barticle}
\bauthor{\bsnm{Moré}, \binits{J.J.}},
\bauthor{\bsnm{Sorensen}, \binits{D.C.}}:
\batitle{Computing a trust region step}.
\bjtitle{SIAM Journal on Scientific and Statistical Computing}
\bvolume{4}(\bissue{3}),
\bfpage{553}--\blpage{572}
(\byear{1983}).
\doiurl{10.1137/0904038}
\end{barticle}
\endbibitem

\bibitem[\protect\citeauthoryear{Nakayama et~al.}{2021}]{pn2021}
\begin{barticle}
\bauthor{\bsnm{Nakayama}, \binits{S.}},
\bauthor{\bsnm{Narushima}, \binits{Y.}},
\bauthor{\bsnm{Yabe}, \binits{H.}}:
\batitle{Inexact proximal memoryless quasi-{Newton} methods based on the
  {Broyden} family for minimizing composite functions}.
\bjtitle{Computational Optimization and Applications}
\bvolume{79},
\bfpage{127}--\blpage{154}
(\byear{2021}).
\doiurl{10.1007/s10589-021-00264-9}
\end{barticle}
\endbibitem

\bibitem[\protect\citeauthoryear{Nesterov}{2013}]{nes2013}
\begin{barticle}
\bauthor{\bsnm{Nesterov}, \binits{Y.}}:
\batitle{Gradient methods for minimizing composite functions}.
\bjtitle{Math. Program. Ser. B}
\bvolume{140},
\bfpage{125}--\blpage{161}
(\byear{2013}).
\doiurl{10.1007/s10107-012-0629-5}
\end{barticle}
\endbibitem

\bibitem[\protect\citeauthoryear{Nesterov}{2018}]{nesterov}
\begin{bbook}
\bauthor{\bsnm{Nesterov}, \binits{Y.}}:
\bbtitle{Lectures on Convex Optimization},
\bedition{2}nd edn.
\bpublisher{Springer},
\blocation{Cham}
(\byear{2018})
\end{bbook}
\endbibitem

\bibitem[\protect\citeauthoryear{Nocedal}{1980}]{lbfgs1980}
\begin{barticle}
\bauthor{\bsnm{Nocedal}, \binits{J.}}:
\batitle{Updating quasi-{Newton} matrices with limited storage}.
\bjtitle{Mathematics of Computation}
\bvolume{35},
\bfpage{773}--\blpage{782}
(\byear{1980}).
\doiurl{10.2307/2006193}
\end{barticle}
\endbibitem

\bibitem[\protect\citeauthoryear{Oviedo}{2021}]{ppa2021}
\begin{botherref}
\oauthor{\bsnm{Oviedo}, \binits{H.}}:
Proximal point algorithm on the stiefel manifold.
Preprint in Optimization–Online:
  http://www.optimization-online.org/DB\_FILE/2021/05/ 8401.pdf
(2021)
\end{botherref}
\endbibitem

\bibitem[\protect\citeauthoryear{Ozoliņ{\v s} et~al.}{2013}]{cm2013}
\begin{barticle}
\bauthor{\bsnm{Ozoliņ{\v s}}, \binits{V.}},
\bauthor{\bsnm{Lai}, \binits{R.}},
\bauthor{\bsnm{Caflisch}, \binits{R.}},
\bauthor{\bsnm{Osher}, \binits{S.}}:
\batitle{Compressed modes for variational problems in mathematics and physics}.
\bjtitle{Proceedings of the National Academy of Sciences of the United States
  of America}
\bvolume{110}(\bissue{46}),
\bfpage{18368}--\blpage{18373}
(\byear{2013}).
\doiurl{10.1073/pnas.1318679110}
\end{barticle}
\endbibitem

\bibitem[\protect\citeauthoryear{Park et~al.}{2020}]{diag2020}
\begin{bchapter}
\bauthor{\bsnm{Park}, \binits{Y.}},
\bauthor{\bsnm{Dhar}, \binits{S.}},
\bauthor{\bsnm{Boyd}, \binits{S.}},
\bauthor{\bsnm{Shah}, \binits{M.}}:
\bctitle{Variable metric proximal gradient method with diagonal
  {Barzilai-Borwein} stepsize}.
In: \bbtitle{ICASSP 2020 - 2020 IEEE International Conference on Acoustics,
  Speech and Signal Processing (ICASSP)},
pp. \bfpage{3597}--\blpage{3601}
(\byear{2020}).
\doiurl{10.1109/ICASSP40776.2020.9054193}
\end{bchapter}
\endbibitem

\bibitem[\protect\citeauthoryear{Powell}{1978}]{damp1978}
\begin{barticle}
\bauthor{\bsnm{Powell}, \binits{M.J.}}:
\batitle{Algorithms for nonlinear constraints that use {Lagrangian} functions}.
\bjtitle{Math. Program.}
\bvolume{14}(\bissue{1}),
\bfpage{224}--\blpage{248}
(\byear{1978}).
\doiurl{10.1007/BF01588967}
\end{barticle}
\endbibitem

\bibitem[\protect\citeauthoryear{Ring and Wirth}{2012}]{ring2012}
\begin{barticle}
\bauthor{\bsnm{Ring}, \binits{W.}},
\bauthor{\bsnm{Wirth}, \binits{B.}}:
\batitle{Optimization methods on {Riemannian} manifolds and their application
  to shape space}.
\bjtitle{SIAM Journal on Optimization}
\bvolume{22}(\bissue{2}),
\bfpage{596}--\blpage{627}
(\byear{2012}).
\doiurl{10.1137/11082885X}
\end{barticle}
\endbibitem

\bibitem[\protect\citeauthoryear{Sato}{2017}]{jd2014}
\begin{barticle}
\bauthor{\bsnm{Sato}, \binits{H.}}:
\batitle{Riemannian {Newton-type} methods for joint diagonalization on the
  stiefel manifold with application to independent component analysis}.
\bjtitle{Optimization}
\bvolume{66}(\bissue{12}),
\bfpage{2211}--\blpage{2231}
(\byear{2017})
{\href{https://arxiv.org/abs/arXiv:1403.8064}{{arXiv:1403.8064}}}.
\doiurl{10.1080/02331934.2017.1359592}
\end{barticle}
\endbibitem

\bibitem[\protect\citeauthoryear{Sherman and Morrison}{1950}]{smw1950}
\begin{barticle}
\bauthor{\bsnm{Sherman}, \binits{J.}},
\bauthor{\bsnm{Morrison}, \binits{W.J.}}:
\batitle{Adjustment of an inverse matrix corresponding to a change in one
  element of a given matrix}.
\bjtitle{Annals of Mathematical Statistics}
\bvolume{21},
\bfpage{124}--\blpage{127}
(\byear{1950})
\end{barticle}
\endbibitem

\bibitem[\protect\citeauthoryear{Shor}{1967}]{shor1967}
\begin{barticle}
\bauthor{\bsnm{Shor}, \binits{N.Z.}}:
\batitle{Application of generalized gradient descent in block programming}.
\bjtitle{Cybernetics}
\bvolume{3}(\bissue{3}),
\bfpage{43}--\blpage{45}
(\byear{1967}).
\doiurl{10.1007/BF01120005}
\end{barticle}
\endbibitem

\bibitem[\protect\citeauthoryear{Wang et~al.}{2017}]{damp2017}
\begin{barticle}
\bauthor{\bsnm{Wang}, \binits{X.}},
\bauthor{\bsnm{Ma}, \binits{S.}},
\bauthor{\bsnm{Goldfarb}, \binits{D.}},
\bauthor{\bsnm{Liu}, \binits{W.}}:
\batitle{Stochastic quasi-{Newton} methods for nonconvex stochastic
  optimization}.
\bjtitle{SIAM Journal on Optimization}
\bvolume{27}(\bissue{2}),
\bfpage{927}--\blpage{956}
(\byear{2017}).
\doiurl{10.1137/15M1053141}
\end{barticle}
\endbibitem

\bibitem[\protect\citeauthoryear{Wen and Yin}{2013}]{wzw2010}
\begin{barticle}
\bauthor{\bsnm{Wen}, \binits{Z.}},
\bauthor{\bsnm{Yin}, \binits{W.}}:
\batitle{A feasible method for optimization with orthogonality constraints}.
\bjtitle{Mathematical Programming}
\bvolume{142},
\bfpage{397}--\blpage{434}
(\byear{2013}).
\doiurl{10.1007/s10107-012-0584-1}
\end{barticle}
\endbibitem

\bibitem[\protect\citeauthoryear{Wright et~al.}{2008}]{wright2008}
\begin{barticle}
\bauthor{\bsnm{Wright}, \binits{S.}},
\bauthor{\bsnm{Nowak}, \binits{R.}},
\bauthor{\bsnm{Figueiredo}, \binits{M.}}:
\batitle{Sparse reconstruction by separable approximation}.
\bjtitle{IEEE Transactions on Signal Processing}
\bvolume{57},
\bfpage{3373}--\blpage{3376}
(\byear{2008}).
\doiurl{10.1109/TSP.2009.2016892}
\end{barticle}
\endbibitem

\bibitem[\protect\citeauthoryear{Xiao et~al.}{2018}]{assn2018}
\begin{barticle}
\bauthor{\bsnm{Xiao}, \binits{X.}},
\bauthor{\bsnm{Li}, \binits{Y.}},
\bauthor{\bsnm{Wen}, \binits{Z.}},
\bauthor{\bsnm{Zhang}, \binits{L.W.}}:
\batitle{A regularized semi-smooth {Newton} method with projection steps for
  composite convex programs}.
\bjtitle{Journal of Scientific Computing}
\bvolume{76},
\bfpage{364}--\blpage{389}
(\byear{2018}).
\doiurl{10.1007/s10915-017-0624-3}
\end{barticle}
\endbibitem

\bibitem[\protect\citeauthoryear{Yang et~al.}{2013}]{ywh2013}
\begin{barticle}
\bauthor{\bsnm{Yang}, \binits{W.H.}},
\bauthor{\bsnm{Zhang}, \binits{L.}},
\bauthor{\bsnm{Song}, \binits{R.}}:
\batitle{Optimality conditions for the nonlinear programming problems on
  {Riemannian} manifolds}.
\bjtitle{Pacific Journal of Optimization}
\bvolume{10},
\bfpage{415}--\blpage{434}
(\byear{2013})
\end{barticle}
\endbibitem

\bibitem[\protect\citeauthoryear{Zhang et~al.}{2021}]{Zhang2021ARS}
\begin{botherref}
\oauthor{\bsnm{Zhang}, \binits{C.}},
\oauthor{\bsnm{Chen}, \binits{X.}},
\oauthor{\bsnm{Ma}, \binits{S.}}:
A riemannian smoothing steepest descent method for non-lipschitz optimization
  on submanifolds.
ArXiv
\textbf{abs/2104.04199}
(2021)
\end{botherref}
\endbibitem

\end{thebibliography}


%

\end{document}